\documentclass[11pt,dvips]{amsart}
\usepackage{amsfonts,amssymb,amsthm}
\usepackage[mathscr]{eucal}
\usepackage{amsfonts}
\usepackage{amsxtra}
\usepackage{amsmath}
\usepackage{amsthm}
\usepackage{latexsym}
\usepackage{verbatim}
\usepackage{csquotes}
\usepackage{epsfig}

\usepackage{psfrag}

\newcommand{\g}{\mathfrak{g}}
\newcommand{\ft}{\mathfrak{t}}

\newcommand{\td}{\mathfrak{t}^{*}}
\newcommand{\C}{\mathbb{C}}
\newcommand{\CP}{\mathbb{CP}}
\newcommand{\R}{\mathbb{R}}
\newcommand{\Q}{\mathbb{Q}}

\newcommand{\Z}{\mathbb{Z}}

\newcommand{\gt}{\widetilde{\gamma}}
\newcommand{\Jt}{\widetilde{J}}

\newcommand{\Oo}{\mathcal{O}}
\newcommand{\Hh}{\mathcal{H}}

\newcommand{\Wh}{\widehat{W}}
\newcommand{\wh}{\widehat{w}}
\newcommand{\Gh}{\widehat{G}}
\newcommand{\Kh}{\widehat{K}}

\newcommand{\Rh}{\widehat{R}}
\newcommand{\Bh}{\widehat{B}}

\newcommand{\ot}{\widetilde{\omega}}

\newcommand{\pb}{\overline{\psi}}
\newcommand{\pt}{\widetilde{\psi}}

\newcommand{\ah}{\widehat{\alpha}}

\newcommand{\Mt}{\widetilde{M}}

\newcommand{\Mh}{\widehat{M}}
\newcommand{\Lt}{\widetilde{\Lambda}}
\newcommand{\lt}{\widetilde{\lambda}}
\newcommand{\lh}{\widehat{\lambda}}

\newcommand{\Lh}{\widehat{\Lambda}}

\newcommand{\vb}{\overline{\varphi}}

\newcommand{\vt}{\widetilde{\varphi}}

\newcommand{\gkme}{E_{\operatorname{GKM}}}
\newcommand{\Id}{{\operatorname{Id}}}

\newcommand{\Sym}{{\operatorname{Sym}}}
\newcommand{\Ad}{{\operatorname{Ad}}}
\newcommand{\id}{{\operatorname{id}}}
\newcommand{\gkmet}{\widetilde{E}_{\operatorname{GKM}}}

\newcommand{\Vt}{\widetilde{V}}
\newcommand{\bh}{\widehat{\beta}}
\newcommand{\w}{\mathsf{w}}

\def\W+{W_{q}^{+}}
\def\W-{W_{q}^{-}}
\def\r{\rightarrow}

\newtheorem{theorem}{Theorem}[section]
\newtheorem{lemma}[theorem]{Lemma}
\newtheorem{proposition}[theorem]{Proposition}
\newtheorem{corollary}[theorem]{Corollary}
\newtheorem{remark}[theorem]{Remark}
\newtheorem{definition}[theorem]{Definition}
\newtheoremstyle{example}
{9pt}{9pt}{}{}{\bfseries}{}{.5em}{}
\theoremstyle{example}
\newtheorem{example}[theorem]{Example}

\newcommand{\labell}[1] {\label{#1}}

\begin{document}

\title[New Techniques for obtaining Schubert-type formulas]{New techniques for obtaining Schubert-type formulas for
Hamiltonian manifolds}

\author{Silvia Sabatini}\address{silvia.sabatini@epfl.ch\\ \'{E}cole Polytechnique F\'{e}d\'{e}rale
 de Lausanne\\ CH-1015 Lausanne}
\author{Susan Tolman}\address{stolman@math.uiuc.edu\\ University of Illinois Urbana-Champaign\\ 1409 W. Green Street\\ Urbana, IL 61801}
\thanks{The second author was partially supported by NSF-DMS Grant \#0707122.}

\begin{abstract}

In \cite{GT}, Goldin and the second author extend some 
ideas from Schubert calculus to the more general setting
of Hamiltonian torus actions on compact symplectic manifolds
with isolated fixed points. (See also \cite{Kn99} and \cite{Kn08}.) 
The main goal of this paper is to build on this work by finding more effective
formulas.

More explicitly, given a generic component  of
the moment map,  they define a {\em canonical class} $\alpha_p$ in
the equivariant cohomology of the manifold $M$ for each fixed point
$p \in M$.  When they exist, canonical classes form a natural basis
of the equivariant cohomology of $M$.
In particular, when $M$ is a
flag variety, these classes are the equivariant Schubert classes.
It is a long standing
problem in combinatorics to find positive integral formulas for the equivariant structure
constants associated to this basis.
Since 
computing the restriction of the canonical classes
to the fixed points 
determines these structure constants,
it is important to find effective formulas for these
restrictions.

In this paper, we introduce new techniques for calculating
the restrictions of a canonical class $\alpha_p$ to a fixed point $q$.
Our formulas are nearly always simpler, in the sense that
they count the contributions over fewer paths.
Moreover, our formula is 
manifestly positive and
integral 
in certain important special cases.

\end{abstract}

\maketitle

\tableofcontents

\section*{Introduction}

In \cite{GT},  
Goldin and the second author
extend some 
ideas from Schubert calculus to the more general setting
of Hamiltonian torus actions on compact symplectic manifolds
with isolated fixed points. 
(Knutson found closely related formulas for the 
 Duistermaat-Heckman measure 
in the algebraic case in \cite{Kn99} and \cite{Kn08}.) 
 Given a generic component  of
the moment map,  they define a canonical class $\alpha_p$ in
the equivariant cohomology of the manifold $M$ for each fixed point
$p \in M$ 
(see Definition~\ref{definitioncc} below).
 When they exist, 
these
canonical classes form a natural basis
of the equivariant cohomology of $M$. 
In particular, when $M$ is a
flag variety, these classes are the equivariant Schubert classes
(see 
\cite{Ku} and Proposition \ref{canonical=Schubert}). It
is a long standing
problem in combinatorics to find positive integral formulas for the equivariant structure
constants associated to this basis.
Since 
computing the restriction of the canonical classes
to the fixed points 
determines these structure constants and hence the (equivariant) cohomology ring
of $M$,
it is important to find effective formulas for these
restrictions.
Building on ideas of V. Guillemin and C. Zara \cite{GZ}, Goldin and
Tolman show that the restriction of a canonical class $\alpha_p$ to
a fixed point $q$ can be calculated by a rational function which
depends only on 
the following information:
the value of the moment map at fixed points, and the restriction
of other canonical classes to points of index exactly two higher.
Moreover, 
the restriction formula in \cite{GT}
is {\em manifestly positive}
whenever the restrictions themselves are all positive,
including when $M$ is a coadjoint orbit.

However, 
the results in \cite{GT}
differ from Schubert calculus in
several important ways. 
For example, 
the individual summands in that 
formula are almost never integral;
essentially, this only holds when $M$ is $\CP^n$.
In contrast, 
in the combinatorics literature, a manifestly positive integral
formula for the restriction of equivariant Schubert classes on $M=G/B$
is already known (see Appendix D.3 in \cite{AJS} and \cite{B}). 
The main goal of this paper is to bridge this gap by giving formulas
which, like the formula in \cite{GT}, 
are valid in the much broader Hamiltonian category,  but which are simpler
in the sense that  they
count the contribution over fewer paths.  
Indeed, we want these contributions to be manifestly positive and
integral whenever possible, and to understand geometrically when
this occurs.
This project was inspired by an early version of \cite{Za}, 
where C. Zara used combinatorial
tools to re-derive 
the integral formula in \cite{AJS} and \cite{B}
for the case of a coadjoint orbit of type $A_n$ 
from the formula in \cite{GT},
by taking
limits as the cohomology class
of the symplectic form varies.

Before giving  more precise statements, let us define a few terms.
Let $T$ be a (compact) torus with Lie algebra $\ft$ and lattice $\ell \subset \ft$,
and $(\cdot,\cdot)$ the natural pairing between $\td$ and $\ft$.
Let $T$  act on a compact symplectic manifold $(M,\omega)$
with moment map $\psi \colon M \to \td$.
 By definition,
$$\iota_{X_\xi} \omega = - d \psi^\xi \quad \mbox{for all } \xi \in \ft,$$
where $X_\xi$ denotes the vector field on $M$ generated by the action
and $\psi^\xi(x)  =  (\psi(x),\xi)$.  
In this case, we say that the triple $(M,\omega,\psi)$ is a
{\bf Hamiltonian $\mathbf{T}$-manifold}.
Now, assume that $M$ has a discrete fixed set and
fix a {\bf generic} $\xi \in \ft$, that is,
assume that $ (\eta, \xi) \neq 0$ for each weight $\eta \in \ell^* \subset \td$
in the isotropy representation of $T$ on $T_p M$ for every fixed point $p$.
Given a fixed point  $p \in M^T$, 
let $\lambda(p)$ be the number of 
 positive weights of the isotropy action on $T_p M$.
Let $\Lambda_p^- \in \Sym(\td)$ be the product of these weights,
where $\Sym(\td)$ denotes the symmetric algebra on $\td$.
(Here, we say that $f \in \Sym(\ft^*)$ is {\bf positive} if $(f, \xi) > 0$.)

\begin{definition}\label{definitioncc}
Let $(M,\omega,\psi)$ be a Hamiltonian $T$-manifold with discrete fixed set, and
let $\varphi  = \psi^\xi$ be a generic component of
the moment map.
A cohomology class $\alpha_p \in H^{2 \lambda(p)}_T(M;A)$
is a
{\bf canonical class} at a fixed point $p$ (with respect to $\varphi$)
if
\begin{enumerate}
\item $\alpha_p(p) = \Lambda_p^-$
\item $\alpha_p(q) = 0$ for all $q \in M^T \smallsetminus \{p\}$ such
that $\lambda(q) \leq \lambda(p)$.
\end{enumerate}
\end{definition}

Canonical classes do not  always exist,
but if they exist then they are unique \cite[Lemma 2.7]{GT}.
Moreover, if there exist canonical classes 
$\alpha_p\in H_T^{2\lambda(p)}(M;A)$ for all $p\in M^T$, 
then by Lemmas~\ref{GT1}  and
\ref{kirwan}
below,
the classes $\{\alpha_p\}_{p\in M^T}$ are a basis for $H_T^*(M;A)$ as a 
module over $H^*(BT;A)$; see also \cite[Proposition 2.3]{GT}.
In this case, our goal will be to compute the restrictions $\alpha_p(q)$
for all $p$ and $q \in M^T$ in terms of paths in the canonical graph.

\begin{definition}\labell{def:cangraph}
Let $(M,\omega,\psi)$ be a Hamiltonian $T$-manifold with discrete fixed set, and
let $\varphi = \psi^\xi$ be a generic component of the moment map.
Assume that 
canonical classes $\alpha_p \in H_T^{2 \lambda(p)}(M;A)$ exist for
all $p \in M^T$.
There is a labelled directed graph $(V,E)$ associated to $(M,\omega,\psi,\varphi)$,
called the {\bf canonical graph}, 
defined as follows.
\begin{itemize}
\item The vertex set $V$ is the fixed set $M^T$;
we label each vertex $p \in V$ by its moment image $\psi(p)\in \td$.
\item
The edge set is
\begin{gather*}
 E = \{(r,r') \in M^T \times M^T \mid \lambda(r') - \lambda(r) =1
\mbox{ and } \alpha_r(r')\neq 0 \}; 
\end{gather*}
we  label each edge $(r,r') \in E$ by $\displaystyle\frac{\alpha_r(r')}{\Lambda_{r'}^-} .$
\end{itemize}
\end{definition}

For example, if $M$ is a GKM space, then a pair of 
distinct fixed points $(p,q)$
is an edge in the canonical graph exactly if
 $\lambda(q) = \lambda(p) + 1$
and they are contained in a $2$-sphere that is 
fixed by a codimension one subgroup of $T$; see \S\ref{secGKM}.
Given any directed graph with vertex set $V$ and edge set 
$E\subset V\times V$, a  {\bf path} of {\bf length $k$}  from $p$ to $q$
is a  $(k+1)$-tuple  
$\gamma = (\gamma_1,\ldots,\gamma_{k+1}) \in V^{k+1}$ so
that $\gamma_1 = p$, $\gamma_{k+1} = q$, and $(\gamma_{i},\gamma_{i+1})
\in E$ for all $1 \leq i \leq k$. For any path $\gamma$, we let $|\gamma|$ denote its length.

We can now give our most general theorem, which is proved in \S\ref{S:mostgeneral}.
It gives
a formula for the restriction of a canonical class $\alpha_p$
to a fixed point $q$.

\begin{theorem}\labell{thm:main}
Let  $(M,\omega,\psi)$ be a Hamiltonian $T$-manifold with discrete fixed set,
and let $\varphi = \psi^\xi$ be a generic component of the moment map.
Assume that 
canonical classes $\alpha_p \in H_T^{2 \lambda(p)}(M;A)$ exist for
all $p \in M^T$.
Given fixed points $p$ and $q$,
let $\Sigma(p,q)$ denote the set of paths from $p$ to $q$ in 
the associated canonical graph $(V,E)$.
Given classes $\w_r \in H^2_T(M;\R)$ for all $r \in M^T$,
$$\alpha_p(q) = \Lambda_q^- \sum_{\gamma \in \Sigma(p,q)}
\prod_{i=1}^{|\gamma|}
\frac{ \w_{\gamma_i} (\gamma_{i+1}) - \w_{\gamma_i}(\gamma_{i})}
{ \w_{\gamma_i} (q) - \w_{\gamma_i}(\gamma_{i})}
\frac{\alpha_{\gamma_i}(\gamma_{i+1})}{\Lambda_{\gamma_{i+1}}^-}$$
whenever the right hand side is well-defined, i.e.,
$ \w_{\gamma_i}(q) \neq \w_{\gamma_i}(\gamma_i)$ for all
$\gamma \in \Sigma(p,q)$ and $1 \leq i \leq |\gamma|$.
\end{theorem}

This generalizes
the formula in
\cite{GT}  
whenever $H^2(M;\R) \neq \R$;
cf. Remark~\ref{rmk:symplectic}.
This case includes, for example, 
any generic coadjoint orbit of
dimension greater than two.
At first glance, the formula above doesn't look simpler than the one
in \cite{GT} -- they both involve sums over the same set of paths.
However, a path $\gamma \in \Sigma(p,q)$ contributes $0$ to the formula if 
$\w_{\gamma_i}(\gamma_i) = \w_{\gamma_i}(\gamma_{i+1})$ for some $1 \leq i \leq |\gamma| - 1$.
Most of our paper (the proof of Theorem~\ref{thm:main} itself takes less than a page)
is dedicated to explaining how to choose the cohomology classes $\w_r$ so that 
only a few paths contribute, and proving that in these cases the formula is
manifestly positive whenever the restrictions themselves are all 
positive (see Remark \ref{positive terms}).

In \S\ref{S:HC} we show 
that we can reduce the number of paths
whenever there exists a cohomology class $\w$ 
whose restriction to $H^2(M;\R)$ is
in the closure of
the Hamiltonian cone (see Definition~\ref{hamcone})
and has 
 the property that
$\w(p) = \w(q) \neq \w(r)$ for some edges $(p,q)$ and $(q,r)$ in the canonical graph.
For example,  if $M$ is a GKM manifold which admits an invariant K\"ahler structure,
then it is enough to
find  a cohomology class  
in the closure of the K\"ahler cone that 
vanishes on the two-sphere that corresponds to the edge $(p,q)$, but not on the two-sphere
that corresponds to $(q,r)$.

In \S\ref{bundles}, we show that our technique is particularly powerful when the 
manifold is a ``strong symplectic fibration'' over 
another Hamiltonian manifold; this class includes, for example,
equivariant fiber bundles 
with the property 
that the
projection map intertwines  
compatible invariant complex structures.
Explicit
computations are especially easy in this case.
In particular, in Corollary~\ref{corollary formula} we give
an  inductive formula 
for the restrictions $\alpha_p(q)$
in terms of the
paths in the base and the canonical classes on the fiber.
Finally, by  Theorem~\ref{tower symp},
our formula is {\em integral}
whenever
$M$ is a ``tower'' of complex projective spaces,
that is, a fiber bundle over 
$\CP^{n}$ whose
fiber is also a tower of complex projective spaces. 
More generally, if the fibers $F_j$ are not projective spaces,
but do satisfy $H^*(F_j;A) \simeq H^*\big(\CP^{n_j};A\big)$ for some subring
$A \subset \R$, then the contributions are all polynomials in 
the weights with coefficients in  $A$.

Since coadjoint orbits of type $A_n$ and $C_n$ are both towers of complex
projective spaces,
we immediately get manifestly positive integral formulas
for the restrictions in these cases.
Similarly,  since coadjoint orbits of type $B_n$ 
are towers whose fibers  satisfy $H^*\big(F_j;\Z\big[\frac{1}{2}\big]\big) 
\simeq H^*\big(\CP^{n_j};\Z\big[\frac{1}{2}\big]\big)$, the contribution of
each path is integral when multiplied by a sufficiently large power of $2$.
(In a more recent version of \cite{Za}, Zara also independently obtained
 formulas for $C_n$ and $B_n$ of this type as well.)
Finally, coadjoint orbits of type $B_n$ and $D_n$ are sufficiently close  to
being towers of complex projective spaces that
we can manipulate
the terms to get manifestly positive integral formulas in these
cases as well.

\subsection*{Acknowledgements}
The authors would particularly like to thank Victor Guillemin; without his support and encouragement this project would not have been possible. 
They would also like to thank Catalin Zara for his mathematical insights, and
Rebecca Goldin for her support.
Finally, we would like to thank the anonymous referee, who made numerous suggestions
that improved the exposition.

\section{Canonical classes}
\labell{s:main}

The main goal of this section is to 
review the properties of canonical classes.
However,  we also need to prove 
a slight variation of these results: 
Lemma~\ref{lambda vanish}.

Let's begin by recalling a definition.
Let $A \subset \R$ be a subring (with unit).
The {\bf equivariant cohomology of $M$ with coefficients in $A$} is
$$H_T^*(M;A) = H^*(M \times_T ET;A);$$
it is 
a module over $H^*(BT;A)$.
Moreover, the inclusion 
$M \to M \times_T ET$ induces a natural restriction
map $H_T^*(M;A) \to H^*(M;A)$. 
Here, $ET$ is a contractible space on which $T$ acts freely, and  $BT = ET/T$.

Let $(M,\omega,\psi)$ be a Hamiltonian $T$-manifold with a discrete fixed
set.
Given  a generic $\xi \in \ft$,
the function
$\varphi = \psi^\xi \colon M \to \R$ is an invariant 
Morse function;
the critical set of $\varphi$ is exactly
the fixed set $M^T$.
Our convention for the moment map implies that,
for each $p \in M^T$, the weights in the negative normal bundle $\nu^-(p)$ of 
$\varphi$ at $p$
are
exactly the positive weights of the isotropy action on $T_p M$,
that is, the weights $\eta$ such that $(\eta, \xi)  > 0$.
Hence, the index of $\varphi$ at $p$ is $2 \lambda(p)$, where $\lambda(p)$
is the number of such weights.
In particular $H^1(M;\R)=0$.
Finally, given $\alpha \in H_T^*(M;A)$ and $q \in M^T$,
let $\alpha(q)$ denote the image of $\alpha$ under the
natural restriction map $H_T^*(M;A) \to H_T^*(\{q\};A)$.

Throughout this paper, we will frequently need the following lemma,
which is identical to \cite[Lemma 2.8]{GT} except that here we consider coefficients
in any subring $A \subset \R$ instead of just $\Z$. The proof goes
through without any change.

\begin{lemma}
\labell{GT1}
Let $(M,\omega,\psi)$ be a Hamiltonian $T$-manifold with discrete fixed set, and
let $\varphi = \psi^\xi$ be a generic component of
the moment map.
Given
a canonical class 
$\alpha_p \in H_T^{2 \lambda(p)}(M;A)$
at $p \in M^T$,
$$\alpha_p(q) = 0 \quad\mbox{for all }q \in M^T\smallsetminus \{p\} \mbox{ such that }
\varphi(q) \leq \varphi(p).$$
\end{lemma}
Lemma \ref{GT1} implies that $\varphi(r) < \varphi(r')$ for all $(r,r') \in E$.
Hence, if $\gamma = (\gamma_1,\dots,\gamma_{|\gamma|+1})$ is a path from
$p$ to $q$ in $(V,E)$, then $\varphi(\gamma_i) < \varphi(q)$ for
all $1 \leq i \leq |\gamma|$.  

The following result is due to 
Kirwan \cite{Ki}; see also \cite{GT}, \cite{TW}.

\begin{lemma}[Kirwan]
\labell{kirwan}
Let $(M,\omega,\psi)$ be a Hamiltonian $T$-manifold with discrete fixed set, and
let $\varphi = \psi^\xi$ be a generic component of
the moment map.  For every fixed point
$p$ there exists a class
$\gamma_p \in H_T^{2 \lambda(p)}(M;\Z)$ so that
\begin{itemize}
\item [(1)] $\gamma_p(p) = \Lambda_p^-$, and
\item [(2')] $\gamma_p(q) = 0$ for every $q \in M^T \smallsetminus \{p\}$
such that $\varphi(q) \leq \varphi(p)$.
\end{itemize}
Moreover, for any such classes, the $\{\gamma_p\}_{p \in M^T}$
are a basis for $H_T^*(M;\Z)$ as a module over $H^*(BT;\Z)$.
\end{lemma}

This has the following corollary, which we
have adapted from \cite[Corollary 2.6]{GT} and \cite[Corollary 2.3]{T}.

\begin{corollary}
\labell{GT2}
Let $(M,\omega,\psi)$ be a Hamiltonian $T$-manifold with discrete fixed set, and
let $\varphi = \psi^\xi$ be a generic component of
the moment map.
Fix $p \in M^T$ and  $\beta \in H^{2i}_T(M;A)$
such that  $\beta(q) = 0$ for all $q \in M^T$ 
satisfying $\varphi(q) < \varphi(p)$.
\begin{itemize}
\item
$\beta(p) = x \Lambda_p^-$ for some $x \in H^{2i - 2 \lambda(p)}
(BT;A)$; in particular,
if $\lambda(p) > i$ then $\beta(p) = 0$.
\item Fix cohomology classes
$\{\gamma_q\}_{q \in M^T}$ so that $\gamma_q$ satisfies conditions
(1) and (2') above for each $q \in M^T$.
Then
$$\beta = \sum_{\varphi(q) \geq \varphi(p)} x_q \gamma_q, \quad \mbox{where}  \
x_q \in H^{2i - 2 \lambda(q)}(BT;A) \mbox{ for all }  q.$$
Here, the sum is over all $q \in M^T$ such that $\varphi(q) \geq \varphi(p)$.
\end{itemize}
\end{corollary}

We also need the following closely related fact.

\begin{lemma}\labell{lambda vanish}
Let $(M,\omega,\psi)$ be a Hamiltonian $T$-manifold with discrete fixed set, and
let $\varphi = \psi^\xi$ be a generic component of the moment map.
Assume that 
canonical classes $\alpha_p \in H_T^{2 \lambda(p)}(M;A)$ exist for
all $p \in M^T$.
Fix $p\in M^T$ and $\beta\in H_T^{2i}(M;A)$ such that $\beta(q)=0$ for all $q\in M^T$ so that
$\lambda(q)< \lambda(p)$. Then $$
\beta=\sum_{\lambda(q)\geq\lambda(p)}x_q\alpha_q,\quad \mbox{where } x_q\in 
H^{2i-2\lambda(q)}(BT;A)\mbox{ for all } q.
$$
Here the sum is over all $q\in M^T$ such that $\lambda(q)\geq \lambda(p)$.
\end{lemma}

\begin{proof}
Since  $\{\alpha_q\}_{q\in M^T}$ is a basis for $H_T^*(M;A)$ as a module over $H^*(BT;A)$, 
we can write
$
\beta=\sum_{q\in M^T}x_q\alpha_q,$
where $x_q\in H^{2i-2\lambda(q)}(BT;A)$ for all $q$.
If the claim doesn't hold,
then there exists $q \in M^T$ so that $\lambda(q) < \lambda(p)$ and $x_q \neq 0$,
but $x_r = 0$ for all $r$ such that $\lambda(r) < \lambda(q)$.
Hence by the definition of canonical class $\beta(q) = x_q \Lambda_q^-$. Since $\beta(q)=0$
this is impossible.
\end{proof}

\section{GKM spaces}\labell{secGKM}

We now restrict our attention to an important special case where 
it is especially easy to calculate canonical classes.
A Hamiltonian $T$-manifold $(M,\omega,\psi)$  
is a 
\textbf{GKM (Goresky-Kottwitz-MacPherson) space}
if $M$ has 
isolated 
fixed points and if, 
for every codimension one subgroup $K\subset T$,  the
fixed submanifold $M^K$ has dimension at most two. 
Equivalently, $M$ is a GKM space if 
the weights of the isotropy representation of $T$ on $T_pM$ are pairwise linearly independent
for every fixed point $p \in M^T$.
\begin{definition}
The \textbf{GKM graph} of a GKM space $(M,\omega,\psi)$ is
the labelled directed  graph $(V,\gkme)$,  defined as follows.
\begin{itemize}
 \item The vertex set $V$ is the fixed set $M^T$; we label each 
$p\in M^T$ by its moment image $\psi(p)\in \td$.
\item Given $p\neq q$ in $V$, there is a directed edge $(p,q)\in \gkme$ exactly if there exists a codimension one subgroup $K\subset T$ so that $p$ and $q$ are contained in the same connected component $N$ of $M^K$. We label each edge $(p,q)$ by the weight $\eta(p,q)$ associated to the isotropy representation of $T$ on $T_qN\simeq \C$.
\end{itemize}
\end{definition}

Observe that $(p,q) \in \gkme$ exactly if $(q,p) \in \gkme$. Moreover,  $\eta(p,q) = - \eta(q,p)$,
and  $\psi(q)-\psi(p)$ is a positive multiple of $\eta(p,q)$ for all $(p,q)\in \gkme$.
Additionally, 
the set of weights 
of the isotropy representation on
the tangent space at any point $p \in V$ is
$$\Pi_p = \Pi_p(M) = \big\{ \eta(r,p) \; \big| \; (r,p) \in \gkme \big\}.$$

 \begin{example}\textit{ The complex projective space $\C P^n$. } 
The natural  action of $(S^1)^{n+1}$ on $\C^{n+1}$ 
descends to an effective Hamiltonian action of $T = (S^1)^{n+1}/S^1$
on $\C P^{n}$. 
The associated GKM graph
is the complete graph on $n+1$ fixed points:
$p_1 = [1,0,\ldots,0],\;p_2 = [0,1,\ldots,0],\ldots,\;p_{n+1} = [0,0,\ldots,1]$. 
Finally, the moment 
image of $p_i$ is
$\frac{1}{n+1}\sum_{j=1}^{n+1} \big(x_j-x_i \big)$, 
and the weight associated to the edge 
$(p_i,p_j)$ is  $x_i-x_j$.
Here, we let $x_1,\dots,x_{n+1}$ be the standard basis for $(\R^{n+1})^*$
and identify $\ft^*$ with  $\big\{ \mu \in (\R^{n+1})^* \big| \sum \mu_i = 0 \big\}$.
\end{example}

Now fix a generic component of the moment map $\varphi=\psi^{\xi}$. 
As we mentioned in the previous section,
the  set of weights  in the isotropy representation on 
the negative normal bundle of $\varphi$
at $p$ is the set of positive weights in 
$\Pi_p(M)$. 
Hence, $\lambda(p)$  is 
the number of edges $(r,p)\in \gkme$ such that $\varphi(r)<\varphi(p)$, and
$$\Lambda_p^- = \prod_{\substack{\eta \in \Pi_p(M) \\ (\eta,\xi) > 0}}\eta.$$

It is possible to strengthen Lemma~\ref{GT1} when
$M$ is a GKM space.
We say that a  path $\gamma = (\gamma_1,\dots,\gamma_{|\gamma|+1}$) in
$(V,\gkme)$ is {\bf ascending} if
$ \varphi(\gamma_{i}) < \varphi(\gamma_{i+1})$ for all $i$.

\begin{lemma}\labell{GT1GKM}
Let $(M,\omega,\psi)$ be a GKM space, and let
$\varphi = \psi^\xi$ be a generic component of the moment 
map.
Given a canonical class  $\alpha_p \in H_T^{2 \lambda(p)}(M;A)$
at $p \in M^T$,
$\alpha_p(q) = 0$ for all $q \in M^T$
such that there are no ascending paths from 
$p$ to $q$ in $(V,\gkme)
.$
\end{lemma}

\begin{proof}
Consider any $q \in M^T$ so that
$\alpha_p(q) \neq 0$ but
$\alpha_p(r) = 0$ for each edge $(r,q) \in \gkme$ such
that $\varphi(r) < \varphi(q)$.
Then $\alpha_p(q)$ is a non-zero multiple of $\eta(r,q)$ for all $(r,q) \in \gkme$
such that $\varphi(r) < \varphi(q)$.  
(To see this,  recall that for each $(r,q) \in \gkme$ there exists 
a sphere $N\subset M$ containing $r$ and $q$ which is fixed by
the codimension one subgroup associated to $\eta$.
)
Since these weights
are pairwise linearly independent, this implies that $\alpha_p$
has degree at least $2 \lambda(q)$, that is, 
$\lambda(q) \leq \lambda(p)$.
By the definition of canonical class, this is impossible unless
$p = q$.  The claim follows.
\end{proof}

We say that $\varphi$ is 
\textbf{index increasing} if $\lambda(p) < \lambda(q)$ 
for every edge $(p,q)\in \gkme$ such that $\varphi(p) < \varphi(q)$. 
In this case, integral canonical classes exist
and it is straightforward to compute the restriction of a canonical class 
$\alpha _p$ to $q$ for any $p$ and $q$ in $M^T$ such that 
$\lambda(q)-\lambda(p)=1$.
Conversely, if there exist  canonical classes $\alpha_p\in H_T^*(M;\Q)$
for all $p\in M^T$, then $\varphi$ is index increasing
\cite[Remark 4.2]{GT}.

More specifically, let
$\xi^{\circ}=\{\beta \in \td \mid (\beta,\xi) =0\}$.
Given  $\eta\in \td$, 
let $\varrho_{\eta} \colon \Sym(\td)\r\Sym(\td)$ be the homomorphism of 
symmetric algebras induced by the projection map which sends $X\in \td$ to 
$X-\frac{ (X,\xi)}{ (\eta,\xi) }\eta\in \xi^{\circ}\subset 
\td$.
Following \cite{GZ}, 
for any $(p,q)\in \gkme$ 
we define 
$$
\Theta(p,q)=\frac{\varrho_{\eta(p,q)}(\Lambda_p^{-})}{\varrho_{\eta(p,q)}
\left(\frac{\Lambda_q^{-}}{\eta(p,q)}\right)}\in \Sym(\td)_0,
$$
where $\Sym(\td)_0$ denotes the ring of fractions of $\Sym(\td)$. Observe 
that  $\varrho_{\eta(p,q)}
\Big(\frac{\Lambda_q^{-}}{\eta(p,q)}\Big)$ 
is not zero, since by the GKM assumption the weights at each fixed point
 are pairwise linearly independent. 
The theorem below was proved in \cite{GZ} over the rationals and
then extended to the integers in \cite{GT}.

\begin{theorem}\labell{existence canonical classes} 
Let $(M,\omega,\psi)$ be a GKM space, and let $(V,\gkme)$ be the associated GKM graph. Let $\varphi=\psi^{\xi}$ be a generic component of the moment map; assume that $\varphi$ is index increasing. Then 
\begin{enumerate}
\item There exist canonical classes $\alpha_p\in H_T^{2 \lambda(p)}(M;\Z)$ for all $p\in M^T$.
\item  Given fixed points $p$ and $q$ such that $\lambda(q)-\lambda(p)=1$,
\begin{equation*}
\alpha_p(q)=\left\{ \begin{array}{ll}
                    \displaystyle\Lambda_q^{-}\frac{\Theta(p,q)}{\eta(p,q)} & \mbox{  if  }(p,q)\in \gkme, \mbox{  and  }\\
                    & \\
                    0 & \mbox{  if  }(p,q)\notin \gkme\\
                    \end{array}
                    \right.
\end{equation*}
\item $\Theta(p,q)\in \Z\smallsetminus \{0\}$ for all $(p,q)\in \gkme$ such that $\lambda(q)-\lambda(p)=1$.
\end{enumerate}
\end{theorem}

In particular, the associated canonical graph  has
vertex set $V = M^T$ and edge set
\begin{equation*}\labell{Egkm}
E = \{ (r,r') \in \gkme \mid \lambda(r') - \lambda(r) = 1 \}.
\end{equation*}

\section{The most general theorem}
\labell{S:mostgeneral}

In this section, we will prove our most general theorem, Theorem~\ref{thm:main}.
As we mentioned in the introduction, it
is a generalization of \cite[Theorem 1.2]{GT}.
More precisely, it is more general 
whenever 
$H^2(M;\R) \neq \R$; see Remark~\ref{rmk:symplectic}.
The main advantage of our formula is that it
usually
allows us to express $\alpha_p(q)$ as a sum over fewer paths.
For the reader's convenience, we will recall the statement of Theorem~\ref{thm:main}.

\begin{quote}
{\em
Let  $(M,\omega,\psi)$ be a Hamiltonian $T$-manifold with discrete fixed set,
and let $\varphi = \psi^\xi$ be a generic component of the moment map.
Assume that 
canonical classes $\alpha_p \in H_T^{2 \lambda(p)}(M;A)$ exist for
all $p \in M^T$.
Given fixed points $p$ and $q$,
let $\Sigma(p,q)$ denote the set of paths from $p$ to $q$ in 
the associated canonical graph $(V,E)$.
Given classes $\w_r \in H^2_T(M;\R)$ for all $r \in M^T$,
$$\alpha_p(q) = \Lambda_q^- \sum_{\gamma \in \Sigma(p,q)}
\prod_{i=1}^{|\gamma|}
\frac{ \w_{\gamma_i} (\gamma_{i+1}) - \w_{\gamma_i}(\gamma_{i})}
{ \w_{\gamma_i} (q) - \w_{\gamma_i}(\gamma_{i})}
\frac{\alpha_{\gamma_i}(\gamma_{i+1})}{\Lambda_{\gamma_{i+1}}^-}$$
whenever the right hand side is well-defined, i.e.,
$ \w_{\gamma_i}(q) \neq \w_{\gamma_i}(\gamma_i)$ for all
$\gamma \in \Sigma(p,q)$ and $1 \leq i \leq |\gamma|$.
}
\end{quote}

\begin{remark}\labell{rmk:symplectic} \rm
By Lemma~\ref{GT1},  
$\varphi(\gamma_i) < \varphi(q)$ for all $\gamma \in \Sigma(p,q)$
and $1 \leq i \leq |\gamma|$;
a fortiori, $\psi(\gamma_i) \neq \psi(q)$.
Therefore, the right hand side of the equation above is  well-defined
if $\w_r $ is a non-zero multiple of $[\omega + \psi]$ for all $r \in M^T$.
(Here we
are using the Cartan model for the equivariant cohomology of $M$.)
In this case, 
the theorem agrees with
\cite[Theorem 1.2]{GT}.
\end{remark}

Note that a path $\gamma \in \Sigma(p,q)$ contributes $0$ to the formula above
exactly
if there exists $1\leq i \leq |\gamma|-1$
 such that $\w_{\gamma_i}(\gamma_i) = \w_{\gamma_i}(\gamma_{i+1})$
 but $\w_{\gamma_i}(q) \neq \w_{\gamma_i}(\gamma_i)$.
Generally speaking, the best result will come from choosing each class $\w_r$
so that $\w_r(r) \neq \w_r(q)$, but $\w_r(r) = \w_r(s)$ 
for as  many edges $(r,s) \in E$ as possible.
In practice, instead of trying to pick the optimal class at each fixed point,
we will often fix an 
ordered list of 
classes.  For
each fixed point we will
just pick the first class that satisfies the hypotheses of Theorem \ref{thm:main}.
As we show below, as long as the forms satisfy the technical condition
\eqref{tech}, this technique gives an elegant answer.
In the next two sections, we will explain  natural geometric
conditions which guarantee that \eqref{tech} is satisfied.

\begin{corollary}\labell{height}
Let  $(M,\omega,\psi)$ be a Hamiltonian $T$-manifold with discrete fixed set, and
let $\varphi = \psi^\xi$ be a generic component of the moment map.
Assume that 
canonical classes $\alpha_p \in H_T^{2 \lambda(p)}(M;A)$ exist for
all $p \in M^T$.
Pick 
classes\footnote{
In practice, there often exists a symplectic form $\omega_j \in \Omega^2(M)$
with moment map
$\psi_j \colon M \to \ft^*$ such that $[\omega_j + \psi_j]  = 
\w_j \in H_T^2(M;\R)$ for each $j$.  In this case, $\w_j(p) = \psi_j(p)$ for all $p \in M^T$;
in particular, $\w_j^\xi(p) = \psi_j^\xi(p)$.  However we do not insist that such symplectic forms exist; we allow the general case.}
$\w_1,\w_2,\dots,\w_k$  in $H_T^2(M;\R)$
such that, for each $j$, 
\begin{equation}\labell{tech}
\alpha_p(q) = 0  
\ \ 
\forall \ 
p,q
\in M^T \mbox{ such that }
\w_j(q) \neq \w_j(p) \mbox{ and } \w_j^\xi(q) \leq \w_j^\xi(p)\;,
\end{equation}
where for each $p\in M^T$, $\w_j^\xi(p)$ denotes $(\w_j(p),\xi)$.

Assume that for each $(r,r') \in E$,
there exists  $j \in \{1,\ldots,k \}$  such that 
$\w_j(r) \neq \w_j(r')$, and define  
$$h(r,r') = \min \big\{ j 
\; \big| \; \w_j(r) \neq \w_j(r') \big\}
\quad \mbox{for all  } (r,r') \in E
.$$
Given $p$ and $q$ in $M^T$, let $\Sigma(p,q)$ denote the set of paths
in the associated canonical graph $(V,E)$ from $p$ to $q$. 
Then
\begin{gather*}
\alpha_p(q) = 
\Lambda_q^- 
\sum_{\gamma \in C(p,q)}
\prod_{i=1}^{|\gamma|} \frac
{ \w_{h(\gamma_i,\gamma_{i+1})}(\gamma_{i+1}) - \w_{h(\gamma_i,\gamma_{i+1})}(\gamma_i)}
{ \w_{h(\gamma_i,\gamma_{i+1})}(q) - \w_{h(\gamma_i,\gamma_{i+1})}(\gamma_i)}
\frac{\alpha_{\gamma_i}(\gamma_{i+1})}{\Lambda_{\gamma_{i+1}}^-},
\
\mbox{where}  \\
C(p,q)= \left\{ \left. \gamma \in \Sigma(p,q) \ \right| \   h(\gamma_1,\gamma_2) \leq h(\gamma_2,\gamma_3) \leq \dots
\leq h(\gamma_{|\gamma|},\gamma_{|\gamma|+1})\right\}. 
\end{gather*}
\end{corollary}

\begin{remark}[\textbf{Positivity}]
\labell{positive terms} \rm
Note that, if $\alpha_r(r')$ is positive for every edge $(r,r')$ in $E$, then
the equation above
is manifestly positive, in the sense that every non-zero term
is positive; a fortiori, the restriction $\alpha_p(q)$ is positive for all
fixed points $p$ and $q$.
To see this, note that $\Lambda_r^-$ is positive by definition for all $r \in M^T$,
while $\w_i(r') - \w_i(r)$ is either positive or zero for each edge $(r,r') \in E$
by assumption (1).
Therefore, the formulas in Theorem~\ref{point} and Theorem~\ref{tower symp}.
which are both proved using the above corollary, a manifestly positive whenever
the restrictions themselves are manifestly positive.
In contrast, in
 general the restriction $\alpha_p(q)$ might not be positive 
(cf. \cite[Example 5.2]{GT}).
\end{remark}

We are now ready to prove our claims.

\begin{proof}[Proof of Theorem \ref{thm:main}]
Since $(\w_p - \w_p(p))(p) = 0$ and $\alpha_p$ is a canonical
class at $p$,
the restriction
$\alpha_p (\w_p - \w_p(p))(r)$
 is trivial
for all $r \in M^T$ such that $\lambda(r) \leq \lambda(p)$.
By Lemma \ref{lambda vanish},
this implies that we can write
\begin{equation*}\labell{main1}
\alpha_p \big(\w_p - \w_p(p)\big) = \sum_{\lambda(r) > \lambda(p)} x_r \alpha_r,
 \ 
\mbox{where }x_r \in H^{2 \lambda(p)- 2\lambda(r) + 2 }(BT;\R)
\ 
\forall \ r
\end{equation*}
By the definition of canonical class,
evaluating the above equation at $r$ implies
that 
 \begin{equation*}
\big(\w_p(r) - \w_p(p)\big) \frac{\alpha_p(r)}{\Lambda_r^-} = x_r \in \R 
 \ 
\mbox{for all }r \in M^T \mbox{ such that }\lambda(r) = \lambda(p) + 1.
\end{equation*}
Moreover,
by dimensional arguments, $x_r = 0$ for all $r \in M^T$ such that
$\lambda(r) > \lambda(p) + 1$.
Hence, 
$$
\alpha_p (\w_p - \w_p(p)) = 
\sum_{(p,r) \in E} (\w_p(r) - \w_p(p)) \frac{\alpha_p(r)}{\Lambda_r^-} \alpha_r.
$$
Restricting to $q$ and dividing by $\w_p(q) -\w_p(p)$ (which is not zero by
assumption), we have
\begin{equation*}
\alpha_p(q)  =  \sum_{(p,r) \in E}
\frac{\w_p(r) - \w_p(p)}{\w_p(q) - \w_p(p)}
\frac{\alpha_p(r)}{\Lambda_r^-}
\alpha_r(q).
\end{equation*}
Since the claim is obvious if $\lambda(q) - \lambda(p) \leq  1$,
the claim now follows by induction.
\end{proof}

\begin{proof}[Proof of Corollary~\ref{height}]
Definition \ref{def:cangraph} and hypothesis \eqref{tech} in Corollary \ref{height} 
 imply that  
$$
\mbox{either }
\w_j^\xi(r) < \w_j^\xi(r')\mbox{ or }\w_j(r) = \w_j(r')\mbox{ for all }(r,r') \in E 
\mbox{ and } 
1 \leq j \leq k.$$

Therefore, if $\gamma$ is a  path from $r$ to $q$ with at least one edge then 
either 
$\w_j^\xi(\gamma_i) < \w_j^\xi(q)$ or  $\w_j(\gamma_i) = \w_j(\gamma_{i+1}) = \w_j(q)$
for each $j$ and for each $i \leq |\gamma|$.
Since, by assumption, there exists $j \in \{1,\ldots,k\}$ such that
$\w_j(\gamma_1)  \neq  \w_j(\gamma_2)$,
this implies that $\w_j^\xi(r) < \w_j^\xi(q)$. A fortiori,
$\w_j(r) \neq \w_j(q)$, and so 
we can define
$$h(r,q) = \min \left\{ j 
\mid \w_j(r) \neq \w_j(q) \right\}
\
\forall \
r \in M^T \smallsetminus \{q\} \mbox{ such that } 
\Sigma(r,q) \neq \emptyset.$$
Moreover, if 
$\w_j(\gamma_i) = \w_j(q)$
for some $\gamma \in \Sigma(p,q)$ and $i  \leq |\gamma|$, 
then $\w_j(\gamma_i) = \w_j(\gamma_{i+1}) = \w_j(q)$ as well.
Therefore, 
\begin{equation}
\labell{height1}
h(\gamma_i, q) \leq h(\gamma_{i+1},q )
 \ 
 \mbox{and} 
 \ 
 h(\gamma_i,q) \leq h(\gamma_i,
\gamma_{i+1})
 \ 
 \mbox{for all } 1 \leq i \leq |\gamma|-1.
\end{equation}

The 
hypotheses of Theorem~\ref{thm:main} will
be satisfied if  we  let the class associated to $r \in M^T$  be
\begin{equation*}
\begin{cases}
\w_{h(r,q)}  & \mbox{ if } r \neq q \mbox{ and } \Sigma(r,q) \neq \emptyset, \mbox{ and } \\
0& \mbox{otherwise}.
\end{cases}
\end{equation*}
Therefore, 
\begin{equation*}
\alpha_p(q) = 
\Lambda_q^- 
\sum_{\gamma \in \Sigma(p,q)}
\prod_{i=1}^{|\gamma|} \frac
{ \w_{h(\gamma_i,q)}(\gamma_{i+1}) - \w_{h(\gamma_i,q)}(\gamma_i)}
{ \w_{h(\gamma_i,q)}(q) - \w_{h(\gamma_i,q)}(\gamma_i)}
\frac{\alpha_{\gamma_i}(\gamma_{i+1})}{\Lambda_{\gamma_{i+1}}^-}.
\end{equation*}
Moreover,  
a path $\gamma \in \Sigma(p,q)$ contributes $0$ to the formula above 
if $\w_{h(\gamma_i,q)}(\gamma_i) = \w_{h(\gamma_i,q)}(\gamma_{i+1})$ for some $i<|\gamma|$.
Therefore 
we only need to consider paths $\gamma$ from $p$ to $q$
so that
$$h(\gamma_i,\gamma_{i+1}) \leq h(\gamma_i,q) \quad \mbox{for all } 1 \leq i \leq |\gamma|.$$
Combining this fact together with \eqref{height1}, we may restrict to paths $\gamma$ so that
 \begin{gather*}
 h( \gamma_1,\gamma_2) \leq h(\gamma_2,\gamma_3) \leq \dots \leq 
h(\gamma_{|\gamma|},\gamma_{|\gamma| + 1})
\quad \mbox{and} \\
\quad h(\gamma_i,\gamma_{i+1}) = h(\gamma_i,q) \quad \mbox{for all } 1 \leq i \leq |\gamma|.
\end{gather*}
\end{proof}

Finally, the lemma below, which we will use in Section~\ref{bundles},
follows from an argument nearly identical to the first three sentences of the proof
of Theorem \ref{thm:main}.

\begin{lemma}\labell{int}
Let $(M,\omega,\psi)$ be a Hamiltonian $T$-manifold with discrete fixed set,
and let $\varphi = \psi^\xi$ be a generic component of the moment map.
Assume that canonical classes  $\alpha_p \in H_T^{2 \lambda(p)}(M;A)$ exist for all
$p \in M^T$; let $(V,E)$ be the canonical graph. 
Given a class $\w \in H_T^2(M;A)$,
$$ \big( \w(r) - \w(p) \big) \frac{\alpha_p(r)}{\Lambda_{r}^-}
\ \in A \quad \mbox{for all } (p,r) \in E. $$
\end{lemma}

\section{The Hamiltonian cone}
\labell{S:HC}

In this section, we give our first application of Theorem \ref{thm:main}.
Here, we use the Hamiltonian cone to pick the closed equivariant
two-forms and characterize which paths contribute to the formula. 

\begin{definition}\labell{hamcone}
Let a torus $T$ act on a 
manifold $M$.
The {\bf Hamiltonian cone} is 
the set of classes 
$\Hh \subset H^2(M;\R)$
which can be represented by
an invariant symplectic form 
that has a  
moment map.
\end{definition}

\begin{theorem}\labell{point}
Let  $(M,\omega,\psi)$ be a Hamiltonian $T$-manifold with discrete fixed set, and
let $\varphi = \psi^\xi$ be a generic component of the moment map.
Assume that 
canonical classes $\alpha_p \in H_T^{2 \lambda(p)}(M;A)$ exist for
all $p \in M^T$.
Pick  classes
$\w_1,\dots,\w_k\in H_T^2(M;\R)$
that restrict to classes in the closure of the component
of the Hamiltonian cone $\Hh \subset H^2(M;\R)$ containing $[\omega]$.
Assume that for each $(r,r') \in E$
there exists  $j$  such that 
$\w_j(r) \neq \w_j(r')$, and define  
$$h(r,r') = \min \big\{ j  \; \big| \;  \w_j(r) \neq \w_j(r') \big\}
\quad \mbox{for all } (r,r') \in E
.$$
Given $p$ and $q$ in $M^T$, let $\Sigma(p,q)$ denote the set of paths
from $p$ to $q$ in  the associated canonical graph $(V,E)$.
Then
\begin{equation*}
\alpha_p(q) = 
\Lambda_q^- 
\sum_{\gamma \in C(p,q)}
\prod_{i=1}^{|\gamma|} \frac
{ \w_{h(\gamma_i,\gamma_{i+1})}(\gamma_{i+1}) - \w_{h(\gamma_i,\gamma_{i+1})}(\gamma_i)}
{ \w_{h(\gamma_i,\gamma_{i+1})}(q) - \w_{h(\gamma_i,\gamma_{i+1})}(\gamma_i)}
\frac{\alpha_{\gamma_i}(\gamma_{i+1})}{\Lambda_{\gamma_{i+1}}^-},
\quad \mbox{where} 
\end{equation*}
$$C(p,q)= \left\{ \left. \gamma \in \Sigma(p,q) \ \right| \   h(\gamma_1,\gamma_2) \leq h(\gamma_2,\gamma_3) \leq \dots
\leq h(\gamma_{|\gamma|},\gamma_{|\gamma|+1})\right\}. $$
\end{theorem}

\begin{remark}\rm
Assume that the following conditions hold:
\begin{enumerate}
\item[(X)] the restrictions of the $\w_i$  form a basis for $H^2(M;\R)$; and
\item[(Y)] the restriction of $\sum_i a_i \w_i$  lies in $\Hh$ for every 
positive $k$-tuple $a \in \R_+^k$.
\end{enumerate}
In this case, Theorem~\ref{point} can also be proved using the limit techniques
found in \cite{Za},
instead of Theorem~\ref{thm:main}.
(The argument still relies on  Lemmas~\ref{cone} and \ref{Propb1}.)
Note that, in this case, we don't need to assume that for each $(r,r') \in E$
there exists $j$ such that $\w_j(r) \neq \w_j(r')$; this holds automatically.
\end{remark}

\begin{remark} \rm
Assume that the torus $T$ acts on a compact manifold $M$, preserving a complex structure
$J \colon TM \to TM.$  If the fixed set is empty the Hamiltonian cone $\Hh$ is empty,
so assume that $M^T \neq \emptyset$.
The {\bf K\"ahler cone} of $M$ is the set
of classes in $H^2(M;\R)$ which can be represented by
a compatible symplectic form.
Since $T$ is compact, 
we can represent
every such class by an invariant symplectic form by averaging.
Moreover, by Frankel's theorem, the action is always Hamiltonian.
Hence, the K\"ahler cone is a subset of the 
Hamiltonian cone.
(Note that Lemma~\ref{cone} is obvious if 
$[\omega']$
is in the K\"ahler cone containing 
$[\omega]$.)
Analogous statements hold 
if $J$ is an almost complex structure
and $H^1(M;\R) = 0$.

Note also that the  K\"ahler cone is convex because any convex combination
of compatible symplectic forms is itself a compatible symplectic form.
In contrast, a convex combination of arbitrary 
symplectic forms
may or may not be symplectic.
\end{remark}

\begin{lemma}\labell{cone}
Let $(M,\omega,\psi)$ be a Hamiltonian $T$-manifold, and let $\varphi = \psi^\xi$
be a generic component of the moment map.
Let $\omega'$ be a symplectic form on $M$ with moment
$\psi'$ so that 
$[\omega']$ lies in the component of 
$\Hh \subset H^2(M;\R)$ containing $[\omega]$.
Then $(\Lambda'_p)^-$, the product of the positive
weights of the isotropy representation 
of $T$ on $(T_pM, \omega')$, is $\Lambda_p^-$ for all $p \in M^T$.
\end{lemma}

\begin{proof}
Let $\varphi' = (\psi')^\xi$.
It is sufficient to prove the claim for all $\omega'$ such 
that $[\omega']$ lies in some neighborhood  of $[\omega]$.
Therefore, we may assume that 
$$\varphi(r) < \varphi(s)\;\;\; \Rightarrow\;\;\; \varphi'(r) < \varphi'(s)
\quad \mbox{for all  } r \mbox{ and }s \in M^T.$$

Fix $p \in M^T$.
By applying Lemma~\ref{kirwan} to $\varphi$,  there exists a class $\gamma_p
\in H_T^{2 \lambda(p)}(M;\Z)$ 
so that $\gamma_p(p) = \Lambda_p^-$ and
$\gamma_p(q) = 0$ for every $q \in M^T \smallsetminus \{p\}$ such
that $\varphi(q) \leq \varphi(p)$.  By the assumption above,
this implies that $\gamma_p(q) = 0$ for every $q \in M^T \smallsetminus \{p\}$
such that $\varphi'(q) < \varphi'(p)$.  By applying Corollary~\ref{GT2}
to $\varphi'$, this implies  that $\Lambda_p^- = \gamma_p(p)$
is a multiple of $(\Lambda_p')^-$.
Since a nearly identical argument shows that
$(\Lambda_p')^-$ is a multiple of $\Lambda_p^-$,
the claim follows from the fact that these are both positive. 
\end{proof}
 
 \begin{lemma}\labell{Propb1}
Let $(M,\omega,\psi)$  be a Hamiltonian
$T$-manifold with discrete fixed set, and
let $\varphi = \psi^\xi$ be a generic component of the moment map.
Fix a class $\w\in H_T^2(M;\R)$ 
that restricts to a class
in the closure of the component of $\mathcal{H}
\subset H^2(M;\R)$ containing $[\omega]$. 
Given a canonical class   $\alpha_p \in H_T^{2 \lambda(p)}(M;A)$ 
at $p \in M^T$ (with respect to $\varphi$),
$$\alpha_p(q) = 0 \quad \mbox{for all } p \mbox{ and } q \in M^T \mbox{ such that }
\w(q) \neq \w(p) \mbox{ and } \w^\xi(q) \leq \w^\xi(p).
$$
\end{lemma}

\begin{proof}
By perturbing $\xi$ slightly, if necessary, we may assume
that $\w(p) = \w(q)$ exactly if $\w^\xi(p) = \w^\xi(q)$ for all $p$ and $q$ in $M^T$.
Hence, there exists $\epsilon > 0$ so that
$\w^\xi(q) < \w^\xi(p) - \epsilon$ for all 
$p$ and
$q \in M^T$ such
that $\w(q) \neq \w(p)$ and $\w^\xi(q) \leq \w^\xi(p)$.
By assumption, there exists a symplectic form $\omega'$ with moment map $\psi'$ so that 
\begin{equation}\label{(a)}
 \big| (\psi')^\xi(p) - \w^\xi(p) \big| < \frac{1}{2} \epsilon\quad\mbox{for all }p \in M^T,
 \end{equation}
and $[\omega']$ lies in the closure of the component of $\mathcal{H}
\subset H^2(M;\R)$ containing $[\omega]$.
By Lemma \ref{cone}, 
the latter fact implies that
the product of the positive
weights for the isotropy action on $(T_pM, \omega')$ is $\Lambda_p^-$ for all $p \in M^T$. Hence, by the definition of canonical class,
 $\alpha_p$ is also the canonical class at $p$ with respect to $\varphi'$.
By Lemma~\ref{GT1}, this implies that 
\begin{equation*} 
\alpha_p(q) = 0 \quad \mbox{for all } p \mbox{ and } q \in M^T \mbox{ such that } (\psi')^\xi(q) 
< (\psi')^\xi(p). 
\end{equation*}
Finally,
\eqref{(a)} implies that
$(\psi')^\xi(q) < (\psi')^\xi(p)$ for all $q \in M^T$
such that $\w^\xi(q) < \w^\xi(p) - \epsilon$.
\end{proof}

\begin{proof}[Proof of Theorem \ref{point}]
The claim follows immediately from Lemma~\ref{Propb1}
and Corollary~\ref{height}.

\end{proof}

Finally, we make the following observation, which
we will not need in this paper.

\begin{lemma}
Let $(M,\omega,\psi)$ be a Hamiltonian $T$-manifold. Let $\omega'$ be a symplectic form on $M$ with moment
map
$\psi'$ so that 
$[\omega']$ lies in the component of 
 $\Hh \subset H^2(M;\R)$ containing $[\omega]$.
Then $c_1(M) = c_1'(M)$, where $c_1(M)$ and $c_1'(M) \in H_T^2(M;\Z)$ are the first equivariant Chern class
associated to $\omega$ and $\omega'$, respectively.
\end{lemma}

\begin{proof}
Let $\varphi = \psi^\xi$ be a generic
component of the moment map.
Since the weights in the representations $(T_qM,\omega)$ and $(T_qM,\omega')$ agree up
to sign, Lemma \ref{cone} implies immediately that 
 $c_1(M)(q) = c'_1(M)(q)$ for all $q \in M^T$
such that $\lambda(q) \leq 1$.
By Lemma~\ref{lambda vanish}, this implies that $c_1(M) - c'_1(M) = 0$. 
\end{proof}

\section{Fiber bundles}
\labell{bundles}

In this section, we show how to use Theorem~\ref{thm:main}
(and Corollary~\ref{height})
to get effective formulas for the restrictions $\alpha_p(q)$ in
the case that our Hamiltonian $T$-manifold is a 
fiber  bundle
over a Hamiltonian $T$-manifold 
(and  certain technical restrictions hold).
In certain very nice cases, 
such as when $M$ is a 
``tower of complex projective spaces''
(see Definition \ref{def tower})
and the restrictions of the canonical classes are positive,
the contribution from each path will be 
a positive 
integer
multiple of
the product of 
 positive weights.
More precisely, 
let $(M,\omega,\psi)$ and $\big(\Mt,\ot,\pt\big)$ be  Hamiltonian $T$-manifolds.
We will consider the following maps.

\begin{definition}
A  map  $\pi \colon M \to \Mt$ is a 
{\bf  strong symplectic fibration}\footnote{
Every map satisfying (1) is a symplectic fibration;
see \cite[Lemma 6.2]{MS}.}
 if
\begin{enumerate}
\item[(1)] the map $\pi$ is an equivariant fiber bundle with  {\bf symplectic fibers}, 
that is,  the restriction of $\omega$ to the fiber $\Mh_p = \pi^{-1}(\pi(p))$ 
is symplectic for all $p \in M$;
and
\item[(2)] as symplectic representations
$(T_p M, \omega) \simeq (T_p \Mh_p, \omega|_{\Mh_p}) \oplus (T_{\pi(p)} \Mt, \ot)$ for
all $p \in M^T$.
\end{enumerate}
\end{definition}

\begin{example}\labell{example essf}\rm
There are several situations where an equivariant fiber bundle $\pi
\colon M \to \Mt$ is {\em automatically} a strong
symplectic fibration.
\begin{itemize}
\item[(i)] Let $J$ and $\Jt$ be compatible almost complex structures 
on $M$ and $\Mt$, respectively. 
If $\pi \colon M \to \Mt$ intertwines
$J$ and $\Jt$, i.e. $d\pi \circ J = \Jt \circ d\pi$,
then $\pi$ is  a strong symplectic fibration. 
The fibers are symplectic because $T_p \Mh_p$ is $J$ invariant for all $p \in M$.
For all $p \in M^T$,
the symplectic perpendicular $H_p = ( T_p \Mh_p)^\omega$ is a complex
subspace and  $T_p M = T_p \Mh_p \oplus H_p$ as complex representations.
Finally, $\pi$ induces an isomorphism of
$(H_p,\omega|_{H_p})$ and $(T_{\pi(p)} \Mt,\ot)$ as symplectic representations.
\item[(ii)]
If $\pi$ has  
symplectic fibers and 
$\omega_\mu =  \mu \omega + (1 - \mu) \pi^*(\ot)$ is symplectic for all 
$\mu  \in (0,1]$,
then $\pi$ is a strong symplectic fibration. 
Since $(0,1]$ is connected,
$(T_p M,\omega_\mu) \simeq (T_p M, \omega)$ as symplectic representations
for all $\mu \in (0,1]$ and all $p \in M^T$.
But by Lemma~\ref{omegat},
for any sufficiently small $\mu > 0$,
$(T_p M, \omega_\mu) \simeq (T_p \Mh_p, \omega|_{\Mh_p})
\oplus (T_{\pi(p)} \Mt, \ot)$ for all $p \in M^T$.
\end{itemize}
\end{example}

\begin{remark}\labell{essf}\rm
Let $\pi \colon M \to \Mt$ be  any equivariant fiber bundle with symplectic 
fibers.
Then
$(T_p M, \omega) \simeq 
(T_p \Mh_p, \omega|_{\Mh_p}) \oplus (H_p, \omega|_{H_p})$ for
all $p \in M^T$, where
 $H_p \subset T_pM$ is the symplectic perpendicular to $T_p \Mh_p$.
Moreover, $d \pi \colon H_p \to T_{\pi(p)} \Mt$ is an 
equivariant isomorphism, and
so the weights in the symplectic representations
$(H_p, \omega|_{H_p})$ and $(T_{\pi(p)} \Mt,\ot)$ necessarily agree up to sign.
The map $\pi$ is a strong symplectic fibration 
if they agree exactly.
\end{remark}

\begin{definition}\labell{def tower}
Let $\{(M_j,\omega_j,\psi_j)\}_{j=0}^{k}$ be Hamiltonian
$T$-manifolds with discrete fixed sets
and let $\{\rho_j \colon  M_{j+1}\rightarrow M_j\}_{j=0}^{k-1}$
be  strong symplectic fibrations. Assume that $M_0$ is a point.
Then, given a subring $A \subseteq \R$,
$M_k$ is a {\bf tower of complex projective spaces over $\mathbf{A}$} 
if the fiber $F_j$ of  $\rho_j$ 
satisfies $H^*(F_j;A) \simeq H^*\big(\CP^{\frac{1}{2} \dim(F_j)};A\big)$ 
as rings for all $j$.
\end{definition}

\noindent
{\bf Notation:}
Given a strong symplectic fibration $\pi \colon M \to \Mt$
and a generic component of the moment map $\varphi = \psi^\xi \colon M \to \R$,
let $\Lh_p^-$ denote the equivariant Euler class of the negative
normal
bundle of the restriction $\varphi |_{\Mh_p}$ at $p \in \Mh_p$, and let 
$2 \lh(p)$ denote the index of $p$ in $\Mh_p$, for all $p \in M^T$.
(Since the restriction of $\omega$ to $\Mh_p$ is symplectic,
the restriction of $\psi$ to $\Mh_p$ is a moment map.)
Similarly,  let $\Lt_q^-$ denote the equivariant
Euler class of the negative normal bundle of $\vt = \pt^\xi$ at $q$
and let $2 \lt(q)$ denote the index of $q \in \Mt$, for all $q \in \Mt^T$.

Finally, given a subring $A \subseteq \R$,
let $A_+ = \{t \in A \mid t >  0 \}$ and
let $A^\times \subset A$ denote the set of units.

We can now state our main theorem in this section.

\begin{theorem}\labell{tower symp}
Let $\{(M_j,\omega_j,\psi_j)\}_{j=0}^{k}$ be Hamiltonian
$T$-manifolds with discrete fixed sets
and let $\{\rho_j \colon  M_{j+1}\rightarrow M_j\}_{j=0}^{k-1}$
be  strong symplectic fibrations.
Let $\varphi_k = \psi_k^\xi$ be a generic component of the moment map.
Fix a subring $A \subseteq \R$.
Assume that 
$M_0$ is a point and 
that
canonical classes $\alpha_p \in H_T^{2 \lambda(p)}(M_k;A)$ exist for
all $p \in M_k^T$.
Let\footnote{
In this paper, our convention is that an empty composition or product is the identity.
Hence $\pi_k = \id_{M_k}$.} 
$\pi_j=\rho_j\circ \rho_{j+1}\circ \cdots \circ \rho_{k-1} \colon M_k \to
M_j$ 
and let $\pb_j=\pi_j^*(\psi_j) \colon M_k \to \ft^*$ for all
$j$.
Finally, define 
$$
h(r,r')=\min \{j\in \{1,\ldots,k\}\mid \pi_j(r)\neq \pi_j(r')\}
\  \mbox{for all distinct }
r, 
r'\mbox{ in }M^T.
$$

\begin{enumerate}
\item[1.]

Given $p$ and $q$ in $M_k^T$, let $\Sigma(p,q)$ denote the set of paths from 
$p$ to $q$ 
in the associated canonical graph $(V,E)$; then
\begin{gather*}
\alpha_p(q) = \sum_{\gamma \in C(p,q)}\Xi(\gamma),\quad\mbox{where }\\
\Xi(\gamma)=\Lambda_q^- 
\prod_{i=1}^{|\gamma|} \frac
{ \pb_{h(\gamma_i,\gamma_{i+1})}(\gamma_{i+1}) - \pb_{h(\gamma_i,\gamma_{i+1})}(\gamma_i)}
{ \pb_{h(\gamma_i,\gamma_{i+1})}(q) - \pb_{h(\gamma_i,\gamma_{i+1})}(\gamma_i)}
\frac{\alpha_{\gamma_i} (\gamma_{i+1})}{\Lambda_{\gamma_{i+1}} ^-} 
\ \ \forall  
\gamma\in C(p,q),
\mbox{ and}\\ 
C(p,q) = \big\{\gamma=(\gamma_1,\ldots,\gamma_{|\gamma|+1}) \in \Sigma(p,q) 
\,\big|\,  h(\gamma_1,\gamma_2) \leq \dots
\leq h(\gamma_{|\gamma|},\gamma_{|\gamma|+1})\big\}. 
\end{gather*}
\item[2.]
Assume that $M_k$ is a tower of complex projective
spaces over $A$.
Then for each path $\gamma \in C(p,q)$, 
\begin{itemize}
\item $\Xi(\gamma)$
can be written as the product of positive weights in
$\ell^*$ and a constant $C$ in $A$; moreover, $C > 0$ if
$\alpha_{r} (r')$ is positive
for all  $(r,r') \in E$. 
\item If $(M,\omega,\psi)$ is a GKM space, then
$\Xi(\gamma)$ can be written as the product of distinct positive weights in 
$\Pi_q(M)$ and a constant $C$ in $A$.
Finally, if $\Theta(r,r') > 0$ for all $(r,r') \in E$, then $C > 0$;
similarly, if $\Theta(r,r') \in A^\times$ for all $(r,r') \in E$,
then $C \in A^\times$.
\end{itemize}
\end{enumerate}
\end{theorem}

\begin{remark}\rm
In fact, if $M_k$ is a GKM
space, then our proof demonstrates that 
claim (1) holds whenever
$\rho_j \colon M_{j+1} \to M_j$ is a
weight preserving  map for all $j$;
(see Definition~\ref{preserving}).
\end{remark}

\begin{remark}\labell{discrete tower}\rm
If $M_k$ has a discrete fixed set (or is a GKM space), 
then $M_j$
has a discrete fixed set (or is a GKM space) for all $j$.
To see this, consider any $q \in M_j^T$.  Since the fiber 
$\rho_j^{-1}(q)$ is a Hamiltonian 
$T$-manifold, there exists  $r \in M^T_{j+1}$ such that 
$\rho_j(r) = q$.
Since the differential $d \rho_j$ is surjective, the set of weights in the
representation $T_q M_j$ is a subset of the set of weights in 
$T_r M_{j+1}$.
\end{remark}

Theorem \ref{tower symp} has the following useful corollary.
\begin{corollary}\labell{corollary formula}
Let $(M,\omega,\psi)$ and $\big(\Mt,\ot,\pt\big)$ be Hamiltonian $T$-manifolds 
with discrete fixed sets, and
let $\pi \colon M \to \Mt$ be a
strong symplectic fibration.
Let $\varphi = \psi^\xi$ be a generic component of the moment map.
Assume that 
canonical classes $\alpha_p \in H_T^{2 \lambda(p)}(M;A)$ exist for
all $p \in M^T$.
Fix $p$ and $q$ in $M^T$. 
\begin{enumerate}
\item[1.] There exist canonical
classes $\ah_s \in H_T^{2 \lh(s)}(\Mh_q;A)$
on the fiber $\Mh_q = \pi^{-1}(\pi(q))$
for all $s \in \Mh_q^T$.
\item[2.]
Given $s \in \Mh_q^T$, let $\overline{\Sigma}(p,s)$
denote the set of paths $\gamma=(\gamma_1,\ldots,\gamma_{k+1})$ from $p$ to $s$
in the associated canonical graph $(V,E)$ such that $\pi(\gamma_i) \neq \pi(\gamma_{i+1})$ 
for all $i$.
Then 
\begin{gather*}
\alpha_p(q)= \sum _{s\in \widehat{M}_q^T} 
  \Big( \sum_{\gamma \in \overline{\Sigma}(p,s)} 
P(\gamma) \Big) \ah_s(q) ,
\end{gather*}
where for all $ s \in \Mh_q^T$ and $\gamma \in \overline{\Sigma}(p,s)$, 
\begin{gather*}
P(\gamma) =  \Lt^-_{\pi(s)} \prod_{i=1}^{|\gamma|}
\frac
 {\pt(\pi(\gamma_{i+1}))-\pt(\pi(\gamma_i))}
 {\pt(\pi(s))-\pt(\pi(\gamma_i))}
 \frac{\alpha_{\gamma_i} (\gamma_{i+1})}{\Lambda_{\gamma_{i+1}} ^-} \; . 
\end{gather*}
\item[3.]
Assume that
$H^*\big(\Mt;A\big) \simeq H^*\big(\CP^{\frac{1}{2} \dim(\Mt)};A\big)$ as rings. 
Then for all $s\in \Mh_q^T$ and each path $\gamma\in \overline{\Sigma}(p,s)$
\begin{itemize}
\item $P(\gamma)$
can be written as the product of positive weights in
$\ell^*$ and a constant $C$ in $A$; moreover, $C > 0$ if
$\alpha_{r} (r')$ is positive
for all  $(r,r') \in E$.
\item If $(M,\omega,\psi)$ is a GKM space, then
$P(\gamma)$ can be written as the product of distinct positive weights in 
$\Pi_q(M)$ and a constant $C$ in $A$.
Finally, if $\Theta(r,r') > 0$ for all $(r,r') \in E$, then $C > 0$;
similarly, if $\Theta(r,r') \in A^\times$ for all $(r,r') \in E$,
then $C \in A^\times$.
\end{itemize}
\end{enumerate}
\end{corollary}
In Lemma \ref{explicit P},
we give a different explicit description of $P(\gamma)$ 
in the case that $(M,\omega,\psi)$ is a GKM space
and $H^*(\Mt;A) \simeq H^*(\CP^{\frac{1}{2}\dim \Mt};A)$.

\subsection*{Proof of Claim 1.\ of Theorem~\ref{tower symp}}

We are now ready to begin the proof of the first part of our main theorem.
We will begin with 
the special case of GKM spaces, where
the proof is easier and
the main ideas are more transparent. 
However, the proof in the general case 
on page \pageref{tgc}
is self-contained;
the reader may skip directly to that case.

\subsubsection*{The case of GKM spaces}

Let $(M, \omega,\psi)$ and $\big(\Mt,\ot,\pt\big)$ be  GKM spaces, and
let $(V,\gkme)$ and 
$(\widetilde{V},\gkmet)$ be the associated GKM graphs.
If $\pi \colon M\r \Mt$ is an equivariant map,
the following statements hold:
\begin{itemize}
\item Given a vertex $p \in V$, $\pi(p) \in \widetilde{V}$.
\item Given an edge $e = (p,q) \in \gkme$, either $\pi(p) = \pi(q) \in \widetilde{V}$ or 
$\pi(e) = (\pi(p),\pi(q)) \in \gkmet$
and $\eta(\pi(e))$ is a 
non-zero
multiple of $\eta(e)$.
\end{itemize}
To see this, let $K \subset T$ be the maximal subgroup so
that $p$ and $q$ are contained in the same connected component $N \subset M^K$.
Since  $\pi$ is equivariant, either $\pi(N)$ is a fixed point in $\Mt$, 
or $\pi(N)$ is the connected
component of $\Mt^{K'}$ 
for some subgroup $K' \varsubsetneq T$ which contains $K$.

\begin{definition}\labell{def:horizontal}
We will say that an edge $(p,q) \in (V,\gkme)$ is {\bf horizontal}
(with respect to $\pi$) if $\pi(p) \neq \pi(q)$; moreover, we 
will say that a path $\gamma$ in $(V,\gkme)$ is {\bf horizontal} 
if all its edges are horizontal. 
\end{definition}

If  $\pi \colon M \to \Mt$ is an equivariant fiber bundle
and  $e \in \gkme$ is a horizontal  edge, then $\eta(e)=\pm\eta(\pi(e))$.
However, this need not hold  for arbitrary equivariant maps.

\begin{definition}\labell{preserving}
We will say that a map
$\pi \colon M \to \Mt$ is {\bf weight preserving} 
if it is equivariant and $\eta(e) = \eta(\pi(e))$
for all horizontal edges $(p,q) \in \gkme$. 
\end{definition}

Note that the composition of two weight preserving  maps
is itself weight preserving.  In contrast, the composition
of two  strong symplectic fibrations may not be
a strong symplectic fibration; indeed, it may not
have symplectic fibers.
However, the following assertion is clear; cf. Remark~\ref{essf}.

\begin{lemma}\labell{GKMlambda}
Let $(M,\omega,\psi)$ and $\big(\Mt,\ot,\pt\big)$ be GKM spaces.
If $\pi \colon M \to \Mt$ is a strong symplectic fibration then
$\pi$ is weight preserving.
\end{lemma}

To prove Claim 1.,
we need to check that the pull-back of a symplectic form and moment map
by a weight preserving  map satisfies criterion \eqref{tech}
of Corollary~\ref{height}.
We will do this in two steps.

\begin{lemma}\labell{increasing path}
Let $(M,\omega,\psi)$ and $(\Mt,\ot,\pt)$ be GKM spaces, and let 
$\pi\colon M \to \Mt$ be a weight preserving  map. 
Let $\varphi = \psi^\xi$ 
be a generic component of the 
moment map.
Given a horizontal edge $(p,q)$ in the GKM graph 
associated to $M$,
$$
\psi^{\xi}(q)-\psi^{\xi}(p)>0\quad\mbox{if and only if}
\quad\pt^{\xi}(\pi(q))-\pt^{\xi}(\pi(p))>0\;.
$$
\end{lemma}

\begin{proof}
Since $(p,q)$ is a horizontal edge and
$\pi$ is a weight preserving map,
$\eta(\pi(p),\pi(q))
= \eta(p,q)$.  Therefore,
$\psi^{\xi}(q)-\psi^{\xi}(p)$ and $\pt^{\xi}(\pi(q))-\pt^{\xi}(\pi(p))$
are both positive multiples of  $\eta(p,q)$.
\end{proof}

\begin{lemma}\labell{GKMb'}
Let $(M,\omega,\psi)$
and $\big(\Mt,\ot,\pt\big)$  be GKM spaces,
and let  $\pi \colon M \to \Mt$ be a  weight preserving  map.
Let $\varphi = \psi^\xi $ 
be a generic component of the moment map.
Given a canonical class
$\alpha_p \in H^{2 \lambda(p)}_T(M;A)$ at  $p \in M^T$,
$$\alpha_p(q) = 0 \quad \mbox{for all } q \in M^T \mbox{ such that } \pi(q) \neq \pi(p)
\mbox{ and } \pt^\xi(\pi(q)) \leq \pt^\xi(\pi(p)).$$
\end{lemma}

\begin{proof}
Assume that $\alpha_p(q) \neq 0$ for some $q \in M^T$.
By Lemma~\ref{GT1GKM}, there exists an 
ascending path $\gamma$ from $p$ to $q$ in $(V,\gkme)$.
By Lemma~\ref{increasing path} and the definition of ascending,
$$\pt^\xi(\pi(\gamma_i)) < \pt^\xi(\pi(\gamma_{i+1})) \quad
\mbox{or} \quad \pi(\gamma_i) = \pi(\gamma_{i+1}) \quad
\mbox{for each } i.$$
\end{proof}

We are now ready to prove Claim 1.
Let
$\w_j = \pi_j^*(\omega_j + \psi_j)
\in H_T^2(M_k;\R)$
for each $j \in \{1,\dots,k\}$.
Since $\pi_k = \id_{M_k}$, it is obvious that $\w_k(r) \neq \w_k(r')$
for all $(r,r') \in \gkme$.
By Lemma~\ref{GKMlambda},  each $\rho_i$ is a weight preserving map,
and so
$\pi_j$ is a 
weight preserving  map for all $j$. 
Therefore, 
in the case of GKM spaces,
Claim 1.\ of Theorem~\ref{tower symp} 
is an immediate consequence of Corollary~\ref{height}
and Lemma~\ref{GKMb'}.

\subsubsection*{The general case}
\labell{tgc}

The proof in the general case is nearly 
identical,
except that it takes more work
to prove Lemma~\ref{sympb}, the analog of Lemma~\ref{GKMb'}.

\begin{lemma}\labell{omegat}
Let $(M,\omega,\psi)$ and $\big(\Mt,\ot,\pt\big)$ be
Hamiltonian $T$-manifolds, and let $\pi \colon M \to \Mt$
be a equivariant fiber bundle with symplectic fibers.
For all sufficiently small $t > 0$,
\begin{enumerate}
\item the two-form $\omega_t = \pi^*(\ot) + t \omega$ is 
symplectic; moreover,
\item as
symplectic representations $(T_p M, \omega_t) \simeq
(T_p \Mh_p, \omega|_{\Mh_p}) \oplus (T_{\pi(p)} \Mt,\ot)$ 
for all $p \in M^T$,
where $\Mh_p$ denotes the fiber $\pi^{-1}(\pi(p))$.
\end{enumerate}
\end{lemma}

\begin{proof}

Let $V \subset TM$ be the kernel of the map $d \pi \colon
TM \to T\Mt$.
By assumption, $\pi$ is a submersion;
hence, $V \subset TM$ is a subbundle.
Since  we have assumed that $\pi $ has symplectic fibers,
the restriction $\omega|_{V}$ is  symplectic.
Since $\pi^*(\ot)|_{V} = 0$, this
implies that the restriction $\omega_t|_{V} = t \omega|_{V}$
is symplectic and that 
$(V_p,\omega_t) \simeq (V_p,\omega)$ 
for all $p \in M^T$ and  $t > 0$.

Let $H = V^{\omega} \subset TM$ be the symplectic perpendicular
to $V$ with respect to $\omega$.
Since  $\pi^*(\ot)|_V = 0$, $H$
is also symplectically perpendicular to $V$ with respect
to $\omega_t$ for all $t \geq 0$.
Moreover, since $\omega|_{V}$ is symplectic,
$H \subset TM$ is a subbundle and
$TM = V \oplus H$.
Thus, the map $d \pi \colon H \to T\Mt$ is an isomorphism.
Since $\ot$ is symplectic, this implies that
the restriction  $\pi^*(\ot)|_H$ is symplectic and
that  $(H_p, \pi^*(\ot)|_{H_p}) \simeq (T_{\pi(p)} \Mt,\ot)$  for all $p \in M^T$.
Since being symplectic is an open condition and $M$ is compact,
analogous statements hold for  $\omega_t$ 
for all sufficiently small $t > 0$.
The claim follows immediately.
\end{proof}

\begin{lemma}\labell{Propb}
Let $(M,\omega,\psi)$  be a Hamiltonian
$T$-manifold with discrete fixed set.
Let $\varphi = \psi^\xi$ be a generic component of the moment map,
and let $\vb \colon M \to \R$ be an invariant Morse-Bott function.
Assume that  for all $\epsilon > 0$
there exists a  symplectic form $\omega' \in \Omega^2(M)$
with moment map $\psi'$ such that:
\begin{itemize}
\item [(a)] $| (\psi')^\xi(x) - \vb(x) | < \epsilon$
for all $x \in M$; and
\item [(b)]  the product of the positive weights for the isotropy action of $T$ on $(T_pM,\omega')$
is
$\Lambda_p^-$
for all $p \in M^T$.
\end{itemize}
If   $\alpha_p \in H_T^{2 \lambda(p)}(M;A)$ is the canonical class 
(with respect to $\varphi$) at $p \in M^T$,
and $\Mh_p$ is the critical component of $\vb$ that
contains $p$, then
\begin{equation*}
\alpha_p(q) = 0 \quad \mbox{for all } q \in M^T \mbox{ so that }
 q \not\in \Mh_p \mbox{ and } \vb(q) \leq \vb(p).
\end{equation*}
Moreover,  the restriction of $\alpha_p$ to $\vb^{-1}(-\infty, \vb(p) - \delta)$
vanishes for all $\delta > 0$.
\end{lemma}

\begin{proof}
We may assume without loss of generality that $\vb(p) = 0$.
  By  assumption, for any $\epsilon > 0$
there exists a symplectic form $\omega'$ with moment
map $\psi'$ such that (a) and (b) hold.
Let $\varphi'$ be $(\psi')^{\xi}$.
By (b) it follows that $\alpha_p$ 
is also the canonical class at $p$ with respect to $\varphi'$; so
Lemma~\ref{GT1} implies that
$$ \alpha_p(q) = 0 \quad \mbox{for all }q \in M^T \mbox{ such that }
\varphi'(q) < \varphi'(p).$$
By injectivity, this implies that the restriction of $\alpha_p$
to $(\varphi')^{-1}(-\infty,\varphi'(p))$ vanishes.
Finally, (a) implies that
$$(\vb)^{-1}(-\infty, -2 \epsilon) \subset (\varphi')^{-1}(-\infty,\varphi'(p)),$$
and so the restriction of $\alpha_p$ to $\vb^{-1}(-\infty, - 2 \epsilon)$
vanishes.

Since $\vb$ is a  Morse-Bott function there exists  $\epsilon > 0$
so that $0$ is the only critical value of $\vb$ in 
$[-2 \epsilon, 2 \epsilon]$.  
Since the restriction of $\alpha_p$ to $\vb^{-1}(-\infty, - 2 \epsilon)$
vanishes,
the restriction
of $\alpha_p$ to $\Mh_p$ is a multiple of the equivariant Euler class of the
negative normal bundle  of $\vb$ at $\Mh_p$,
and so there exists
$\alpha'_p \in 
H^{2 \lambda(p)}_T\big((\vb)^{-1}(-\infty,  2 \epsilon);A \big)$
so that 
$\alpha'_p|_{\Mh_p} = \alpha_p|_{\Mh_p}$, but
$\alpha'_p|_C = 0$ for every other critical set $C$ of $\vb$ so
that $\vb(C) \leq 2 \epsilon$.
Moreover, since $\vb$ is invariant and the fixed set is discrete,
every fixed point is critical.
Hence,
\begin{gather}
\labell{Propb5}
\alpha'_p(q) =  \alpha_p(q) \quad \mbox{for all }q \in \Mh_p \cap M^T,
\quad \mbox{and} \\ 
\labell{Propb6}
\alpha_p'(q) = 0 \quad \mbox{for all }q \in M^T
\mbox{ such that } q \not\in \Mh_p \mbox{ and }\vb(q) \leq 2\epsilon. 
\end{gather}

By (a),
$(\varphi')^{-1}(-\infty, \epsilon) \subset (\vb)^{-1}(-\infty, 2\epsilon)$.
Hence,  we can restrict
$\alpha'_p$ to 
$(\varphi')^{-1}(-\infty,\epsilon)$; moreover, this restriction
satisfies \eqref{Propb5}
and 
\begin{equation}
\labell{Propb9}
\alpha_p'(q) = 0 \quad \mbox{for all }q \in M^T
\mbox{ such that } q \not\in \Mh_p \mbox{ and }\varphi'(q) < \epsilon. 
\end{equation}
By surjectivity, we can extend $\alpha'_p$
to a class
(which we still call $\alpha'_p$) on $M$
with the same properties.
Moreover, by the definition of canonical class,
\begin{equation*}
\alpha_p(q) = 0 \quad \mbox{for all }q \in M^T \smallsetminus \{p\} 
\mbox{ such that }\lambda(q) \leq \lambda(p).
\end{equation*}
Therefore, by
\eqref{Propb5} and
\eqref{Propb9}, 
\begin{equation*}
\alpha_p(q) = \alpha'_p(q) \quad \mbox{for all } q \in M^T \mbox{ such that }\varphi'(q) < \epsilon \mbox{ and }
\lambda(q) \leq \lambda(p).
\end{equation*}
Assume that there exists $r \in M^T$ such that 
$\alpha_p(r) \neq \alpha'_p(r)$ and $\varphi'(r) < \epsilon$ but
$\alpha_p(s) = \alpha'_p(s)$ for all $s \in M^T$ such that $\varphi'(s) < \varphi'(r)$.
By the equation above, this implies that $\lambda(r) > \lambda(p)$.
Since $\beta = \alpha_p - \alpha'_p$ has degree $2 \lambda(p)$,
this contradicts Lemma~\ref{GT2}.
Therefore, 
\begin{equation*}
\alpha_p(q) = \alpha'_p(q) \quad \mbox{for all } q \in M^T \mbox{ such that }\varphi'(q) < \epsilon.
\end{equation*}
Finally, since 
$(\vb)^{-1}((-\infty,0]) \subset
 (\varphi')^{-1}(-\infty, \epsilon)$
by (a),
this implies that
\begin{equation*}
\alpha_p(q) = \alpha'_p(q) \quad \mbox{for all } q \in M^T \mbox{ such that }\vb(q) \leq 0.
\end{equation*}
Therefore,  the claim follows immediately from \eqref{Propb6}.
\end{proof}

\begin{lemma}\labell{sympb}
Let $\{(M_j,\omega_j,\psi_j)\}_{j=1}^{k}$ be Hamiltonian
$T$-manifolds with discrete fixed sets
and let $\{\rho_j \colon  M_{j+1}\rightarrow M_j\}_{j=1}^{k-1}$
be strong symplectic fibrations.
Let $\varphi_k = \psi_k^\xi$ be a generic component of the moment map.
Let  $\pi = \rho_1 \circ \rho_2
\circ \dots \circ \rho_{k-1} \colon M_k \to M_1$.
Given a canonical class $\alpha_p \in H_T^{2 \lambda(p)}(M_k;A)$ 
at $p \in M_k^T$,
$$\alpha_p(q) = 0 
\
\mbox{for all } q \in M_k^T 
\mbox{ so that } \pi(q) \neq \pi(p) \mbox{ and } \psi_1^\xi(\pi(q)) \leq
\psi_1^\xi(\pi(p)).$$
\end{lemma}

\begin{proof}

Given $q \in M_k^T$,
let $q_j  = (\rho_j \circ \rho_{j+1} \circ \dots \circ \rho_{k-1})(q) \in M_j$ 
for all $j$, 
let $X_j = \rho_{j-1}^{-1}(q_{j-1}) \subset M_j$ for $j > 1$,
and let $X_1 = M_1$.
By the definition of  strong symplectic fibration and induction on $k$,
as symplectic representations 
$$
(T_q M_k, \omega_k) \simeq
(T_{q_1} X_1, \omega_1|_{X_1} ) \oplus \dots 
\oplus (T_{q_k} X_k, \omega_k|_{X_k}).$$
Therefore, by Lemma~\ref{omegat} and induction on $k$,
for any $\epsilon > 0$ 
there exists a symplectic form $\omega_k' \in \Omega^2(M_k)$ 
with moment map
$\psi_k'$ such that:
\begin{itemize}
\item $\big| (\psi_k')^\xi(x) - \pi^*(\psi_1)^\xi(x) \big| < \epsilon$
for all $x \in M_k$; and
\item  
as symplectic
representations  $(T_q M_k,\omega_k) \simeq (T_q M_k, \omega_k')$ 
for all $q \in M_k^T$.
\end{itemize}

Since $M_1$ has a discrete fixed set,
$\psi_1^\xi \colon M_1 \to \R$ is a Morse function on $M_1$
with critical set $M_1^T$.
Since $\pi$ is an equivariant fiber bundle, this
implies that  $\pi^*(\psi_1)^\xi$ 
is an invariant Morse-Bott function on $M$, 
and that the critical component  of $\pi^*(\psi_1)^\xi$
that contains  $p \in M_k^T$ is the fiber $\pi^{-1}(\pi(p))$. 
Therefore,  the claim follows from
Lemma~\ref{Propb}.
\end{proof}

We are now ready to prove Claim 1.
Let 
$\w_j = [\pi_j^*(\omega_j + \psi_j)]
\in H_T^2(M_k;\R)$
for each $j \in \{1,\dots,k\}$.
Since $\pi_k = \id_{M_k}$, Lemma~\ref{GT1} implies 
that $\w_k(r) \neq \w_k(r')$
for all $(r,r') \in E$.
Therefore, Claim 1.\ of Theorem~\ref{tower symp} 
is an immediate consequence of Corollary~\ref{height}
and Lemma~\ref{sympb}.

\subsection*{Proof of Claim 2.\ of Theorem~\ref{tower symp}}

Let a 
maximal
subtorus 
$T \subset SU(n+1)$ 
act on $(\CP^n,\omega)$,
and let $\varphi$
be a generic component of
the  moment 
map 
$\psi \colon \CP^n \to \td$.
If $[\omega]$ generates $H^2(\CP^n;\Z)$ then 
$$\Lambda_p^- = \prod_{\varphi(y) < \varphi(p)} \psi(p) - \psi(y),$$
where the sum is over all $y\in (\CP^n)^T$ such that 
$\varphi(y)<\varphi(p)$.
The next lemma, which is the key ingredient in the proof of Claim 2.\ of
Theorem~\ref{tower symp} and Claim 3.\ 
of Corollary~\ref{corollary
formula}, generalizes this fact to  other manifolds with 
isomorphic cohomology rings.

\begin{definition}\labell{magnitude}
Fix a GKM space
$(M,\omega,\psi)$ 
with GKM graph $(V,\gkme)$. The {\bf magnitude}
of an edge $(r,s)\in \gkme$  
is $$m(r,s)=\displaystyle\frac{\psi(s)-\psi(r)}{\eta(r,s)}\;.$$
\end{definition}

\begin{lemma}\labell{fake}
Let $(M,\omega,\psi)$ be a Hamiltonian $T$-manifold with discrete fixed 
set, and let
$\varphi = \psi^\xi$ be a generic component of the moment map.
Assume that $[\omega]$ generates $H^2(M;\Z)$ and that 
$H^*(M;A)\simeq H^*\big(\CP^{\frac{1}{2} \dim M};A\big)$ as rings. Given $p\in M^T$, fix
a subset $S \subset \{y\in M^T\ \mid \varphi(y)<\varphi(p)\}$. Then
\begin{itemize}
\item
$\Lambda_p^-\prod _{y\in S}\frac{1}{\psi(p)-\psi(y)}$ 
can be written as the product of positive weights in $\ell^*$ and a constant in $A_+$.
\item 
If
 $(M,\omega,\psi)$ is a GKM space with GKM graph $(V,\gkme)$,  
then 
$m(r,s)$ is a unit in $A_+$ for all $(r,s)\in \gkme$ such that $\varphi(r)<\varphi(s)$. In particular
$\Lambda_p^-\prod _{y\in S}\frac{1}{\psi(p)-\psi(y)}$ 
can be written as the product of distinct positive weights in $\Pi_p(M)$
and a unit in $A_+$.
\end{itemize}
\end{lemma}

\begin{proof}
Since the fixed set is discrete and $\varphi$ is a perfect
Morse function, there is exactly one fixed point of index $2i$ for 
all $i \in \big\{0,\dots,\frac{1}{2} \dim(M) \big\}$.
Therefore, there  are exactly $\lambda(p)$ fixed points $y$ with $\lambda(y) < \lambda(p)$.
Moreover, by \cite[Lemma 2.7]{LT}, the fact that $[\omega]$ is integral
implies that
$[\omega + \psi - \psi(y)] \in H_T^2(M;\Z)$ for all $y \in M^T$.
Therefore, we may define a class
$$\beta  = \prod_{\lambda(y) < \lambda(p)} [\omega + \psi - \psi(y)] \ \in \  
H_T^{2 \lambda(p)} (M;\Z),$$
where the product is over all $y \in M^T$ such that $\lambda(y) < \lambda(p)$.

Since $H^{2i}\big(M;\R\big) = 
H^{2i}\big(\CP^{\frac{1}{2} \dim (M )};\R\big)$ for all $i$,
\cite[Proposition 3.4]{T} 
(and the fact that rational $\xi \in \ft$ are dense) 
implies that
\begin{equation}\labell{fake2}
 \varphi(y) < \varphi(p) \quad \mbox{exactly if} \quad \lambda(y) < \lambda(p) \qquad
\mbox{for all} \ y \in M^T.
\end{equation}
Since $\beta(y) = 0$ for all $y \in M^T$ such that
$\lambda(y) < \lambda(p)$, 
Lemma~\ref{kirwan},
Corollary~\ref{GT2}, and \eqref{fake2} together imply that we can write
\begin{equation*}
 \beta   = \sum_{\lambda(y) \geq \lambda(p)} x_y \gamma_y,
\end{equation*}
where the sum is over $y \in M^T$ such that $\lambda(y) \geq \lambda(p)$,
$\gamma_y \in H_T^{2 \lambda(y)}(M;\Z)$,
$x_y \in H^{2\lambda(p) - 2 \lambda(y)}(BT;\Z)$ for all $y \in M^T$,
and $\{ \gamma_y\}_{y \in M^T}$ is a basis for $H_T^*(M;\Z)$
as a $H^*(BT;\Z)$ module.
Since $p$ is the only fixed point  with index $2 \lambda(p)$, 
by degree considerations this implies that
\begin{equation}\labell{fake3}
\beta = x_p \gamma_p, \quad \mbox{where} \ x_p \in \Z. 
\end{equation}

Since  $[\omega]^{\lambda(p)}$ is the image of $\beta$ under the natural
restriction map from $H^*_T(M;\Z)$ to $H^*(M;\Z)$,
this implies that
$$
[\omega]^{\lambda(p)}=x_p\gt_p 
\quad\mbox{for all }p\in M^T, 
$$
where $\gt_p$, the restriction of $\gamma_p$, generates $H^{2 \lambda(p)}(M;\Z)$.
Moreover, since we have assumed that $[\omega]$ generates $H^2(M;\Z)$ and that
$H^*(M;A) \simeq H^*\big(\CP^{\frac{1}{2}\dim M};A\big)$ as rings,
$[\omega]^{\lambda(p)}$ generates $H^{2 \lambda(p)}(M;A)$.
Hence the equation above implies that $x_p$ must be invertible in $A$.
Therefore, evaluating both sides of \eqref{fake3} at $p$, 
\begin{equation}\labell{xp}
{\Lambda_p^-}
\prod_{\varphi(y) < \varphi(p)} 
 \frac{1}{  \psi(p) - \psi(y)  } 
= \frac{1}{x_p} \quad \in A.
\end{equation}

Now observe that 
$\psi(p)-\psi(y) $ is the product of a positive integer and a positive weight in $\ell^*$
for all $y\in M^T$ such that $\varphi(y)<\varphi(p)$. This proves the first claim.

If $M$ is a GKM space, then since
the GKM graph is $\frac{1}{2}\dim(M)$-valent and
has $\frac{1}{2}\dim(M)+1$ vertices, it is a complete graph. 
Therefore 
$\psi(p)-\psi(y)=m(y,p)\eta(y,p)$ for all  $y \in M^T$; moreover $m(y,p)\in \Z_+$ and $\eta(y,p)$ is a positive weight in $\Pi_p(M)$ for all $y\in M^T$
such that $\varphi(y) <  \varphi(p)$.
By \eqref{xp} we have that 
\begin{equation}\labell{xp2}
x_p=\displaystyle\prod_{\varphi(y) < \varphi(p)} m(y,p)
\end{equation}
is
a unit in $A_+$, which implies that $m(y,p)$ is a unit in $A_+$
for all $y,p\in M^T$ such that $\varphi(y)<\varphi(p)$.  
Finally observe that the weights in $\Pi_p(M)$ are all distinct since $M$
is a GKM space, and the second claim follows immediately.
(Notice that in the GKM case \eqref{fake2} directly
follows from the fact that the GKM graph is complete.)
\end{proof}

\begin{remark}\rm
Fix a GKM space
$(M,\omega,\psi)$ 
with GKM graph $(V,\gkme)$; 
assume 
that $H^*(M;\R)\simeq H^*\big(\CP^{\frac{1}{2}\dim M};\R\big)$
and 
that $[\omega]$ generates $H^2(M;\Z)$.
Then by \eqref{xp2},
 the magnitudes of the
edges of $(V,\gkme)$ determine the ring structure of $H^*(M;\Z)$.
\end{remark}

We 
also need a technical lemma.

\begin{lemma}\labell{fiberonto}
Let $(M,\omega,\psi)$ and $\big(\Mt,\ot,\pt\big)$ be Hamiltonian $T$-manifolds,
and 
let $\pi \colon M \to \Mt$ be a
strong symplectic fibration.
The natural restriction map
$ H_T^*(M;\Z) \to H^*\big(\Mh_p;\Z\big)$ is surjective
for all $p\in M$,
where $\Mh_p$ is the fiber $\pi^{-1}(\pi(p))$.
\end{lemma}

\begin{proof}
Let $\varphi = \psi^\xi$ be a generic component of the moment map.
Since $\Mt \times_T ET$ is connected
and $\pi$ induces a fiber bundle $M \times_T ET \to \Mt \times_T ET$ with fiber
$\Mh_p$, 
we may assume that  $\pi(p) \in \Mt^T$ is the minimal fixed point. 
Hence, $\Lt^-_{\pi(p)} = 1$.

Consider any point $q \in \Mh_p^T$.
By Lemma~\ref{kirwan},
there exists a class 
$\gamma_q \in H_T^{2 \lambda(q)}(M;\Z)$
such that
$\gamma_q(q)  = \Lambda_q^-$
and $\gamma_q(r) = 0$ for  all $r \in M^T\smallsetminus \{q\}$ such that $\varphi(r) \leq
\varphi(q)$.
By definition of strong symplectic fibration and the
paragraph above, 
$\Lambda_q^- = \Lt_{\pi(q)}^- \Lh_q^- = \Lt_{\pi(p)} \Lh_q^- = \Lh_q^-$.
Therefore, if $\bh_q$ denotes the restriction of $\gamma_q$ to $\Mh_p$,
then
$\bh_q(q)  = \Lh_q^-$, and
$\bh_q(r) = 0$ for  all $r \in \Mh_q^T\smallsetminus \{q\}$ such that $\varphi(r) \leq
\varphi(q)$.
By Lemma~\ref{kirwan}, this implies that
 $\big\{\bh_q\big\}_{q \in \Mh_p^T}$  is
a basis for $H_T^*(\Mh_p;\Z)$ as a $H^*(BT;\Z)$ module.
Hence, the map  $H_T^*(M;\Z) \to H_T^*(\Mh_p;\Z)$ is surjective.
Finally,  since the fixed set is discrete,
$H^*(M^T;\Z)$ is torsion-free, and so 
 the natural restriction map
$H_T^*(\Mh_p;\Z) \to H^*(\Mh_p;\Z)$ is surjective;
see, for example,  \cite[\S2]{T}.
\end{proof}
\begin{proof}[Proof of Claim 2.\ of Theorem~\ref{tower symp}]
Let $q_j = \pi_j(q)$ for all $j$,
and let
$X_j = \rho_{j-1}^{-1}(q_{j-1}) \subset M_j$ be the fiber over $q_{j-1}$.
Note that the value of $\Xi(\gamma)$ doesn't change if
we multiply $\omega_j + \psi_j$ by a non-zero constant 
or add any constant to it.
Moreover, by Lemma~\ref{fiberonto} 
the restriction map from $H_T^2(M_j;\Z)$ to  $H^2(X_j;\Z)$ is surjective.
Therefore, since 
$H^2(X_j;\R) = \R$,
we may assume that $[\omega_j + \psi_j]$ 
lies in\footnote{
Since the fixed set is discrete $H_T^2(M_j;\Z)$ and $H^2(X_j;\Z)$ are torsion-free.
Therefore, we can identify these groups with their images in $H_T^2(M_j;\R)$ and
$H^2(X_j;\R)$, respectively.}
$H^2_T(M_j;\Z)$
and that $[\omega_j|_{X_j}]$ generates $H^2(X_j;\Z)$.

Let $\Lh^-_{q_j}$ denote  the equivariant Euler class of the negative
normal bundle of $\psi_j^\xi|_{X_j}$ at $q_j \in X_j$, and let
$\Lambda^-_{q_j}$ denote  the equivariant Euler class of the
negative normal bundle of $\psi_j^\xi$ at $q_j \in M_j$.
By the definition of  strong symplectic fibration, 
$\Lambda^-_{q_j} = \Lh^-_{q_j} \Lambda^-_{q_{j-1}}$ for all $ j.$
Since $M_0$ is a point, this implies by induction that
\begin{equation*}
\Lambda_q^- = \prod_{j=1}^k \Lh^-_{q_j}.
\end{equation*}

Therefore, to prove the claim it is  enough to prove that
given $h \in \{1,\dots,k\}$ such that the
fiber $X_h$ satisfies
$H^*(X_h;A) \simeq H^*\big(\CP^{\frac{1}{2} \dim{X_h}};A\big)$  as rings,
$r$ and $s$ in $M_k^T$ such
that $\pi_h(s) = \pi_h(q) = q_h$, and a path $\gamma$ from $r$ to $s$ such that
$h(\gamma_i,\gamma_{i+1}) = h$ for all $i \in \{1,\dots,|\gamma|\}$,
if we define
$$\Xi_h(\gamma)=\Lh_{q_h}^- 
\prod_{i=1}^{|\gamma|} \frac
{\pb_{h}(\gamma_{i+1}) - \pb_{h}(\gamma_i)}
{ \pb_{h}(q) - \pb_{h}(\gamma_i)}
\frac{\alpha_{\gamma_i} (\gamma_{i+1})}{\Lambda_{\gamma_{i+1}} ^-} 
$$
then
\begin{itemize}
\item[(a1)]
$\Xi_h(\gamma)$ can be written as the product of positive weights in
$\ell^*$ and a constant $C$ in $A$; moreover, $C > 0$ if 
$\alpha_r(r')$ is positive
 for all  $(r,r') \in E$.
 \item[(b1)]
 If $(M_k,\omega_k,\psi_k)$ is a GKM space, then
$\Xi_h(\gamma)$ can be written as the product of distinct positive weights in 
$\Pi_q(M)$ and a constant $C$ in $A$.
Finally, if $\Theta(r,r') > 0$ for all $(r,r') \in E$, then $C > 0$;
similarly, if $\Theta(r,r') \in A^\times$ for all $(r,r') \in E$,
then $C \in A^\times$.
\end{itemize}

To prove this, first note that since $h(\gamma_i,\gamma_{i+1}) =  h$ for all $i$
and $\pi_h(s) = \pi_h(q)$,
$\pi_h(\gamma_i) \in X_h$  and 
$\pi_h(\gamma_i) \neq \pi_h(\gamma_{i+1})$ for all $i$.
So by Lemma \ref{sympb} (or Lemmas \ref{GKMlambda} and \ref{GKMb'} 
if $M_k$ is GKM)
\begin{equation}\labell{3'}
\pb_{h}^\xi(\gamma_i) < \pb_{h}^\xi(\gamma_{i+1})
 \quad \mbox{for all} \ 1 \leq i \leq |\gamma|.
\end{equation}
Hence, $\pi_h(\gamma_i) \neq \pi_h(\gamma_j)$ for all $i \neq j$
and $\pb_h^\xi(\gamma_i)  < \pb_h^\xi(s)$ for all
$1 \leq i \leq |\gamma|$.
Therefore, since
$[\omega_j|_{X_j}]$ generates $H^2(X_j;\Z)$ and
$H^*(X_h;A) \simeq H^*\big(\CP^{\frac{1}{2} \dim{X_h}};A\big)$, 
Lemma~\ref{fake}  implies that
\begin{itemize}
\item[(a2)] $\Lh_{q_h}^- 
\prod_{i=1}^{|\gamma|} \frac
{1}
{ \pb_{h}(q) - \pb_{h}(\gamma_i)}$ can be written as the product of positive weights in $\ell^*$
and a constant in $A_+$.
\item[(b2)] If $(M_k,\omega_k,\psi_k)$ is a GKM space then  $\Lh_{q_h}^- 
\prod_{i=1}^{|\gamma|} \frac
{1}
{ \pb_{h}(q) - \pb_{h}(\gamma_i)}$ can be written as the product of distinct positive weights in
$\Pi_q(M_k)$ and a unit in $A_+$.
\end{itemize}
Here, in the case that $M_k$ is a GKM space, we
use the fact that by Remark \ref{discrete tower} $M_j$ is also a GKM space for all $j$;
moreover by Lemma \ref{GKMlambda} $\rho_j$ is a weight preserving map for all $j$, hence $\pi_h$ is weight preserving as well and
$\Pi_{q_h}(X_h) \subset \Pi_{q_h}(M_h)$
is a subset
of $\Pi_q(M_k)$. 

Since $[\omega_h + \psi_h]$ is an integral class,
Lemma~\ref{int} and \eqref{3'} together imply that  for all $1 \leq i \leq |\gamma|$,
\begin{equation}\labell{tower1}
 \left(\pb_h(\gamma_{i+1})  - \pb_h(\gamma_i) \right) 
\frac{\alpha_{\gamma_i}(\gamma_{i+1})}{\Lambda_{\gamma_{i+1}}^-} 
\in 
\begin{cases}
A, & \mbox{and}\\
A_+  &
\mbox{if }\;\alpha_r(r')\;\mbox{ is positive }
\ \forall  (r,r') \in E.
\end{cases}
\end{equation}

If $M_k$ is a GKM space
then 
by Theorem~\ref{existence canonical classes},
$\frac{\alpha_r(r')}{\Lambda_{r'}^-} = \frac{\Theta(r,r')}{\eta(r,r')}$ 
for all $(r,r') \in E$.
Moreover Lemma~\ref{GT1} implies that
$\psi_k^\xi(r) < \psi_k^\xi(r')$ and so 
$\eta(r,r')$ is positive 
because it is a positive multiple of $\psi_k(r') - \psi_k(r)$.
Therefore,
if $M_k$ is GKM, then
\begin{equation}\labell{tower5}
 \alpha_{r} (r')\;
\mbox{ is positive exactly if  }\;
\Theta(r,r') > 0
\;\; \mbox{for all }(r,r') \in E.
\end{equation}
Moreover, 
$\frac{\pb_h(\gamma_{i+1})-\pb_h(\gamma_i)}{\eta(\gamma_i,\gamma_{i+1})}=
m(\pi_h(\gamma_i),\pi_h(\gamma_{i+1}))$ 
because $\pi_h$ is a weight preserving map. 
Hence
by Lemma \ref{fake} we have that
\begin{equation}\labell{units}
 \left(\pb_h(\gamma_{i+1})  - \pb_h(\gamma_i) \right) 
\frac{\alpha_{\gamma_i}(\gamma_{i+1})}{\Lambda_{\gamma_{i+1}}^-} 
\in A_+^\times
\Leftrightarrow
\Theta(\gamma_i,\gamma_{i+1})
\in A_+^\times,
\end{equation}
where $A_+^\times$ denotes the set of positive units in $A$.
The claim now follows from (a2), (b2), \eqref{tower1}, \eqref{tower5} and \eqref{units}.
\end{proof}

\subsection*{Proof of Corollary \ref{corollary formula}}

The proof uses the following lemma.
\begin{lemma}\labell{cgfsymp}
Let $(M,\omega,\psi)$
and $\big(\Mt,\ot,\pt\big)$  be Hamiltonian $T$-manifolds with discrete fixed sets,
and let $\pi \colon M \to \Mt$ be a strong symplectic fibration.
Let $\varphi = \psi^\xi$ be a generic component of the moment map.
Given $q\in M^T$, consider the fiber $\Mh_q=\pi^{-1}(\pi(q))$.
If 
$\alpha_s \in H_T^{2\lambda(s)}(M;A)$
 is a canonical class at $s \in \Mh_q^T$, then
there exists $\ah_s \in H_T^*(\Mh_q;A)$ such that
$$\Lt_{\pi(q)}^- \ah_s = \alpha_s|_{\Mh_q}.$$
\end{lemma}

\begin{proof}
Define $\vb = \pi^*(\pt)^\xi  \colon M \to \R$. 
Since $\Mt$ has a discrete fixed set,
$\pt^\xi \colon \Mt \to \R$ is a Morse function 
with critical set $\Mt^T$.
Since $\pi$ is a fiber bundle, this implies that  $\vb$ 
is an invariant Morse-Bott function on $M$ and that
the critical component  of $\vb$
that contains $q$ is the fiber $\Mh_q$. 
Moreover, the index of $\vb$ at $\Mh_q$ is 
$2\lt(\pi(q))$,
and the equivariant Euler class
of the negative normal bundle  of $\vb$ at $\Mh_q$ is $\Lt^-_{\pi(q)}$.
By the definition of
strong symplectic fibration,
Lemma~\ref{omegat} and 
Lemma~\ref{Propb} 
imply that
for any $s \in \Mh_q^T$  the restriction
of $\alpha_s$ to $\vb^{-1}\big(-\infty, \vb(q) - \delta \big)$ vanishes 
for all $\delta > 0$.
Thus, by a standard Morse theory argument, there exists $\ah_s
\in H^{2 \lambda(s) - 2\lt(\pi(q))}_T(\Mh_q; A)$ such that
$
\Lt^-_{\pi(q)} \ah_s = \alpha_s|_{\Mh_q}.
$
\end{proof}

\begin{proof}[Proof of Corollary \ref{corollary formula}]

Since $\pi$ is a strong symplectic fibration,
$\Lambda^-_s = \Lt_{\pi(s)}^- \Lh_s^- = \Lt_{\pi(q)}^- \Lh_s^-$  and
$\lambda(s) = \lt(\pi(q)) + \lh(s)$ for all $s  \in  \Mh_q^T$.
Hence,
$\lh(r) \leq \lh(s)$ exactly if
$\lambda(r) \leq \lambda(s)$ for all $r$ and $s$ in $\Mh_q^T$.

By Lemma~\ref{cgfsymp}, for all $s \in \Mh_q^T$ there exists a class
$\ah_s \in H_T^{2 \lh(s)}(\Mh_q;A)$ such that
$\Lt_{\pi(q)}^- \ah_s = \alpha_s|_{\Mh_q}.$
Since $\alpha_s \in H_T^{2 \lambda(s)}(M; A)$ is a canonical 
class, 
the paragraph above implies that
$\ah_s$ is a canonical class at $s$ on $\Mh_q$ with respect to
the restriction $\varphi|_{\Mh_q}$.
This proves the first claim.
Moreover,
applying Theorem \ref{thm:main} (and Remark \ref{rmk:symplectic}) to  
$\Mh_q$, we
have 
\begin{equation}\labell{alphafiber}
\widehat{\alpha}_s(q)= 
\Lh_q^-
\sum_{\gamma\in \widehat{\Sigma}(s,q)}\prod_{i=1}^{|\gamma|}
\frac{\psi(\gamma_{i+1})-\psi(\gamma_i)}{\psi(q)-\psi(\gamma_i)}
\frac{\alpha_{\gamma_i}(\gamma_{i+1})}{\Lambda_{\gamma_{i+1}}^-}
\quad \mbox{for
all } s \in \Mh_q^T,
\end{equation} 
where $\widehat{\Sigma}(s,q)$ is the set of paths from $s$ to $q$ in 
the canonical graph associated to $\Mh_q$.

Now we can apply Theorem \ref{tower symp} to
$\pi \colon M \to \Mt$.
Observe that a path
$\gamma=(\gamma_1,\ldots,\gamma_{|\gamma|+1})$ from $p$ to
$q$  lies in $C(p,q)$ exactly if there exists
$j\in\{1,\ldots,|\gamma|+1\}$ such that 
$\pi(\gamma_i) \neq \pi(\gamma_{i+1})$ for all $i < j$
and $\pi(\gamma_i) = \pi(\gamma_{i+1})$ for all $i \geq j$, 
that is, so that
$(\gamma_1,\ldots,\gamma_j)$ belongs to $\overline{\Sigma}(p,\gamma_j)$,
and  $(\gamma_j,\ldots,\gamma_{|\gamma|+1})$ belongs 
to $
\widehat{\Sigma}(\gamma_j,q)$.
Hence, since $\Lambda_q^- = \Lh_q^- \Lt_{\pi(q)}^-$,
the second claim follows immediately from \eqref{alphafiber}
and Theorem \ref{tower symp}.

Finally, the third claim follows from (a1) and (b1).

\end{proof}

\begin{remark}\rm\labell{fibers are GKM}
Claim 1.\ of Corollary~\ref{corollary formula} is much easier to prove
for GKM spaces.  To see this,
let $(M,\omega,\psi)$ and $\big(\Mt,\ot,\pt\big)$ be  GKM spaces, and
let $\pi \colon M \to \Mt$ be a
strong symplectic fibration.
Let $\varphi = \psi^\xi$ be a generic component of the moment map.
Assume that 
canonical classes $\alpha_p \in H_T^{2 \lambda(p)}(M;A)$ exist for
all $p \in M^T$.
By Remark 4.3 in \cite{GT}, this implies that $\varphi$ is index increasing on $M$. 
Since $\pi$ is a strong symplectic fibration, 
$(\Mh_q,\omega|_{\Mh_q},\psi|_{\Mh_q})$ is a GKM space for all $q\in M^T$, and its GKM graph is just the restriction of the GKM graph of $M$
to $\Mh_q^T$. 
Moreover, $\lambda(s)-\lambda(r)=\widehat{\lambda}(s)-\widehat{\lambda}(r)$ 
for all $r,s\in \Mh_q^T$; so $\varphi|_{\Mh_q}$ is also index increasing. Therefore the claim follows
from Theorem \ref{existence canonical classes}. 
\end{remark}

In our final lemma,
we show how to express the polynomials $P(\gamma)$ appearing 
in Corollary \ref{corollary formula} in terms of the magnitudes
of the edges of the GKM graph associated to the 
base;  see Definition \ref{magnitude}.

\begin{definition}\labell{def:sv}
Fix a GKM space
$(\Mt,\ot,\pt)$ 
with GKM graph $(\Vt,\gkmet)$,
and let  
$\pt^\xi$ be a generic component of the moment map.
Given
an ascending path 
$\gt=(\gt_1,\ldots,\gt_{|\gt|+1})$, 
the set of {\bf skipped vertices} of $\gt$ is defined to be
 $$SV(\gt)=\left\{r\in \widetilde{V} 
\; \left| \; \pt^{\xi}(r)<\pt^{\xi}\big(\gt_{|\gt|+1}\big)  \right. \right\} \smallsetminus \big\{ \gt_1,\dots,\gt_{|\gt| + 1} \big\}.$$ 
\end{definition}

\begin{lemma}\labell{explicit P}
Let $(M,\omega,\psi)$ and $\big(\Mt,\ot,\pt\big)$ be GKM spaces with GKM graphs $(V,\gkme)$ and $(\Vt,\gkmet)$ 
and let $\pi\colon M \r \Mt$ be a 
strong symplectic fibration.
Let $\varphi=\psi^{\xi}$ be a generic component of the moment map. 
Assume that $\varphi$ is index increasing. 
Also assume that $H^*(\Mt;A)\simeq H^*\big(\CP^{\frac{1}{2}\dim (\Mt)};A\big)$. 

Given $p$ and $s \in M^T$ and a 
horizontal path
$\gamma=(\gamma_1,\ldots,\gamma_{|\gamma|+1})$ from $p$ to $s$ in the
canonical graph $(V,E)$, define 
$$
P(\gamma)=
\Lt^-_{\pi(s)} \prod_{i=1}^{|\gamma|}
\frac
 {\pt(\pi(\gamma_{i+1}))-\pt(\pi(\gamma_i))}
 {\pt(\pi(s))-\pt(\pi(\gamma_i))}
 \frac{\alpha_{\gamma_i} (\gamma_{i+1})}{\Lambda_{\gamma_{i+1}} ^-}.
$$
Then
\begin{gather*}
P(\gamma)=
 \prod_{i=1}^{|\gamma|}\frac{m(\pi(\gamma_i),\pi(\gamma_{i+1}))\Theta(\gamma_i,\gamma_{i+1})}
 {m(\pi(\gamma_i),\pi(s))}\prod_{r\in SV(\pi(\gamma))}\eta(r,\pi(s)). 
\end{gather*}
\end{lemma}

\begin{proof}
Observe that by Theorem \ref{existence canonical classes}, canonical classes
$\alpha_p$ exist for all $p\in M^T$ and $(V,E)\subset (V,\gkme)$.
By Lemma~\ref{GKMlambda}, $\pi$ is  weight preserving; hence
$\eta(\gamma_i,\gamma_{i+1})=\eta(\pi(\gamma_i),\pi(\gamma_{i+1}))$ for all
$i$, and
by the definition of magnitude 
$\pt(\pi(\gamma_{i+1}))-\pt(\pi(\gamma_i)) = 
\eta(\gamma_i,\gamma_{i+1}) m(\pi(\gamma_i),\pi(\gamma_{i+1}))$.
Since $H^*(\Mt;A)\simeq H^*\big(\CP^{\frac{1}{2}\dim (\Mt)};A\big)$,
$(\Vt,\gkmet)$
is a complete graph (see the proof of Lemma \ref{fake}),
and  so
$\pt(\pi(s))-\pt(\pi(\gamma_i))
= \eta(\pi(\gamma_{i}),\pi(s))m(\pi(\gamma_{i}),\pi(s))$ for all $i\leq |\gamma|$.
Moreover, by Lemma \ref{GT1}, $\gamma$ is an ascending path; hence
Lemma~\ref{increasing path} implies that $\pi(\gamma)$ is ascending as well
(with respect to $\pt^\xi$), and so
$\displaystyle\prod_{i=1}^{|\gamma|}\frac{\widetilde{\Lambda}_{\pi(s)}^-}{\eta(\pi(\gamma_i),\pi(s))}=\prod _{r\in SV(\pi(\gamma))}\eta(r,\pi(s))$.
Finally observe that by
Theorem \ref{existence canonical classes} $\displaystyle\frac{\alpha_{\gamma_i}(\gamma_{i+1})}{\Lambda_{\gamma_{i+1}}^-}=
\frac{\Theta(\gamma_i,\gamma_{i+1})}{\eta(\gamma_i,\gamma_{i+1})}$ for all $i$.
\end{proof}

\section{Positive integral formulas for Schubert classes}

\labell{sec:co}

We are now ready to apply our results
to the important special case of coadjoint orbits.
Our main goal is to get positive integral formulas for 
equivariant
Schubert classes on generic 
coadjoint orbits of type $A_n$, $B_n$, $C_n$, and $D_n$.

Let $G$ be a compact simple Lie group with Lie algebra $\mathfrak{g}$, and 
let $( \cdot, \cdot)$ denote the natural pairing between $\g^*$ and $\g$.
Let $T\subset G$ 
be a maximal torus with Lie algebra $\ft$,  $R \subset \td$ denote
the set of roots, and  $W$  the Weyl group of $G$.
Let $\langle \cdot, \cdot \rangle$ be a positive definite symmetric bilinear form on $\g$ which
is $G$-invariant; we use it to
embed $\td$ in $\g^*$.

Given a  point $p_0 \in \ft^*$, consider the coadjoint orbit
$\Oo_{p_0} = G \cdot p_0.$  Let $P_{p_0} \subset G$ be the stabilizer of $p_0$;
the map which takes $g \in G$ to $g \cdot p_0 \in \Oo_{p_0}$
induces an identification $\Oo_{p_0} = G/P_{p_0}$.
There
is natural 
$G$-invariant complex structure  $J$ and
a compatible symplectic form $\omega$ (the Kostant-Kirillov form) on  
$\Oo_{p_0}$;
the moment map is the inclusion map $\Oo_{p_0} \hookrightarrow \g^*$.
Hence, the moment map $\psi \colon \Oo_{p_0} \to \ft^*$ for the $T$
action is the composition of this inclusion with the natural projection
from $\g^*$ to $\ft^*$.
Moreover,
$(\mathcal{O}_{p_0},\omega,\psi)$ is a GKM space. (See \cite{GHZ}.)
Finally, 
we will say that $\Oo_{p_0}$ is {\bf generic} if $p_0 \in \ft^*$
lies in the interior of a Weyl chamber.

\subsection{The canonical graph of a generic coadjoint orbit}
\labell{cangraph}

Fix a generic coadjoint orbit $\Oo_{p_0}$.
As we will see, in this special case, the canonical classes exist
and are exactly the equivariant Schubert classes.
The main goal of this subsection is to give 
an explicit description of the associated canonical graph,
including all labels.

\begin{proposition}\labell{cocan}
Let the maximal torus $T$ of a compact simple Lie group $G$ 
act on a generic coadjoint orbit $\Oo_{p_0} \subset \g^*$ with moment map 
$\psi \colon \Oo_{p_0}  \to \ft^*$.
Let $\varphi = \psi^\xi$ 
be a generic component of 
the moment map that achieves its minimal value at $p_0 \in \ft^*$, and
let 
$R^+=\{\alpha\in R\mid (\alpha,\xi)>0\}$.
There exist canonical classes 
$\alpha_p\in H_T^{2\lambda(p)}(\Oo_{p_0} ;\Z)$ for all
$p\in \Oo_{p_0}^T$. 
Under the identification of the Weyl group $W$ with $\Oo_{p_0}^T$ given by
$w \mapsto w(p_0)$,
the canonical graph  is $(W,E)$, where 
\begin{gather*}
\psi(w) = w(p_0) \quad \mbox{for all } w \in W, \\
E = \big\{ (w, w s_\beta) \in  W \times W \; \big| \;
l(w s_\beta) = l(w) +1 \mbox{ and }  \beta \in R \big\}, 
\quad \mbox{and} \\
\frac{\alpha_{w(p_0)}\big(w'(p_0)\big)}{\Lambda_{w'(p_0)}^{-}}
                    = 
\frac{1}{w(\beta)} 
\quad
 \mbox{for all } (w, w') \in E,
\mbox{where } w' = w s_\beta \mbox{ and } \beta \in R^+. 
\end{gather*}
\end{proposition}

We begin by describing the GKM graph.

\begin{lemma}\labell{coGKM}
Let the maximal torus $T$ of a compact simple Lie group $G$ 
act on a generic coadjoint orbit $\Oo_{p_0} \subset \g^*$ with moment map 
$\psi \colon \Oo_{p_0}  \to \ft^*$.
Let $\varphi = \psi^\xi$ 
be a generic component of 
the moment map that achieves its minimal value at $p_0 \in \ft^*$, and
let 
$R^+=\{\alpha\in R\mid (\alpha,\xi)>0\}$.
Under the identification of the Weyl group $W$ with $\Oo_{p_0}^T$ given by
$w \mapsto w(p_0)$,
the  GKM graph is $(W,\gkme)$, where
\begin{gather*}
\psi(w) = w(p_0) \quad \mbox{for all } w \in W, \\
\gkme = \big\{ (w, w s_\beta) \in  W \times W \; \big| \;
\beta \in R \big\}, \mbox{and} \\
\eta(w, w') = w(\beta) = - w'(\beta)
\ 
 \mbox{for all } (w, w') \in E,
\mbox{where } w' = w s_\beta \mbox{ and } \beta \in R^+. 
\end{gather*}
\end{lemma}

\begin{proof}
As proved in \cite{GHZ}, 
the  GKM graph $(V,\gkme)$ of the coadjoint orbit $\Oo_{p_0}$
can be described as follows:
\begin{itemize}
\item The map from the Weyl group $W$ to $\ft^*$ which takes $w$ to $w(p_0)$ induces a
bijection between the elements of the Weyl group and the vertices 
$V=\Oo_{p_0}^T \subset \td \subset \g^* $.
The restriction of the moment map $\psi$ to $V$ is the inclusion map, that is, $\psi(p) = p$ for all $p \in V$.
\item There exists an edge $e \in \gkme$ between two vertices $p_1=w_1(p_0)$ and $p_2=w_2(p_0)$  
if and only if $w_2=s_{\alpha}w_1$, where $s_{\alpha}$ is the reflection associated to
 some $\alpha \in R$. 
In this case, the weight $\eta(p_1, p_2)$ is the  unique $\alpha \in R$ such that
$w_2 = s_{\alpha} w_1$  and $\langle  p_2, \alpha \rangle > 0$.
\end{itemize}
In particular, the set of weights of the isotropy representation
on 
$(T_p \Oo_{p_0},\omega)$
is
\begin{equation}\labell{all}
\Pi_{p}(\Oo_{p_0})=  
 \{ \alpha \in R \mid \langle p, \alpha \rangle > 0 \}
\quad \mbox{for all } p \in V.
\end{equation}
Since $p_0$ is the minimum, $ (\alpha,\xi) < 0$ 
for every weight $\alpha \in \Pi_{p_0}(\Oo_{p_0})$.
By \eqref{all}, this implies that 
\begin{equation}\labell{minimum}
\langle p_0,\alpha \rangle < 0 \quad \mbox{for all } \alpha \in R^+.
\end{equation}
Moreover, it is easy to check that   
\begin{equation}\labell{switch}
s_{w(\beta)}w=ws_{\beta}
\quad \mbox{for all }w \in W \mbox{ and }\beta \in R.
\end{equation}
Since the Weyl group takes $R$ to itself, this implies that
there exists an edge $e \in \gkme$ between two vertices 
$p_1=w_1(p_0)$ and $p_2=w_2(p_0)$  
if and only if $w_2=w_1 s_\beta$ 
for some $\beta \in R^+$. 
In this case,  
since $\langle \cdot , \cdot  \rangle$ is $G$-invariant, 
\eqref{minimum} implies that 
$\langle p_2, w_2 (\beta) \rangle  = \langle w_2(p_0),  w_2 (\beta) \rangle =
\langle p_0, \beta \rangle < 0.$
Therefore,  $\eta(p_1, p_2) =  - w_2(\beta) = 
- w_1 s_\beta (\beta) =  w_1(\beta)$.
\end{proof}

In Figure \ref{GKM graphB2} we draw the GKM graph of a generic coadjoint orbit
of $SO(5)$ through the point $p_0=-2x_1-x_2$ 
(here every pair of edges $(p,q)$ and $(q,p)$ is represented by one single edge),
together with the canonical graph associated to a component of the moment
map which achieves its minimum at $p_0$.


\begin{figure}[h!]
\begin{center}
\epsfxsize=\textwidth
\leavevmode
\psfrag{m}{$p_0$}
\includegraphics[height=2in]{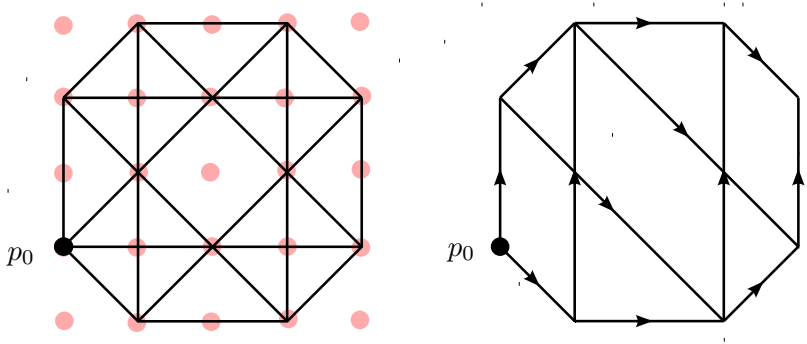}

\end{center}
\caption{The GKM graph and the canonical graph of a generic coadjoint orbit of $SO(5)$}
\label{GKM graphB2}
\end{figure}

We will need the following standard facts about root systems \cite{Hum}.
Given a set of positive roots $R^+$, let $R_0 \subset R^+$ be the associated simple roots.
Every  element $w$ of the Weyl group $W$ can be written as a product of simple 
reflections, i.e. $w=s_1\cdots s_r$, where $s_i=s_{\alpha_i}$ and $\alpha_i\in R_0$ 
for all $i=1,\ldots,r$ 
\cite[\S 1.5]{Hum}.
The {\bf length} of $w$, denoted $l(w)$, is the smallest $r$ for which such an expression 
exists. 
Any such  expression with $r = l(w)$ is a {\bf reduced expression}
for $w$.
\begin{enumerate}
\item 
[1.]
Given $w \in W$ and $\beta \in R^+$,
$l(w s_{\beta}  ) > l(w)$ exactly if $w(\beta) \in R^+$ 
 \cite[\S 5.7]{Hum}.
\item 
[2.]
If  $w=s_1\cdots s_r$ is a reduced expression for $w \in W$, 
where $s_i=s_{\alpha_i}$ for some $\alpha_i\in R_0$ for all $i$, 
then (see \cite[page 14] {Hum})
$$
 R^+ \cap w^{-1}(-R^+) =
\{\beta_1,\ldots,\beta_r\} \quad \mbox{ where }
\beta_i=s_r\cdots s_{i+1}(\alpha_i).
$$
Moreover, the $\beta_i$ are distinct.
\end{enumerate}

\begin{lemma}\labell{llambda}
Let the maximal torus $T$ of a compact simple Lie group $G$ 
act on a generic coadjoint orbit $\Oo_{p_0} \subset \g^*$ with moment map 
$\psi \colon \Oo_{p_0}  \to \ft^*$.
Let $\varphi = \psi^\xi$ 
be a generic component of 
the moment map that achieves its minimal value at $p_0 \in \ft^*$, and
let 
$R^+=\{\alpha\in R\mid (\alpha,\xi)>0\}$.
Then for any $w \in W$ with reduced expression  $w = s_1 \dots s_r$, 
$$
\Pi_{w(p_0)}^- (\Oo_{p_0}) 
=
\{\eta_1,\ldots,\eta_r\} \quad \mbox{ where }
\eta_i=s_1\cdots s_{i-1}(\alpha_i);  $$
moreover, the $\eta_i$ are distinct.  Therefore,
$$
\lambda(p) = l(w)\quad \mbox{for all } p=w(p_0)\in V.
$$
\end{lemma}

\begin{proof}
By Lemma~\ref{coGKM},
$
\Pi_{w(p_0)} (\Oo_{p_0})
= w(-R^+)
$.
Thus, since the set of weights $\Pi_p^-(\Oo_{p_0})$
in the negative normal bundle at $p$ is the set of positive weights in
the representation $(T_p\Oo_{p_0},\omega)$,
\begin{equation}\labell{positive}
\Pi_p^- (\Oo_{p_0})= 
 R^+ \cap w(-R^+) = - w \left(R^+ \cap  w^{-1}(-R^+)\right)\quad\mbox{for all }p=w(p_0)\in V.
\end{equation}
Hence, the first claim follows directly from Fact 2.\ above.
Since $\lambda(p)=|\Pi_p^-(\Oo_{p_0})|$,
the next claim is immediate.
\end{proof}

By Theorem~\ref{existence canonical classes}, the next lemma demonstrates
that canonical classes exist on $\Oo_{p_0}$, thus proving the
first claim of Proposition~\ref{cocan}.
Recall that $\varphi$ is {\em index increasing} exactly if
 $\lambda(p)<\lambda(q)$ for every $(p,q)\in \gkme$ such that $\varphi(p)<\varphi(q)$.

\begin{lemma}\labell{flag index increasing}
Let the maximal torus $T$ of a compact simple Lie group $G$ 
act on a generic coadjoint orbit $\Oo_{p_0} \subset \g^*$ with moment map $\psi \colon \Oo_{p_0}  \to \ft^*$.
Then each  generic component of the moment map,
$\varphi = \psi^\xi $, 
 is index increasing.
\end{lemma}

\begin{proof}
Assume that $\varphi$ achieves its minimum value at $p_0\in \td$, and define
$R^+=\{\alpha\in R\mid (\alpha,\xi)>0\}$.
Consider an edge $(p_1, p_2) = (w_1(p_0), w_2(p_0)) \in \gkme$ so that $\varphi(p_2)
> \varphi(p_1)$.
By Lemma~\ref{coGKM}, there
exists $\beta \in R^+$
so that 
$w_2 = w_1 s_\beta$ and $\eta(p_1,p_2) = w_1(\beta)$.
Since  $\psi(p_2) -\psi(p_1)$ 
is a positive multiple of $\eta(p_1, p_2)$, the fact
that $\varphi(p_2) > \varphi(p_1)$ implies that 
$ w_1(\beta)\in R^+$. 
By Fact 1.\ above, this implies that $l(w_2) =
l( w_1 s_\beta) > l(w_1)$.
Therefore, Lemma~\ref{llambda} implies that $\lambda(p_2) > \lambda(p_1)$,
as required.
\end{proof}

Given  a choice of positive roots $R^+$,
$\Oo_{p_0}$ can be identified as a $T$-space with the flag variety $G_{\C}/B$, 
where $G_{\C}$ is the complexification of $G$
and $B$ is the Borel subgroup associated to $R^+$.
In the Schubert calculus literature, there is a well-know basis for $H_T^*(G_{\C}/B;\Z)$, whose
elements are called \emph{equivariant Schubert classes}. 

Define 
\begin{equation}\labell{def Lambda}
\Lambda_w^-=\prod\{\eta\in R^+\mid w^{-1}(\eta)\in -R^+\}.
\end{equation}
For every $w \in W$
there exists a unique element $K_w\in H_T^{2 l(w)}(G_{\C}/B;\Z)$ satisfying the following conditions:
\begin{itemize}
\item[$(1)'$] $K_w(w')=0$ for all $w'\in W\setminus\{w\}$ such that $l(w')\leq l(w)$.
\item[$(2)'$] $K_w(w)=\Lambda_w^-.$
\end{itemize}
Moreover the set $\{K_w\}_{w\in W}$ is a basis for 
$H_T^*(G_{\C}/B;\Z)$ as a module over $H^*(BT;\Z)$ (see \cite{Ku}).

\begin{proposition}\labell{canonical=Schubert}
Let the maximal torus $T$ of a compact simple Lie group $G$ 
act on a generic coadjoint orbit $\Oo_{p_0} \subset \g^*$ with moment map 
$\psi \colon \Oo_{p_0}  \to \ft^*$.
Let $\varphi = \psi^\xi$ 
be a generic component of 
the moment map that achieves its minimal value at 
$p_0 \in \ft^*$, 
and let $\{\alpha_p \}_{p\in \Oo_{p_0}^T}$ be the  
canonical classes.
 Let $G_{\C}$ be the complexification of $G$, $B$ the Borel subgroup associated to 
$R^+=\{\alpha\in R\mid (\alpha,\xi)>0\},$ 
and $\{K_w\}_{w\in W}$ be the set of equivariant Schubert classes. 
The canonical classes are the equivariant 
Schubert classes, i.e., $$\alpha_{w(p_0)}(w'(p_0)) =K_w(w') \quad\mbox {for all } w,w' \in W.$$

\end{proposition}
\begin{proof}
Since, by \cite[Lemma 2.7]{GT}, canonical classes are characterized by properties $(1)$ and $(2)$ in Definition \ref{definitioncc}, it
is sufficient to prove that the equivariant Schubert classes also satisfy these properties. 
But this follows immediately by observing that 
Lemma~\ref{coGKM}, Lemma~\ref{llambda}, and \eqref{def Lambda} 
imply that $(1)$ is equivalent to $(1)'$ and $(2)$ to $(2)'$. 
\end{proof}
The
explicit description of the canonical graph  given in
Proposition~\ref{cocan} follows immediately from
Theorem \ref{existence canonical classes},
Lemmas~\ref{coGKM} and \ref{llambda},
and the following proposition,
which 
describes
the integers $\Theta(p,q)$ that appear in
Theorem \ref{existence canonical classes}. 

\begin{proposition}\labell{theta one1}
Let the maximal torus $T$ of a compact simple Lie group $G$ 
act on a generic coadjoint orbit $\Oo_{p_0} \subset \g^*$ with moment map $\psi \colon \Oo_{p_0}  \to \ft^*$.
Let $\varphi = \psi^\xi $ 
be
a generic component of the moment map.
Then
$$
\Theta(p,q)= 1
\quad \mbox{for all }(p,q)\in \gkme \mbox{ with }\lambda(q)-
\lambda(p)=1. 
$$
\end{proposition}

\begin{proof}
Assume that $\varphi$ achieves its minimum value at $p_0\in \td$, and let
$R^+=\{\alpha\in R\mid (\alpha,\xi)>0\}$.
Consider an edge $(p, q) = (w(p_0), w'(p_0)) \in \gkme$ such that
$\lambda(q)-\lambda(p)=1$.
By Lemma~\ref{coGKM}, there
exists $\beta \in R^+$
so that 
$w' = w s_\beta$ and $\eta(p,q) = w (\beta)$.
By \eqref{switch}, we can also write $w' = s_\alpha w$, where 
$\alpha = w(\beta) = \eta(p,q)$.

Let $\Pi_p^{-}(\Oo_{p_0})$ and $\Pi_q^{-}(\Oo_{p_0})$ denote the set of weights in the negative normal bundle of $\varphi$ at $p$ and $q$, 
respectively. In order to prove that 
$\Theta(p,q)=1$, it is sufficient to find a bijection 
$f \colon  \Pi_p^{-}(\Oo_{p_0})\rightarrow \Pi_q^{-}(\Oo_{p_0})\smallsetminus \{\alpha\}$ 
such that 
for each 
$\eta \in \Pi_p^{-}(\Oo_{p_0})$,
$\eta = f(\eta) \mod \alpha$, i.e. there exists an integer $n$ such that 
$f(\eta)-\eta=n \alpha$.

Let  $w'=s_1s_2\cdots s_r$ be a reduced expression for $w'$, 
where $s_i=s_{\alpha_{i}}$ for some $\alpha_{i}\in R_0$ for all $i$.
By Lemma \ref{llambda}
$l(w')=l(w)+1>l(w)$.  Therefore,
 by the Strong Exchange Condition (see \cite[Section 5.8]{Hum}) 
$w=s_1\cdots\widehat{s_j}\cdots s_r$ for some (unique) $j$, 
where $\widehat{s_j}$ indicates that we are omitting the $j$'th term.
Let $\widetilde{w}=s_1s_2\cdots s_{j-1}$. Then by \eqref{switch} we have 
that for all $j\leq k\leq r$
$$
s_1s_2\cdots s_k = \widetilde{w}s_js_{j+1}\cdots s_k=
s_{\widetilde{w}(\alpha_j)}\widetilde{w}s_{j+1}\cdots s_k=s_{\widetilde{w}(\alpha_j)}
s_1 s_2\cdots \widehat{s_j} \cdots s_k
$$ 
In particular, $w'  = s_{\widetilde{w}(\alpha_j)} w$, and so 
$s_{\widetilde{w}(\alpha_j)}
= s_{\alpha}$.
Hence,
\begin{equation*}
s_1s_2\cdots s_k(\alpha_{k+1}) \equiv
s_1 s_2\cdots \widehat{s_j} \cdots s_k(\alpha_{k+1}) \mod \alpha \quad \mbox{for all } j \leq k < r.
\end{equation*}
Moreover, the fact that  $l(w') > l(w)$ implies that $\alpha > 0$.
Therefore,  Lemma~\ref{llambda} implies that 
$\widetilde{w}(\alpha_j) = \alpha$, 
\begin{gather*}
\Pi_q^{-}\smallsetminus \{\alpha\}=\{\alpha_1
,\ldots,s_1\cdots 
s_{j-2}(\alpha_{j-1}),s_1\cdots s_j(\alpha_{j+1}),\ldots,s_1\cdots s_{r-1}(\alpha_r)\},
\end{gather*}
and
\begin{equation*}
\Pi_p^{-}=\{\alpha_1,
\ldots,s_1\cdots s_{j-2}(\alpha_{j-1}),
s_1\cdots s_{j-1}(\alpha_{j+1}),\ldots,s_1\cdots\widehat{s_j}\cdots s_{r-1}(\alpha_r)\}.
\end{equation*}
The claim follows immediately.
\end{proof}

\begin{remark} \rm
Let $w_1$ and $w_2$ be two elements of the Weyl group $W$ such that 
$l(w_1) < l(w_2)$ and $w_2= w_1 s_{\beta}$, for some $\beta \in R$;
in this case, we 
will
write $w_1 \r w_2$. 
The \textit{Bruhat order} is the transitive closure of this order,
i.e., 
$w\prec w'$ in the Bruhat order if there exists a sequence of elements of the Weyl group 
$w_0,w_1,\dots,w_m$ such that $w_0=w,\;w_m=w'$ and $w_i\r w_{i+1}$ for all 
$i=0,\ldots,m-1$. By
Lemma~\ref{coGKM}, \ref{llambda} and \ref{flag index increasing},  
$w \prec w'$ exactly if there exists
an ascending path from $w$ to $w'$ in $(V,\gkme)$. 
\end{remark}

\subsection{Maps between coadjoint orbits}

Before turning to consider individual Lie groups, we need to establish a few facts about 
the maps between different coadjoint orbits.

Consider
two points $p_0$ and  $\widetilde{p}_0\in \td$ such that 
$P_{\widetilde{p}_0}\supset P_{p_0}$, where $P_{p_0}$ and
$P_{\widetilde{p}_0}$
are the stabilizers of $p_0$ and $\widetilde{p}_0$, respectively.
Let $\Oo_{p_0}$ and  $\mathcal{O}_{\widetilde{p}_0}$ be 
the coadjoint orbits through $p_0$ and  $\widetilde{p}_0$,  respectively,
and
let $(V,\gkme)$ and $(\widetilde{V},\gkmet)$ be the GKM graphs associated 
to $\mathcal{O}_{p_0}$ and $\mathcal{O}_{\widetilde{p}_0}$, 
respectively. 
Since
$\Oo_{p_0} = G/P_{p_0}$ and 
$\mathcal{O}_{\widetilde{p}_0}\simeq G/P_{\widetilde{p}_0}$, 
there is a natural projection map 
 \begin{equation*}
  \begin{array}{rccc}
\pi \colon  & \mathcal{O}_{p_0} & \r & \mathcal{O}_{\widetilde{p}_0}\\
     & g\cdot p_0 & \mapsto & g\cdot \widetilde{p}_0\,. \\
 \end{array}    
\end{equation*}

\begin{proposition}\labell{coadjoint preserving}
The natural projection  $\pi \colon   \mathcal{O}_{p_0}  \r  
\mathcal{O}_{\widetilde{p}_0}$ described above is a 
strong symplectic fibration.
\end{proposition}
\begin{proof}
It is well known that $\pi$ is a $T$-equivariant fiber bundle with symplectic fibers, isomorphic to $P_{\widetilde{p}_0}/P_{p_0}$. Moreover, 
we can choose
the complex structures $J$ and $\widetilde{J}$ on $\Oo_{p_0}$
and $\Oo_{\widetilde{p}_0}$  so that
$\pi$ intertwines  
them.
Hence, the claim  is a direct consequence of the discussion in 
Example \ref{example essf} $(i)$.
\end{proof}

Given any fixed point $q\in \Oo_{p_0}^T$, 
let $\widehat{\Oo}_q=\pi^{-1}(\pi(q))$ 
be the fiber over $\pi(q)$. 
It is  a GKM space; 
the associated GKM graph is just the restriction to the fiber of 
the GKM graph associated to $\Oo_{{p}_0}$.
These fibers are equivariantly symplectomorphic, but only with
respect to a non-trivial automorphism of the torus.
In \cite{GSZ}, the authors analyze 
projections of GKM spaces from  a combinatorial
point of view, and describe how the GKM structure of different fibers changes.
As our next result shows,  this is particularly 
well behaved
when $\Oo_{p_0}$ is a generic coadjoint orbit.

\begin{proposition}\labell{twist}
Assume that $\Oo_{p_0}$ is a generic coadjoint orbit.
Let $\varphi = \psi^\xi$ be a generic component of the moment map that achieves
its minimal value at $p_0 \in \ft^*$.
Given any $s \in \Oo_{p_0}^T$,
let $\widehat{\alpha}_s\in  H_T^{2\widehat{\lambda}(s)}(\widehat{\Oo}_{s};\Z)$
be the canonical class\footnote{
These  classes exist by  Corollary \ref{corollary formula} and 
Proposition \ref{coadjoint preserving}.
Here, $2\widehat{\lambda}(s)$ is the index of $\varphi|_{\widehat{\Oo}_s}$ 
at $s$. }
 at $s$ with respect to $\varphi|_{\widehat{\Oo}_s}$, 
regarded as a map from $\widehat{\Oo}_{s}^T$
to $H^*(BT;\Z)$.
Let  $\tau \in W$ be an element of the Weyl group such that 
$\tau(p_0)$ is the point in $\widehat{\Oo}_{\tau(p_0)}$  at which 
$\varphi|_{\widehat{\Oo}_{\tau(p_0)}}$ achieves its minimum value.
Then
$$ \widehat{\alpha}_{\tau(r)}(\tau(s))  =  
 \tau \left( \widehat{\alpha}_{r}(s) \right)
\quad \mbox{for all } r, s \in \widehat{\Oo}_{p_0}^T.$$
Here, by a slight abuse of notation,
$\tau \colon H^*(BT;\Z) \to H^*(BT;\Z)$ is the map
induced by $\tau \colon \ft^* \to \ft^*$ under the identification  
$H^*(BT;\R) \simeq \Sym(\ft^*)$.
\end{proposition}

\begin{proof}
To begin, consider any element $g$ in $N(T)$, the normalizer of $T$.
With respect to the automorphism of $T$
given by $t \mapsto g t g^{-1}$, the maps
$f \colon \Oo_{p_0} \to \Oo_{p_0}$
and $\widetilde{f} \colon 
\Oo_{\widetilde{p}_0} \to \Oo_{\widetilde{p}_0}$
given by left multiplication by $g$ are  equivariant symplectomorphisms;
moreover, $\widetilde{f} \circ \pi = \pi \circ f$.
Hence, $f$ induces an equivariant symplectomorphism from
$\widehat{\Oo}_{p_0}$ to $\widehat{\Oo}_{g \cdot p_0}$.

So assume that  $g$ represents $\tau \in W = N(T)/T$, 
that is, 
$\tau = \Ad_{g}^* \colon \ft^* \to \ft^* $.
Since $\Ad^*_{g^{-1}}$ is the transpose of $\Ad_g$, the
homomorphism  $t \mapsto g t g^{-1}$ 
and the linear transformation $\Ad_{g^{-1}}^* = \tau^{-1}$
induce the same automorphism of 
$H^*(BT;\Z) \subset  \Sym(\ft^*)$.
Hence, if we fix any $r \in \Oo_{p_0}^T$,
the map $\widehat{\Oo}_{p_0}^T \to H^*(BT;\Z)$
defined by  $s \mapsto \tau^{-1} \left( \widehat{\alpha}_{ \tau(r)}(\tau(s)) \right)$
is  an equivariant  cohomology class on $\widehat{\Oo}_{p_0}$.
In fact,  this class is the canonical class on $\widehat{\Oo}_{p_0}$
at $r$ with respect to
the restriction of 
$\psi^{\xi'}$, where $\xi' = Ad_{g^{-1}}(\xi)$.

Finally, 
since $\Oo_{p_0}$ is a generic coadjoint orbit,
the set of weights $\Pi_s(\Oo_{p_0})$ of the isotropy representation on the tangent
space at the fixed point $s \in \Oo_{p_0}$ agrees with $\Pi_{p_0}(\Oo_{p_0})$ up to
sign, that is,
 $\Pi_s(\Oo_{p_0}) \cup -\Pi_s(\Oo_{p_0}) =
 \Pi_{p_0}(\Oo_{p_0}) \cup -\Pi_{p_0}(\Oo_{p_0})$. 
Since $\pi$ is a strong symplectic fibration, 
this implies that -- up to sign --  $\Pi_s(\widehat{\Oo}_{p_0})$ agrees with
$\Pi_{p_0}(\widehat{\Oo}_{p_0})$  for all fixed points $s \in \widehat{\Oo}_{p_0}$.
Because $p_0$ is the point in $\widehat{\Oo}_{p_0}$
where both  $\psi^\xi$ and $\psi^{\xi'}$ 
achieve their minimum value,
this implies that 
$(\alpha,\xi)>0$
exactly if 
 $(\alpha,\xi')>0$
 for every weight $\alpha \in \Pi_s(\widehat{\Oo}_{p_0})$
and  each fixed point  $s \in \widehat{\Oo}_{p_0}$.
Hence, by \cite[Remark 2.4]{GT} the canonical classes 
on $\widehat{\Oo}_{p_0}$
with respect to the restriction of $\psi^{\xi'}$
are exactly 
the canonical classes with respect to the restriction of $\psi^\xi$.

\end{proof}

Finally, given
the projection $\pi\colon  \mathcal{O}_{p_0}  \r  
\mathcal{O}_{\widetilde{p}_0}$,
 we can describe 
explicitly  how to lift  ascending paths in $(\Vt,\gkmet)$.

\begin{lemma} 
\labell{liftpath}
Let $\widetilde{\varphi} \colon \Oo_{\widetilde{p}_0} \to \R$ 
be a generic component of the moment map.
Given $p \in V$ and an ascending path $\gt$ in $(\Vt,\gkmet)$ that
begins at $\pi(p)$, there exists  
a unique path $\gamma$ of length $|\gt|$ in $(V,\gkme)$ such that 
\begin{itemize}
\item $\gamma$ begins at $p$,
\item $V(\pi(\gamma)) = V(\gt)$, and 
\item $\lambda(\gamma_{i+1}) > \lambda(\gamma_i)$ for all $i$.
\end{itemize}
If $\gt_{i+1}=s_{\beta_i}(\gt_{i})$ for  $\beta_i \in R$ for 
each $1 \leq i \leq |\gt|$,
then the endpoint of $\gamma$ is $w(p)$, where
$$w = s_{\beta_{|\gt|}} s_{\beta_{|\gt|- 1}} \dots s_{\beta_1}.$$
\end{lemma}

\begin{proof}
Fix $p \in V$.
Since $\pi$ is an equivariant fiber bundle, 
there is a unique 
 lift
$\gamma$ of each path $\gt$ starting at 
$p$, 
that is, a unique
path $\gamma$ of length $|\gt|$ in $(V,\gkme)$ that starts at $p$
such that $\pi(\gamma_i,\gamma_{i+1})=(\gt_i,\gt_{i+1})$ for all $i$.
By Lemma \ref{increasing path} and Lemma~\ref{flag index increasing},
$\lambda(\gamma_{i+1}) > \lambda(\gamma_i)$ for all $i$ exactly if
$\gt$ is ascending; this proves the first claim.
The second claim  is an consequence of the fact
that 
$\pi \colon \mathcal{O}_{p_0} \to \mathcal{O}_{\widetilde{p}_0}$  
satisfies $\pi(w(p_0))=w(\pi(p_0))$ for all $w\in W$. 
\end{proof}


\subsection{Generic coadjoint orbits of type $A_n$}
\labell{An}
As we will see below, a generic coadjoint orbit of $SU(n)$ is a tower
of complex projective spaces 
over $\Z$.
Therefore,  Theorem~\ref{tower symp} (together with the results of \S\ref{cangraph})
immediately implies that in this case each restriction of any equivariant Schubert class 
can be expressed as a sum of terms $\Xi(\gamma)$ over certain paths $\gamma$,
where each term is the product of distinct positive roots.
In Proposition~\ref{formula an} below, we give this description explicitly;
this formula is equivalent to a particular case of the combinatorial formula given in \cite{AJS} (Appendix D.3) and \cite{B}, as proved by  
Zara in \cite{Za}.

Let $G=SU(n+1)$, and let $T\subset G$ be the subtorus of diagonal matrices.
Under the natural identification of the dual of the Lie algebra of $T$ as
$\td = 
\big\{ \mu \in (\R^{n+1})^* \; \big| \;  \sum_{i=1}^{n+1} \mu_i = 0
\big\}$,
the roots are the vectors $x_i - x_j \in \td$ for all $1\leq i\neq j\leq n$.
Here, and throughout this section,
 $\{x_i\}_{i=1}^{n+1}$ is the standard basis of $(\R ^{n+1})^{*}$. 
The Weyl group $\mathcal{S}_{n+1}$ of $G$ is the group of permutations of $n+1$ elements.

\begin{proposition}\labell{formula an}
Let $B \subset G_\C$ be the Borel subgroup associated to the positive roots 
$R^+ = \{ x_i - x_j \mid 1 \leq i < j \leq n+1\},$ where $G_\C$ is the complexification of 
$G = SU(n+1)$.
Given
$w$ and $w'$ in $\mathcal{S}_{n+1}$, let 
$K_w \in H_T^{ 2 l(w)}(G_\C/B;\Z)$ 
be the equivariant Schubert
class associated to $w$.
Let $C(w,w')$ be the set of tuples 
$\underline{\sigma} = (\sigma_1,\ldots,\sigma_{|\underline{\sigma}|+1}) 
\in \left(\mathcal{S}_{n+1}\right)^{|\underline{\sigma}|+1}$ 
such that $\sigma_1 = w$, $\sigma_{|\underline{\sigma}|+1} = w'$, and
the following properties hold
for all $1 \leq i \leq |\underline{\sigma}|$:
\begin{itemize}
\item  $l(\sigma_{i+1})=l(\sigma_i)+1$;
\item $\sigma_{i+1} = \sigma_i \, s_{x_{h_i} -x_{k_i}}$ 
for some $1 \leq h_i < k_i \leq n+1$; 
and
\item $h_i \leq h_{i+1}$.
\end{itemize}
\begin{itemize}
\item[(1)]
For all $w$ and $w'$ in 
$\mathcal{S}_{n+1}$,
\begin{gather*}
K_w(w') = \sum_{\underline{\sigma} \in C(w,w')}\Xi(\underline{\sigma}),\quad\mbox{where }\\
\Xi(\underline{\sigma})= \Lambda_{w'}^-
\, \left(\prod_{i=1}^{|\underline{\sigma}|}\frac{1}{x_{\sigma_{i}(h_i)}-
x_{\sigma_{|\underline{\sigma}|+1}(h_i)}} \right)
\quad \mbox{for all }\underline{\sigma}\in C(w,w').  \\
\end{gather*}
\item[(2)] 
For all 
 $\underline{\sigma}\in C(w,w')$,  \ 
$\Xi(\underline{\sigma})$
is the product of distinct positive roots.
\end{itemize}
\end{proposition}

 \begin{proof}
For each $0 \leq j \leq n$,
fix a point 
$$\mu^j
\in \td \quad \mbox{such that }
\mu^j_1<\cdots<\mu^j_{j}<\mu_{j+1}^j=\cdots=\mu_{n+1}^j;$$
for simplicity assume that $\mu_{j+1}^j = \mu^j_j + 1$.
Let $(\mathcal{O}_{\mu^j},\omega_j,\psi_j)$ be the coadjoint orbit through
$\mu^j$ for each $j$.
Observe that $\Oo_{\mu^0}$ is a single point and that
$\mathcal{O}_{\mu^n}$ is isomorphic to $\mathcal{F}l(\C^{n+1})$,
the variety of complete flags in $\C^{n+1}$.
The stabilizer of $\mu^j$ is
$$P_{\mu^j}=S\big(U(1)\times \ldots \times U(1)\times U(n-j+1)\big) 
\quad \mbox{for all } j;$$ 
in particular, $P_{\mu^{j+1}}  \subset P_{\mu^{j}}$.
By Proposition \ref{coadjoint preserving}, 
the natural projection map
$\rho_j\colon \Oo_{\mu^{j+1}}  \r \Oo_{\mu^j} $ 
is a strong symplectic fibration 
with fiber  $P_{\mu^{j}}/P_{\mu^{j+1}} \simeq \C P^{n-j}$
for all $0 \leq j < n$. 
So $\Oo_{\mu^n}$ is a tower of complex
projective spaces 
over $\Z$.

Each element $\sigma \in \mathcal{S}_{n+1}$
can be represented in 
one line notation by $\sigma=\sigma(1),\ldots,\sigma(n+1)$;
the action of $\sigma$
on a point 
$\mu =\sum_{i=1}^{n+1}\mu_ix_i\in \mathfrak{t}^*$ is given by
$\sigma(\mu)=\sum_{i=1}^{n+1}\mu_ix_{\sigma(i)}$.
Let $\pi_j=\rho_j\circ \rho_{j+1}\circ \cdots \circ \rho_{n-1}\colon \Oo_{\mu^n}\r \Oo_{\mu^j}$,
and define
$$
h(\sigma,\sigma')=\min\{j\in \{1,\ldots,n\}\mid \pi_j( \sigma (\mu^n))\neq \pi_j(\sigma'(\mu^n))\}
\  
\forall \ 
\sigma \neq \sigma'
\mbox{ in }\mathcal{S}_{n+1}.
$$
Fix any distinct $\sigma$ 
and $\sigma'$ 
in $\mathcal{S}_{n+1}$.
Since $\pi_j(\sigma(\mu^n))=\sigma(\mu^j)$ and
 $\mu_i^{j} = \mu_{j+1}^{j}$  exactly if $i > j$,
$\pi_j( \sigma(\mu^n))=\pi_j(\sigma'(\mu^n))$ exactly
if 
$\sigma(i)=\sigma'(i)$ for all $0\leq i \leq j$.
Hence,   
\begin{gather*}
h(\sigma, \sigma')=
\min\{j \in \{1,\dots,n\} \mid \sigma(j) \neq \sigma'(j)\};
\quad\mbox{in particular } \\
h(\sigma, \sigma s_{x_h-x_k})= h \quad \mbox{for all }
1 \leq h<k \leq n+1 .
 \end{gather*}
Let $\pb_j = \pi^*(\psi_j) \colon\Oo_{\mu^n} \to \ft^*$ for all $j$. Since $\psi_j\colon\Oo_{\mu^j}^T\r \td$ is the inclusion map,
$\pb_j(\sigma(\mu^n))= \sum_{i=1}^{n+1}\mu_i^jx_{\sigma(i)}$ for all $j$.
Since $\sum_{m=1}^{n+1} ( x_{\sigma'(m)} - x_{\sigma(m)}) = 0$ 
and
$\mu_j^j + 1 = \mu_{j+1}^j=\ldots=\mu_{n+1}^j$, this implies that 
for all $j$
\begin{gather*}
\begin{split}
\labell{andiff}
\pb_j(\sigma'(\mu^n)) - \pb_j(\sigma(\mu^n)) &
= \sum_{m=1}^{n+1} \mu^j_m (x_{\sigma'(m)}-x_{\sigma(m)}) 
\\
& = \sum_{m=1}^{n+1} (\mu^j_m - \mu^j_{j+1}) (x_{\sigma'(m)}-x_{\sigma(m)})  \\
& = \sum_{m=1}^{j} (\mu^j_m - \mu^j_{j+1}) (x_{\sigma'(m)}-x_{\sigma(m)})\,  
; \mbox{ therefore, }  
\end{split} \\
\pb_j(\sigma'(\mu^n))-\pb_j(\sigma(\mu^n))=
x_{\sigma(j)}-x_{\sigma'(j)}
\quad  \mbox{for all }
 j \leq h(\sigma,\sigma'). 
 \end{gather*}
Thus, 
for all $1 \leq h < k \leq n+1$ we have
\begin{equation*}
\pb_{h}(\sigma s_{x_h-x_k}(\mu^n))-
\pb_{h}(\sigma(\mu^n)) 
=
x_{\sigma(h)}- x_{\sigma(k)}
= \sigma(x_h - x_k). 
\end{equation*}

To conclude,
let $\varphi = \psi_n^\xi \colon \Oo_{\mu^n} \to \R$ 
be a generic component of the moment
map that achieves its minimum value at $\mu^n$.
By the 
definition of $\mu^n$,
the set $R^+$ coincides with $\{\alpha\in R\mid (\alpha,\xi)>0\}$.
So by Proposition \ref{canonical=Schubert}, canonical
classes on $\Oo_{\mu^n}$ correspond to equivariant Schubert
classes on $G_{\C}/B$ through the usual identification of $\Oo_{\mu_n}^T$ with $W$.
Both claims now follow
directly
from Theorem \ref{tower symp} and Proposition~\ref{cocan}.
(Here, we also use the fact that 
$h(\sigma_i,\sigma_{i+1}) \leq h(\sigma_i,\sigma_{|\underline{\sigma}|+1})$
for any $\underline{\sigma}=(\sigma_1,\dots,\sigma_{|\underline{\sigma}|+1})
\in C(w,w')$.)
\end{proof}


\subsection{Generic coadjoint orbits of type $C_n$}

Let $G=Sp(n)$ be the symplectic group, i.e. the quaternionic unitary group $U(n; \mathbb{H})$. 
As we will see below, a generic coadjoint orbit of $Sp(n)$ is a tower of
complex projective spaces
over $\Z$. Therefore Theorem~\ref{tower symp} (together with the results of \S\ref{cangraph})
immediately implies that in this case each restriction of any equivariant Schubert class 
can be expressed as a sum of terms $\Xi(\gamma)$ over certain paths $\gamma$,
where each term is the product of distinct positive roots.
In Proposition~\ref{formula cn2} below, we give this description explicitly, 
cf.\ \cite{Za}.

Let 
$T \subset G$ be a maximal torus.
We can identify the dual of the Lie algebra of $T$ as $\td = (\R^n)^*$;
the roots are the vectors $\pm x_i \pm x_j\in \td$ 
and $\pm 2x_i\in \td$ for all $1\leq i \neq j \leq n$.
Here, and throughout this section, $\{x_i\}_{i=1}^n$ is the standard basis of $(\R^n)^*$.
The Weyl group $W$ of $G$
is the group of signed permutations of $n$ elements.
Each element $\tau \in W$
can be represented in 
one line notation
by $\tau=(-1)^{\epsilon _1}\sigma(1),\ldots, (-1)^{\epsilon _n}\sigma(n)$, 
where $\epsilon_i\in \{0, 1\}$ for all $i$ and 
$\sigma \in \mathcal{S}_n$.

\begin{proposition}\labell{formula cn2}
Let $B\subset G_{\C}$ be the Borel subgroup associated to the positive roots 
$R^+=\{x_i\pm x_j\mid 1\leq i<j\leq n\}\cup\{2x_i\mid 1\leq i\leq n\}$, where $G_{\C}$
is the complexification of $G=Sp(n)$.
Given $w$ and $w'$ in $W$, let $K_w\in H_T^{2l(w)}(G_{\C}/B;\Z)$ be the equivariant Schubert class
associated to $w$. Let $C(w,w')$ be the set of tuples $\underline{\tau}=(\tau_1,\ldots,\tau_{|\underline{\tau}|+1})\in W^{|\underline{\tau}|+1}$
such that $\tau_1=w$, $\tau_{|\underline{\tau}|+1}=w'$ and the following properties hold for all $1\leq i\leq |\underline{\tau}|$:
\begin{itemize}
\item $l(\tau_{i+1})=l(\tau_i)+1$;
\item either $\tau_{i+1}=\tau_is_{x_{h_i}\pm x_{k_i}}$ for some $1\leq h_i<k_i\leq n$ or
$\tau_{i+1}=\tau_is_{2x_{h_i}}$ for some $1\leq h_i\leq n$; and
\item $h_i\leq h_{i+1}$.
\end{itemize}
Let $\tau_i=(-1)^{\epsilon_1^i}\sigma_i(1),\ldots,(-1)^{\epsilon_n^i}\sigma_i(n)$, where $\sigma_i\in \mathcal{S}_n$ and $\epsilon_j^i\in \{0,1\}$ for all $i$ and $j$.
\begin{itemize}
\item[(1)] For all $w$ and $w'$ in $W$,
\begin{gather*}
K_w(w') = \sum_{\underline{\tau} \in C(w,w')}
\Xi(\underline{\tau}),\quad\mbox{where }\\
\Xi(\underline{\tau})=\Lambda_{w'}^- 
\left(
\prod_{i=1}^{|\underline{\tau}|}
\frac{1}{
(-1)^{\epsilon_{h_i}^{|\underline{\tau}|+1}}
x_{
\sigma_{|\underline{\tau}|+1}
(h_i)
}-
(-1)^{\epsilon_{h_i}^{i}}
x_{\sigma_{i}(h_i)}
}
\right)
 \ 
\mbox{for all }
\underline{\tau}\in C(w,w').
\end{gather*}
\item[(2)] 
For all 
$\underline{\tau}\in C(w,w')$, \ 
$\Xi(\underline{\tau})$
is the product of distinct positive roots.
\end{itemize}
\end{proposition}
\begin{proof}
For each $0\leq j\leq n$, fix a point
$$
\mu^j \in \td\quad\mbox{such that }\mu^j_1<
\dots
<\mu^j_j<0=\mu^j_{j+1}=
\dots
=\mu^j_n;
$$
for simplicity assume that $\mu_j^j=-1$. 
Let $(\Oo_{\mu^j},\omega_j,\psi_j)$ be the coadjoint orbit through $\mu^j$ for each $j$.
Observe that $\Oo_{\mu^0}$ is a single point.
The stabilizer of $\mu^j$ is
$$
P_{\mu^j}= S^1\times\ldots\times S^1\times U(n-j;\mathbb{H})\quad\mbox{for all }j\;; 
$$
in particular, $P_{\mu^{j+1}}\subset P_{\mu^{j}}$. 
By Proposition \ref{coadjoint preserving}, 
the natural projection map
$\rho_j\colon \Oo_{\mu^{j+1}}  \r \Oo_{\mu^j} $ 
is a strong symplectic fibration with fiber  $P_{\mu^{j}}/P_{\mu^{j+1}}\simeq \CP^{2(n-j)-1}$
for all $0 \leq j < n$. So $\Oo_{\mu^n}$ is a tower of complex projective spaces 
over $\Z$.

Let $\tau=(-1)^{\epsilon _1}\sigma(1),\ldots, (-1)^{\epsilon _n}\sigma(n)$, 
where $\epsilon_i\in \{0, 1\}$ for all $i$ and 
$\sigma \in \mathcal{S}_n$; 
the action of $\tau$
on a point $\mu=\sum_{i=1}^n\mu_ix_i\in \td$ is given by
$\tau(\mu)=\sum_{i=1}^n (-1)^{\epsilon_i}\mu_ix_{\sigma(i)}$.
Let $\pi_j=\rho_j\circ \rho_{j+1}\circ \cdots \circ \rho_{n-1}\colon \Oo_{\mu^n}\r 
\Oo_{\mu^j}$,
and define
$$
h(\tau,\tau')=\min\{j\in \{1,\ldots,n\}\mid \pi_j( \tau (\mu^n))\neq \pi_j(\tau'(\mu^n))\}
\  
\forall \
\tau \neq \tau'
\mbox{ in }W.
$$
Let
$\tau=(-1)^{\epsilon_1}\sigma(1),\ldots,(-1)^{\epsilon_n}\sigma(n)$ and 
$\tau'=(-1)^{\epsilon'_1}\sigma'(1),\ldots,(-1)^{\epsilon'_n}\sigma'(n)$ 
in $W$ 
be distinct.
Since $\pi_j(\tau(\mu^n))=\tau(\mu^j)$ and  $\mu_i^{j} = \mu_{j+1}^{j}=0$
exactly if $i > j$,
$\pi_j( \tau(\mu^n))=\pi_j(\tau'(\mu^n))$ exactly
if 
$\sigma(i)=\sigma'(i)$ and $\epsilon_i=\epsilon'_i$ for all $0\leq i \leq j$ .
Hence,   
\begin{gather*}
h(\tau, \tau')=
\min\{j \in \{1,\dots,n\} \mid \sigma(j) \neq 
\sigma'(j)\mbox{ or }\epsilon_j\neq \epsilon'_j\};
\quad \mbox{in particular,} \\
h(\tau, \tau s_{x_h \pm x_k})= h 
\quad \mbox{for all } 
1 \leq h<k \leq n, 
\quad \mbox{and} 
\\
h(\tau,\tau s_{2 x_h})=h\quad\mbox{for all }
1\leq h \leq n.
\end{gather*}
Let $\pb_j = \pi^*(\psi_j) \colon\Oo_{\mu^n} \to \ft^*$ for all $j$.
Since $\psi_j \colon \Oo_{\mu^j}^T \to \ft^*$ is the inclusion map,
$\pb_j(\tau(\mu^n))
= \sum_{i=1}^{n}\mu_i^j(-1)^{\epsilon_i}x_{\sigma(i)}$ for all $j$.
Hence,
\begin{equation*}
\labell{andiff2}
\pb_j(\tau'(\mu^n)) - \pb_j(\tau(\mu^n))
= \sum_{m=1}^{j} \mu^j_m \left((-1)^{\epsilon'_m}x_{\sigma'(m)}-(-1)^{\epsilon_m}x_{\sigma(m)}\right) 
 \quad
\mbox{for all } j, 
\end{equation*}
and so
\begin{equation*}
\pb_j(\tau'(\mu^n))-\pb_j(\tau(\mu^n))=
 \left((-1)^{\epsilon_j}x_{\sigma(j)}-(-1)^{\epsilon'_j}x_{\sigma'(j)}\right) 
\quad \mbox{for all } j \leq h(\tau,\tau');
\end{equation*}
therefore
\begin{equation*}
\pb_{h}(\tau s_{x_h \pm x_k} (\mu^n)))-
\pb_{h}(\tau(\mu^n)) 
=
(-1)^{\epsilon_h} x_{\sigma(h)}\mp (-1)^{\epsilon_k}  x_{\sigma(k)}
= \sigma(x_h \pm x_k)
\end{equation*}
for all  $1 \leq h < k \leq n$, and
\begin{equation*}
 \pb_{h}(\tau s_{2x_h} (\mu^n))-
\pb_{h}(\tau(\mu^n)) =(-1)^{\epsilon_h}(2x_{\sigma(h)})
= \sigma( 2 x_h)
\quad \mbox{for all } 1 \leq h \leq n.
\end{equation*}
To conclude, let $\varphi = \psi_n^\xi \colon \Oo_{\mu^n} \to \R$ 
be a generic component of the moment
map 
that
achieves its minimum value at $\mu^n$.
By the definition of $\mu^n$, the set $R^+$ coincides with $\{\alpha\in \R\mid (\alpha,\xi)>0\}$.
So by Proposition \ref{canonical=Schubert}, canonical classes on $\Oo_{\mu^n}$
correspond to equivariant Schubert classes on $G_{\C}/B$ through the usual
identification of $\Oo_{\mu^n}^T$ with $W$.
The claim now follows directly
from Theorem \ref{tower symp}
and Proposition~\ref{cocan}.
(Here, we use the fact that 
$h(\tau_i,\tau_{i+1}) \leq h(\tau_i,\tau_{|\underline{\tau}|+1})$
for any $\underline{\tau}=(\tau_1,\dots,\tau_{|\underline{\tau}|+1})
\in C(w,w')$.)

\end{proof}


\subsection{Generic coadjoint orbits of type $B_n$}\labell{section Bn}

The main result of this section is an inductive 
\textit{positive integral formula} that expresses the restrictions
of the equivariant Schubert classes on a generic coadjoint orbit of type $B_n$
in terms  of
products of distinct positive roots
with positive integer coefficients, and 
the restriction of equivariant Schubert classes on
a generic coadjoint orbit of  type $B_{n-1}$.
To find this formula, we will apply
Corollary~\ref{corollary formula} to the natural projection
from a generic coadjoint orbit of $SO(2n+1)$ to 
$Gr_2^+(\R^{2n+1})$,
the Grassmannian of oriented
$2$-planes in $\R^{2n+1}$.

Let $G=SO(2n+1)$, $T\subset G$ be a maximal torus, and
$W$ be the associated Weyl group;
assume $n > 1$.
We can identify the dual of the Lie algebra of $T$ with $(\R^n)^*$ so
that the set of roots is
$$R = \{ \pm x_i\pm x_j \mid 1 \leq i < j \leq n\} \cup
\{ \pm x_i \mid 1 \leq i \leq n \} \subset \td.$$

Let $\Gh=SO(2n-1)$.
We can identify the dual of the Lie algebra of a maximal torus $\widehat{T}$ of $\Gh$
with the set of $(a_1,\dots,a_n) \in \ft^*$ such that $a_1 = 0$.
This identifies the roots of $\Gh$ with the set
$$\Rh = \{ \pm x_i \pm x_j \mid 2 \leq i  < j \leq n\} \cup
 \{ \pm x_i \mid 2 \leq i \leq n\} \subset R,$$
and
the Weyl group $\Wh$ of $\Gh$ with
the subgroup of $W$ generated by reflections across the roots in 
$\Rh$; moreover, it
induces a map from $H^*(B\widehat{T};\Z)$ to $H^*(BT;\Z)$.
Equivalently, 
let $\Vt = \{ \pm x_1, \pm x_2, \dots, \pm x_n\}$;
$\Wh$ is the kernel of the map
$\pi \colon W \to \Vt$ 
defined by 
$\pi(w) = w(-x_1)$.

To state our main theorem, we will 
need several additional definitions.
Let 
$$R^+ = \{x_i \pm x_j \mid 1 \leq i < j \leq n \} \cup 
\{x_i \mid 1 \leq i \leq n\} \subset R$$ be the set of positive roots.
Define
\begin{equation}\label{BE}
E = \{( \tau, \tau s_\beta) \in W \times W \mid l(\tau s_\beta) = l (\tau) + 1 
\mbox{ and } \beta \in R\}.
\end{equation}
Given $w$ and $w' \in W$, 
let $\overline{\Sigma}(w,w')$ denote
the set of paths $\gamma = (\gamma_1,\dots, \gamma_{|\gamma|+1})$
from $w$ to $w'$ in $(W,E)$ such that $\pi(\gamma_i) \neq \pi(\gamma_{i+1})$
for all $i$.  Equivalently,  $\overline{\Sigma}(w,w')$
is the set of paths from $w$ to $w'$ such that each edge is of the form
$( \tau, \tau s_\beta)$, where $ l(\tau s_\beta) = l (\tau) + 1 $
and  $\beta \in R \smallsetminus \Rh$.
Given a sequence $\gt \in \Vt^k,$ 
let $V(\gt)$ be the set of ``vertices" of $\gt$:
$$V(\gt) =   \{ \gt_1,\dots,\gt_{k}\} \subset  \Vt.$$  
We need the following lemma, 
which we prove on page~\pageref{proveBov}. 

\begin{lemma}\labell{B order vertices}
Given any  path $\gamma = (\gamma_1,\gamma_2,\dots, \gamma_{|\gamma|+1}) \in 
\overline{\Sigma}(w,w')$, the sequence  $\gt =
\pi(\gamma) =
(\pi(\gamma_1),\pi(\gamma_2),\dots,\pi(\gamma_{|\gamma|+1}))$
is a subsequence of $(-x_1, -x_2,\ldots, -x_n,x_n,\dots,x_2,x_1)$.
\end{lemma}

\begin{definition}\label{B complete}
A path $\gamma \in \overline{\Sigma}(w,w')$  
with $\pi(\gamma) = \gt$
is 
\textbf{incomplete} if both the following 
conditions are satisfied:
\begin{itemize}
\item[(i)] $\{ \pi(w'), -\pi(w') \}\subset V(\gt)$, and
\item[(ii)] $\gamma$  does not contain any edge $e$ of the form
$(\tau, \tau s_{x_1}),$ 
that is, an edge such that $\pi(e) =
(-x_j,x_j)$ for some $j \in \{1,\dots,n\}$. 
\end{itemize}
Otherwise $\gamma$ is  \textbf{complete}. 
\end{definition}

\begin{definition}
A path $\gamma\in \overline{\Sigma}(w,w')$ 
with $\pi(\gamma) = \gt$
is 
\textbf{relevant} if either it is complete or 
if it  is incomplete and 
$x_{k(\gamma)+1}\in V(\gt)$, where\footnote{ 
Observe 
that if $\gamma$ is incomplete then
condition (i) in the definition above implies 
that $\{j\mid \{-x_j,x_j\}\subset V(\gt)\} \neq \emptyset$ 
and 
-- by Lemma~\ref{B order vertices} --
condition (ii) implies that $k(\gamma)<n$.}
$k(\gamma) = 
\max\{j\mid \{-x_j,x_j\}\subset V(\gt)\}$. 
\end{definition}

Finally, given 
a path $\gamma\in \overline{\Sigma}(w,w')$, define 
\begin{gather*}
P(\gamma)=
\Lt^-_{\pi(w')}
\left(
 \prod_{i=1}^{|\gamma|}\frac
{1}
{\pi(w')-\pi(\gamma_i)}
\frac
{\pi(\gamma_{i+1})-\pi(\gamma_i)}
{\eta(\gamma_i,\gamma_{i+1})}
\right), 
\end{gather*}
where $\Lt^-_{\pi(w')} $ is the product of the $\alpha \in R^+ $ such that
$ \langle \pi(w'), \alpha \rangle > 0$ and \\ $\pi(s_{\alpha}(w'))\neq \pi(w').$

The main theorem of this section can be stated as follows.

\begin{theorem}\labell{mainBnprecise} 
Let $B \subset G_\C$ and $\Bh \subset \Gh_\C$ be the Borel subgroups 
associated to  $R^+$ and $R^+ \cap \Rh$, respectively,
where $G_\C$ and $\Gh_\C$ are  the complexifications of $G = SO(2n + 1)$
and $\Gh = SO(2n - 1)$,  
and other symbols are defined as above.
Given
$w$
and $w'$
in $W$,
let $K_w \in H^{2 l(w)}_T(G_\C/B;\Z)$ be
the equivariant Schubert class associated to $w$, 
and let $\tau \in W$ be the shortest element such that $\pi(\tau) = \pi(w')$.
For all $\wh \in \Wh$,  let 
$\Kh_{\wh} \in H_{\widehat{T}}^{2 l(\wh)}(\Gh_\C/\Bh;\Z)$ be
the equivariant Schubert class associated to $\wh$, 
and let 
$R(w,\tau \wh) \subset \overline{\Sigma}(w, \tau \wh)$  denote
the set of relevant paths from $w$ to $\tau \wh$ in $(W,E)$. 
\begin{itemize}
\item[(1)]
For all 
$w$ and $w'$ in $W$
\begin{gather*}
K_w(w')=\sum _{\wh \in \Wh}
\bigg(\sum _{\gamma \in R(w,\tau \wh) }Q(\gamma)\bigg) 
\; \tau \left(\widehat{K}_{\wh}(\tau^{-1} w') \right),
\end{gather*}
where for every $\gamma \in R(w,\tau \wh)$
\begin{gather*}
Q(\gamma) = 
\begin{cases}
\begin{array}{ll}
P(\gamma)
 & \mbox{  if $\gamma$ is complete}\\
 \displaystyle P(\gamma) \frac{2\pi(w')}{\pi(w')+x_{k(\gamma)+1}}
 & \mbox{  if $\gamma$ is incomplete }.
\end{array}
\end{cases}
\end{gather*}
\item[(2)] 
$Q(\gamma)$ is the product of distinct positive roots
and a constant which is either $1$ or $2$  for all $\gamma \in 
R(w,\tau \wh).$
\end{itemize}
\end{theorem}

\begin{example}
Consider the case that $G = SO(5)$.
Let $w=s_2$ and $w'=s_1s_2s_1s_2$,
where  $s_1=s_{x_1-x_2}$ and $s_2=s_{x_2}$. 
We want to compute
$K_w(w')$ using Theorem \ref{mainBnprecise}.
Since $\pi^{-1}(\pi(w'))=\{\tau,w'\}\subset W$,
where  $\tau=s_1s_2s_1$,  $\tau$ is the shortest element in $\pi^{-1}(\pi(w'))$.
Since $\widehat{W} = \{\Id,s_2\}$ and $\tau s_2 = w'$,
we need to find the sets of relevant paths $R(w,\tau)$ and $R(w,w')$.
It is straightforward to check that the following hold:
\begin{itemize}
\item $\overline{\Sigma}(w,\tau)=\{\gamma^1,\gamma^2\}$, where $\gt^1=\pi(\gamma^1)=(-x_1,x_2,x_1)$ and $\gt^2=(-x_1,-x_2,x_1)$; 
so the paths $\gamma^1$ and $\gamma^2$ are incomplete, and $\gamma^1$ is relevant.
Hence $R(w,\tau)=\{\gamma^1\}$; moreover,  $Q(\gamma^1)=x_1$.
\item $\overline{\Sigma}(w,w')=\{\gamma^3\}$, 
where $\gt^3=\pi(\gamma^3)=(-x_1,-x_2,x_2,x_1)$; so the path $\gamma^3$ is complete
and hence relevant. So $R(w,w') = \{\gamma^3\}$ and $Q(\gamma^3)=1$.
\end{itemize}
Moreover $\widehat{K}_{\Id}\equiv 1$, $\widehat{K}_{s_2}(s_2)=x_2$ and $\tau(x_2)=x_2$.
Therefore, Theorem \ref{mainBnprecise} implies that $$
K_w(w')=Q(\gamma^1)\tau(\widehat{K}_{\Id}(s_2))+
Q(\gamma^3)\tau(\widehat{K}_{s_2}(s_2))=
x_1+x_2\;.
$$
\end{example}

To prove Theorem \ref{mainBnprecise},
we 
need to translate it into 
geometrical language.
Fix a point
$$\mu^j \in \td
\ \ 
 \mbox{such that } 
\mu^j_1<\cdots<\mu^j_j<0= \mu^j_{j+1} =  \cdots =\mu^j_n
\ \ 
 \mbox{for each } 0 \leq j \leq n;$$
for simplicity assume that $\mu_j^j=-1$. Let
$(\Oo_{\mu^j}, \omega_j,\psi_j)$ be the coadjoint orbit
through $\mu^j$ for each $j$.
The stabilizer of $\mu^j$ is
$$SO(2) \times \dots \times SO(2) \times SO(2n-2j + 1)
\quad \mbox{for all } j;$$
in particular, 
$P_{\mu^i} \subseteq P_{\mu^j}$ for all $1 \leq j \leq i \leq n$.
Moreover,
let $\varphi=\psi_n^{\xi} \colon \Oo_{\mu^n} \r \R$  be a generic
component of the moment map that achieves its minimum value at $\mu^n$.
Observe that, by the definition of $\mu^n$, the set $R^+$
coincides with $\{\alpha \in R \mid ( \alpha, \xi ) > 0\}$.
By Proposition~\ref{cocan}, there exists a canonical
class $\alpha_p \in H_T^{2 \lambda(p)}( \Oo_{\mu^n};\Z)$ 
for each $p\in \Oo_{\mu^n}^T$.
The
map from $W$  to
$\Oo_{\mu^n}^T$ given by $w \mapsto w(\mu^n)$ identifies the
canonical graph of $\Oo_{\mu^n}$ with $(W,E)$; 
we
shall identify these without further comment.
Moreover, by 
Proposition~\ref{theta one1}
and Lemma~\ref{llambda}, 
 $\Theta(r,r')=1$ for each edge $(r,r')$ in $E$.

Let $\pi \colon \Oo_{\mu^n} \to \Oo_{\mu^1}$ be the natural projection.
Note that  $\overline{\Sigma}(w,w')$
is exactly the set of  horizontal paths (with respect to $\pi$)
from $w$ to $w'$; 
see Definition~\ref{def:horizontal}.
Moreover, 
given any $\gamma \in \overline{\Sigma}(w,w')$,
the projection $\gt = \pi(\gamma)$ is an
ascending path in the GKM graph $(\Vt,\gkmet)$ associated to 
$(\Oo_{\mu^1}, \omega_1, \psi_1)$
by Lemmas~\ref{GKMlambda} and \ref{increasing path}; 
this proves Lemma~\ref{B order vertices}. \labell{proveBov}
Finally, note that  $(\Vt,\gkmet)$ is a complete graph,
where $\Vt = \{ \pm x_1, \pm x_2, \dots, \pm x_n \}$,
and that  $\Lt_{\pi(w')}^-$ is the equivariant
Euler class of the negative normal bundle of 
$\vt = \psi_1^\xi \colon \Oo_{\mu^1} \to \R$ at $\pi(w')$.

\begin{remark}\rm

By Proposition~\ref{coadjoint preserving},
the natural projection map $\rho_j \colon \Oo_{\mu^{j+1}} \to \Oo_{\mu^j}$
is a strong symplectic fibration 
with fiber 
the Grassmannian
$P_{\mu^j}/P_{\mu^{j+1}} \simeq Gr_2^+(\R^{2n-2j+1})$ for all 
$0 \leq j < n$.
Hence, since 
$H^*\big(Gr_2^+(\R^{2n-2j+1});\Z\big[\frac{1}{2}\big]\big) \simeq
H^*\big(\CP^{2n-2j-1};\Z\big[\frac{1}{2}\big]\big)$ and $\Theta(r,r')=1$ for each edge $(r,r')$ in $E$,
Theorem~\ref{tower symp} (together with Proposition \ref{canonical=Schubert})
immediately implies 
that, for any 
$w,w'\in W$
we can express the restriction
$K_w(w')$
as a sum of terms $\Xi(\gamma)$ over paths $\gamma
\in C(w,w')$
where each term 
is a polynomial in the simple roots with positive {\em rational}
coefficients; more precisely,
$\Xi(\gamma)$
is the product of distinct positive roots and a constant that is
a (possibly negative) power of 
$2$; cf. \cite{Za}.
\end{remark}

To prove the theorem, we need to analyze how the 
expression of $P(\gamma)$ changes depending on
whether $\gamma$ is complete or incomplete.

\begin{proposition}\labell{preciseP}
Let $\gamma=(\gamma_1,\ldots,\gamma_{|\gamma|+1})$ be a path in $\overline{\Sigma}(p,s)$.
Let $\gt = \pi(\gamma)$ and let $SV(\gt)$ be the skipped vertices of $\gt$; see Definition~\ref{def:sv}.
Then
$$
P(\gamma)= c
\prod_{r\in SV(\gt)}\eta(r,\pi(s)), \quad \mbox{where } c  = 
\begin{cases}
1 \mbox{ or } 2
& \mbox{if } \gamma \mbox{ is complete, and } \\
\frac{1}{2}
& \mbox{if } \gamma \mbox{ is incomplete.} \\
\end{cases}
$$
More precisely, $c = 2$ exactly if neither condition in Definition~\ref{B complete} is satisfied.
\end{proposition}

\begin{proof}
The edge $(-x_j,x_j) \in \gkmet$ has magnitude $2$ for all $j$;
all the other edges $(r,r') \in \gkmet$   have magnitude $1$; 
see Definition \ref{magnitude}.
Moreover, by Lemma~\ref{B order vertices}, $\gt$
can have at most one edge of type 
$(-x_j,x_j)$.
Therefore, 
since $\Theta(r,r') = 1 $ for all $(r,r') \in E$, 
the claim follows from Lemma~\ref{explicit P}.
\end{proof}

We also need  the following two lemmas.

\begin{lemma}\labell{uneven1}
Let $\gamma$ be a path in $ \overline{\Sigma}(p,s)$. 
If 
$\{-x_l,x_l\}\subset V(\pi(\gamma))$ 
for some  $l<n$, 
then $\{-x_{l+1},x_{l+1}\}\cap V(\pi(\gamma)) 
\neq \emptyset$.
\end{lemma}

\begin{lemma}\labell{uneven2}
Let $\gamma$ be a path in $\overline{\Sigma}(p,s)$
such that 
$\{-x_l,x_l \} \subset V(\pi(\gamma))$ 
for some $l$. 
If $\{-x_{l+1},x_{l+1}\} \cap V(\pi(\gamma))
= \{x_{l+1}\}$ for some $l$, 
then there exists a unique path 
$\gamma' \in \overline{\Sigma}(p,s)$
such that  $V(\pi(\gamma'))$ 
is obtained from $V(\pi(\gamma))$ 
by replacing the vertex $x_{l+1}$ with $-x_{l+1}$. 
That is, $-x_{l+1} \in  V(\pi(\gamma'))$
and $V(\pi(\gamma'))
\smallsetminus \{-x_{l+1}\} = V(\pi(\gamma)) 
\smallsetminus \{x_{l+1}\}$.
A similar claim holds if 
$\{-x_{l+1},x_{l+1}\} \cap V(\pi(\gamma))
= \{-x_{l+1}\}.$
\end{lemma}

To simplify the proof of these lemmas, 
let $s_l = s_{x_l - x_{l+1}}$ denote the reflection
across the root $x_l - x_{l+1}$ 
for all  $ l\in\{1,\ldots,n-1\}$, 

We recall the following 
relations; for
all $ l\in\{1,\ldots,n-1\}$ 
and $ j\in \{1,2,\ldots,n\}$ with $j\notin\{l,l+1\}$
\begin{equation}\labell{relation reflections}
\begin{array}{c}
s_{x_l}=s_ls_{x_{l+1}}s_l  \\
s_{x_l\pm x_j}=s_ls_{x_{l+1}\pm x_j}s_l  \\
s_{x_l+x_{l+1}}=s_ls_{x_l+x_{l+1}}s_l \\
\end{array}
\end{equation}

\begin{proof}[Proof of Lemma \ref{uneven1}]
Let $\gt = \pi(\gamma)$.
Suppose that, on the contrary,  $\{-x_l,x_l\} \subset V(\gt)$ 
but
$\{-x_{l+1},x_{l+1}\}\cap V(\gt)=\emptyset$. 
Let $\gt'$ be the ascending path in $(\Vt,\gkmet)$
such that $V(\gt') = \{ -x_{l+1}, x_{l+1}\} \cup V(\gt)$.
There exists $\beta_i\in R$ 
such that 
$\gt_{i+1}=s_{\beta_i}(\gt_i)$ 
for all $i=1,\ldots,|\gt|$, and 
there exists $\delta_i \in R$ such that
$\gt'_{i+1}=s_{\delta_i}(\gt'_i)$ 
for all $i=1,\ldots,|\gt'|$.
Define $w$ and $w'$  in the Weyl group $W$ of $G$ by
$$w = s_{\beta_{|\gt|}}s_{\beta_{|\gt|-1}}\cdots s_{\beta_1} 
\quad \mbox{and} \quad
w' = s_{\delta_{|\gt'|}}s_{\delta_{|\gt'|-1}}\cdots s_{\delta_1}.$$
\begin{itemize}
\item 
If $\gt=(\ldots, -x_l,x_l,\ldots)$,
then  $w=w_1s_{x_l}w_2$ and $w'=w_1s_ls_{x_{l+1}}s_lw_2$ for some $w_1$
and $w_2\in W$. Hence
$w=w'$
by \eqref{relation reflections}.
\item If $(-x_l,x_l)$ is not an edge of $\gt$,
then
there exists $i$ and $h > l$ so that
$w=w_1s_{x_l\pm x_h}w_0s_{x_l\pm x_i}w_2$ and 
$w'=w_1s_ls_{x_{l+1}\pm x_h}w_0s_{x_{l+1}\pm x_i}s_lw_2$  
for some $w_0,w_1,$ and $w_2 \in W$ such that $w_0$ commutes with $s_l$.
Hence by \eqref{relation reflections}
we 
again
have $w'=w_1s_ls_{x_{l+1}\pm x_h}s_lw_0s_ls_{x_{l+1}\pm x_i}s_lw_2=
w_1s_{x_l\pm x_h}w_0s_{x_l\pm x_i}w_2=w$. 
\end{itemize}

Moreover, by Lemma~\ref{liftpath} there exists
a path $\gamma'$ of length $|\gt'|$ in $(V,\gkme)$ that starts at $p$ 
such that $V(\pi(\gamma')) = V(\gt')$ 
and 
$\lambda(\gamma'_{i+1})> \lambda(\gamma'_i)$ for all 
$1 \leq i \leq |\gt'|$.
Moreover, since $w = w'$,  the endpoints $s = w(p)$ of $\gamma$ and
$s' = w'(p)$ of $\gamma'$ are equal.
On the other hand,
the fact that
$\gamma\in\overline{\Sigma}(p,s)
\subset \Sigma(p,s)
$ implies that $\lambda(s)-\lambda(p)=
|\gt|$. 
Moreover, 
$\lambda(\gamma'_{i+1}) > \lambda(\gamma'_i)$ for all 
$1 \leq i \leq |\gt'|$.
Hence, $\lambda(s')-\lambda(p)\geq |\gt'| = |\gt| + 2$.  
Since $s = s'$, this is impossible.
\end{proof}

\begin{proof}[Proof of Lemma \ref{uneven2}]
Let $\gt = \pi(\gamma)$.
Assume that $\{-x_{l+1},x_{l+1} \} \cap V(\gt) = \{x_{l+1}\}$.
Let $\gt'$ be the ascending path in $(\Vt,\gkmet)$
such that $V(\gt')$ is obtained from $V(\gt)$ by replacing the
vertex $x_{l+1}$ with $-x_{l+1}$.
As before,
there exists $\beta_i\in R$ 
such that 
$\gamma_{i+1}=s_{\beta_i}(\gamma_i)$ 
for all $i=1,\ldots,|\gamma|$, and 
there exists $\delta_i \in R$ such that
$\gamma'_{i+1}=s_{\delta_i}(\gamma'_i)$ 
for all $i=1,\ldots,|\gamma'|$.
By Lemma 
 \ref{liftpath},
this implies that
$\gt_{i+1}=s_{\beta_i}(\gt_i)$ 
and $\gt'_{i+1}=s_{\delta_i}(\gt'_i)$ 
for all $i$.
Define $w$ and $w' \in W$ by
$$w = s_{\beta_{|\gt|}}s_{\beta_{|\gt|-1}}\cdots s_{\beta_1} 
\quad \mbox{and} \quad
w' = s_{\delta_{|\gt'|}}s_{\delta_{|\gt'|-1}}\cdots s_{\delta_1}.$$
\begin{itemize}
\item If $\gt=(\ldots,-x_l,x_{l+1},x_l,\ldots)$ 
then $w=w_1s_ls_{x_l+x_{l+1}}w_2$ and $w'=w_1s_{x_l+x_{l+1}}s_lw_2$ for some 
$w_1,w_2\in W$. 
Hence
$w=w'$
by \eqref{relation reflections}. 
\item 
If $(-x_1,x_{l+1})$ is not an edge of $\gt$,
then there exists $h$ and $i>l+1$  so that
$w=w_1s_ls_{x_{l+1}\pm x_h}w_0s_{x_l\pm x_i}w_2$ and 
$w'=w_1s_{x_l\pm x_h}w_0s_{x_{l+1}\pm x_i}s_lw_2$ 
for some $w_0,w_1,$ and $w_2 \in W$ such that $w_0$ commutes with $s_l$.
Hence 
again
by \eqref{relation reflections} we have
$w=w_1s_ls_{x_{l+1}\pm x_h}s_lw_0s_ls_{x_l\pm x_i}w_2=
w_1s_{x_l\pm x_h}w_0s_{x_{l+1}\pm x_i}s_lw_2=w'$.
\end{itemize}

One the other  hand,  
the fact that
$\gamma\in\overline{\Sigma}(p,s)$ implies that $\lambda(s)-\lambda(p)=
|\gamma|$. 
Moreover, since $\gt'$ is an  ascending path,  
Lemma \ref{increasing path} 
and Lemma~\ref{flag index increasing} together imply that
$\lambda(\gamma'_{i+1})-\lambda(\gamma'_i)\geq 1$ for all 
$1 \leq i \leq |\gamma'| = |\gamma|$.
But this is impossible unless 
 $\lambda(\gamma'_{i+1})-\lambda(\gamma'_{i})=1$ for all $i$, 
which implies that $\gamma'\in \overline{\Sigma}(p,s)$.
\end{proof}
We are now ready to prove Theorem \ref{mainBnprecise}

\begin{proof}[Proof of Theorem \ref{mainBnprecise}]
Let $p = w(\mu^n)$ and $q = w'(\mu^n)$.
For all $s \in 
\widehat{\Oo}_q^T$,
 let $\ah_s \in H^*_T(\widehat{\Oo}_q;\Z)$ be the canonical class on the fiber 
$\widehat{\Oo}_q = \pi^{-1}(\pi(q))\subset \Oo_{\mu^n}$.
Since
Proposition \ref{coadjoint preserving} implies
that $\pi$ is a strong symplectic fibration
and
$\psi_1$ is the inclusion,
Corollary \ref{corollary formula}, Proposition~\ref{cocan} and Lemma \ref{coGKM} together
imply that
$$
\alpha_p(q)=\sum_{s\in \widehat{\Oo}_q^T}
\Big(\sum_{\gamma\in \overline{\Sigma}(p,s)}P(\gamma)\Big)
\ah_s(q).
$$ 
If $\gamma\in\overline{\Sigma}(p,s)$ is complete, 
then $\gamma$ is  relevant and $P(\gamma)=Q(\gamma)$. 
On the other hand, by
Lemmas~\ref{uneven1} and \ref{uneven2},
the set of incomplete paths can be decomposed into
pairs of paths $\gamma$ and $\gamma'$, 
so that $V(\gt') = V(\pi(\gamma'))$ is obtained from 
$V(\gt)=V(\pi(\gamma))$ by replacing $x_{k(\gamma)+1}$ 
by $-x_{k(\gamma)+1}$, where $k(\gamma) = \max\{j \mid \{-x_j, x_j\}
\subset V(\gt)\}$.
In particular,
$SV(\gt)\smallsetminus\{-x_{k(\gamma)+1}\}=
SV(\gt')\smallsetminus\{x_{k(\gamma)+1}\}$. 
Additionally, by the definition of $k(\gamma)$,
$\pi(s) \neq \pm x_{k(\gamma)+1}$, and so
$\eta( \pm x_{k(\gamma)+1},\pi(s))
= 
\pi(s) \mp x_{k(\gamma)+1}.$
Hence by Proposition \ref{preciseP} 
\begin{equation*}
\begin{split}
Q(\gamma) & = 
\pi(s)  \prod_{r \in SV(\gt) \cap SV(\gt')} \eta(r,\pi(s)) \\
&= 
\frac{
 \eta(-x_{k(\gamma)+1},\pi(s))+
\eta(x_{k(\gamma)+1},\pi(s) 
}{2} 
\left( \prod_{r \in SV(\gt) \cap SV(\gt')} \eta(r,\pi(s)) \right)
\\
&= P(\gamma)+P(\gamma').
\end{split}
\end{equation*}
Since
$\gamma$ is relevant, but $\gamma'$ is not,
this  implies that
$$\displaystyle\sum_{\gamma\in \overline{\Sigma}(p,s)}P(\gamma)=\sum_{\gamma \in R(p,s)}Q(\gamma).$$
By Propositions~\ref{canonical=Schubert} and \ref{twist} , this proves part $(1)$ of Theorem \ref{mainBnprecise}.
Finally, 
by the definition of $SV(\gt)$, $\eta(r,\pi(s))$ is a positive root for all $r \in SV(\gt)$. Hence
if $\gamma$ is complete then $Q(\gamma) = P(\gamma)$ is the product of distinct positive roots
and a constant which is either $1$ or $2$. 
On the other hand, if $\gamma$ is incomplete then,
by definition of incomplete path, 
$\pi(s)$ must be a positive root.   
and so again  $Q(\gamma)$ is  the product of distinct positive roots.
\end{proof}


\subsection{Generic coadjoint orbit of type $D_n$}

The main result of this section is an inductive 
\textit{positive integral formula} that expresses the restrictions
of the equivariant Schubert classes on a generic coadjoint orbit of type $D_n$
in terms  of
products of distinct positive roots
with positive integer coefficients, and 
the restriction of equivariant Schubert classes on
a generic coadjoint orbit of  type $D_{n-1}$.
To find this formula, we will apply
Corollary~\ref{corollary formula} to the natural projection
from a generic coadjoint orbit of $SO(2n)$ to 
$Gr_2^+(\R^{2n})$, the Grassmannian of oriented
$2$-planes in $\R^{2n}$ 
for all $n \geq 4$. 
(If $n = 3$ then a generic coadjoint orbit of type $D_n$ 
is also the complete flag on $\C^4$,
and so we can use the techniques of \S\ref{An}.)
Despite the fact that $H^{2n-2}(Gr_2^+(\R^{2n});\R) = \R^2 \neq
\R = H^{2n - 2}(\CP^{2n -2};\R)$, we will then
proceed as in section \ref{section Bn}.

Let $G=SO(2n)$, $T \subset G$ be a maximal torus, and $W$ be the associated
Weyl group; assume $n>1$.
We can identify the dual of the Lie algebra of $T$ with $\td
= (\R^n)^{*}$ so that the set of roots is
 $$R=\{\pm x_i \pm x_j  \mid 1 \leq i < j \leq n\}\subset \td.$$
 Let $\widehat{G}=SO(2n-2)$. We can identify the dual of the Lie algebra of a maximal
 torus $\widehat{T}$ of $\widehat{G}$ with the set of $(a_1,\ldots,a_n)\in \td$ such that $a_1=0$.
 This identifies the roots of $\widehat{G}$ with the set
 $$
 \widehat{R}=\{\pm x_i \pm x_j\mid 2\leq i <j\leq n\}\subset R,
 $$
and the Weyl group $\widehat{W}$ of $\widehat{G}$ with the subgroup of $W$
 generated by reflections across the roots in $\widehat{R}$;
 moreover, it
 induces a map from $H^*(B\widehat{T};\Z)$ to $H^*(BT;\Z)$.
 Equivalently, 
 let $\Vt=\{\pm x_1,\pm x_2,\ldots,\pm x_n\}$;
 $\widehat{W}$ is the kernel of the map $\pi\colon W\to \Vt$
 defined by $\pi(w)=w(-x_1)$.
 
To state our main theorem, we will need several additional definitions.
Let 
$$R^+ = \{x_i \pm x_j \mid 1 \leq i < j \leq n \} \subset R$$ be the set of positive roots.
Define
\begin{equation}\label{DE}
E = \{( \tau, \tau s_\beta) \in W \times W \mid l(\tau s_\beta) = l (\tau) + 1 
\mbox{ and } \beta \in R\}.
\end{equation}
Given $w$ and $w' \in W$, 
let $\overline{\Sigma}(w,w')$ denote
the set of paths $\gamma = (\gamma_1,\dots, \gamma_{|\gamma|+1})$
from $w$ to $w'$ in $(W,E)$ such that $\pi(\gamma_i) \neq \pi(\gamma_{i+1})$
for all $i$.  Equivalently,  $\overline{\Sigma}(w,w')$
is the set of paths from $w$ to $w'$ such that each edge is of the form
$( \tau, \tau s_\beta)$, where $ l(\tau s_\beta) = l (\tau) + 1 $
and  $\beta \in R \smallsetminus \Rh$.
Given a sequence 
$\gt\in \Vt^k$, let
$V(\gt)$ be the set of ``vertices" of $\gt$:
$$
V(\gt)=\{\gt_1,\ldots,\gt_k\}\subset \Vt.
$$
We need the following lemma,
which we prove on page~\pageref{proveDov}.
\begin{lemma}\labell{D order vertices}
Given any path $\gamma=(\gamma_1,\gamma_2,\ldots,\gamma_{|\gamma|+1})\in \overline{\Sigma}(w,w')$,
the sequence $\gt=\pi(\gamma)=(\pi(\gamma_1),\pi(\gamma_2),\ldots,\pi(\gamma_{|\gamma|+1}))$
is a subsequence of $(-x_1,-x_2,\ldots,-x_n,x_n,\ldots,x_2,x_1)$.
\end{lemma} 

\begin{definition}
A path $\gamma \in \overline{\Sigma}(w,w')$ with $\pi(\gamma)=\gt$
is 
\textbf{incomplete} if  $\{-\pi(w'),\pi(w')\}\subset V(\gt)$.
Otherwise $\gamma$ is  \textbf{complete}. 
\end{definition}

\begin{definition}
A path $\gamma\in \overline{\Sigma}(w,w')$ 
with $\pi(\gamma)=\gt$
is 
\textbf{relevant} if either it is complete or 
if it  is incomplete and 
$x_{k(\gamma)+1}\in V(\gt)$, where\footnote{
Observe that if $\gamma$ is incomplete then by definition
$\{j\mid \{-x_j,x_j\}\subset V(\gt)\} \neq \emptyset$. Moreover, 
since $\pm x_i$ is not a root $\gamma$ cannot contain any edge whose 
projection is $(-x_n,x_n)$, and  so -- by Lemma~\ref{D order vertices} --
 $k(\gamma)<n$.
}
 $k(\gamma) = 
\max\{j\mid \{-x_j,x_j\}\subset V(\gt)\}$. 
\end{definition}
Finally, given 
a path $\gamma\in \overline{\Sigma}(w,w')$, define 
\begin{gather*}
P(\gamma)=
\Lt^-_{\pi(w')}
\left(
 \prod_{i=1}^{|\gamma|}\frac
{1}
{\pi(w')-\pi(\gamma_i)}
\frac
{\pi(\gamma_{i+1})-\pi(\gamma_i)}
{\eta(\gamma_i,\gamma_{i+1})}
\right), \\ 
\end{gather*}
where 
$\Lt^-_{\pi(w')} $ is the product of the $\alpha \in R^+ $ such that
$ \langle \pi(w'), \alpha \rangle > 0$ and \\ $\pi(s_{\alpha}(w'))\neq \pi(w').$

The main theorem of this section can be stated as follows.

\begin{theorem}\labell{mainDnprecise} 
Fix $n > 3$.
Let $B \subset G_\C$ and $\Bh \subset \Gh_\C$ be the Borel subgroups 
associated to  $R^+$ and $R^+ \cap \Rh$, respectively,
where $G_\C$ and $\Gh_\C$ are  the complexifications of $G = SO(2n)$
and $\Gh = SO(2n - 2)$, 
and other symbols are defined as above.
Given $w$ and $w'$ in $W$,
let $K_w \in H^{2 l(w)}_T(G_\C/B;\Z)$ be
the equivariant Schubert class associated to $w$, 
and let $\tau \in W$ be the shortest element such that $\pi(\tau) = \pi(w')$.
For all $\wh \in \Wh$,  let 
$\Kh_{\wh} \in H_{\widehat{T}}^{2 l(\wh)}(\Gh_\C/\Bh;\Z)$ be
the equivariant Schubert class associated to $\wh$, 
and let 
$R(w,\tau \wh) \subset \overline{\Sigma}(w, \tau \wh)$  denote
the set of relevant paths from $w$ to $\tau \wh$ in $(W,E)$. 

\begin{itemize}
\item[(1)]
For all $w$ and $w'$ in $W$
\begin{gather*}
K_w(w')=\sum _{\wh \in \Wh}
\bigg(\sum _{\gamma \in R(w,\tau \wh) }Q(\gamma)\bigg) 
\; \tau \left(\widehat{K}_{\wh}(\tau^{-1} w') \right),
\end{gather*}
where for every $\gamma \in R(w,\tau \wh) $
\begin{gather*}
Q(\gamma) = 
\begin{cases}
\begin{array}{ll}
P(\gamma)
 & \mbox{  if $\gamma$ is complete}\\
 \displaystyle P(\gamma) \frac{2\pi(w')}{\pi(w')+x_{k(\gamma)+1}}
 & \mbox{  if $\gamma$ is incomplete }.
\end{array}
\end{cases}
\end{gather*}
\item[(2)] 
$Q(\gamma)$ is the product of distinct positive roots
for all $\gamma\in R(w,\tau\widehat{w})$.
\end{itemize}

\end{theorem}

 To prove Theorem \ref{mainDnprecise}, we need to translate it into geometrical language.
 Fix a point
$$\mu^j \in \td
\ \ 
\mbox{such that } 
\mu^j_1<\cdots<\mu^j_j<0= \mu^j_{j+1} =  \cdots =\mu^j_n
\ \ 
 \mbox{for each } 0 \leq j \leq n;$$
for simplicity assume that $\mu_j^j=-1$. Let
$(\Oo_{\mu^j}, \omega_j,\psi_j)$ be the coadjoint orbit
through $\mu^j$ for each $j$.
The stabilizer of $\mu^j$ is
$$SO(2) \times \dots \times SO(2) \times SO(2n-2j)
\quad \mbox{for all } j;$$
in particular, $P_{\mu^{i}} \subseteq P_{\mu^j}$
for all $1\leq j\leq i \leq n$.
Moreover,
let $\varphi=\psi_n^{\xi} \colon \Oo_{\mu^n} \r \R$  be a generic
component of the moment map that achieves its minimum value at $\mu^n$.
Observe that, by the definition of $\mu^n$, the set $R^+$
coincides with $\{\alpha \in R \mid  (\alpha, \xi) > 0\}$.
By Proposition~\ref{cocan}, there exists a canonical
class $\alpha_p \in H_T^{2 \lambda(p)}( \Oo_{\mu^n};\Z)$ 
for each $p\in \Oo_{\mu^n}^T$.
The map from $W$  to
$\Oo_{\mu^n}^T$ given by $w \mapsto w(\mu^n)$ identifies the
canonical graph of $\Oo_{\mu^n}$ with $(W,E)$;
we shall identify these without further comment.
Moreover, by 
Proposition~\ref{theta one1}
and Lemma~\ref{llambda}, 
 $\Theta(r,r')=1$ for each edge $(r,r')$ in $E$.

Let $\pi \colon \Oo_{\mu^n} \to \Oo_{\mu^1}$ be the natural projection.
Note that  $\overline{\Sigma}(w,w')$
is exactly the set of  horizontal paths (with respect to $\pi$)
from $w$ to $w'$;
see Definition~\ref{def:horizontal}.
Moreover, 
given any $\gamma \in \overline{\Sigma}(w,w')$,
the projection $\gt = \pi(\gamma)$ is an
ascending path in the GKM graph $(\Vt,\gkmet)$ associated to 
$(\Oo_{\mu^1}, \omega_1, \psi_1)$
by Lemmas~\ref{GKMlambda} and \ref{increasing path}; 
this proves Lemma~\ref{D order vertices}. \labell{proveDov}
Note
that 
$(\Vt,\gkmet)$ is not a complete graph, because
it doesn't contain the edge $(-x_j,x_j)$ for any $j$, 
but it does contain all other edges.
Finally, observe
that $\Lt_{\pi(w')}^-$ is the equivariant
Euler class of the negative normal bundle of $\vt = \psi_1^\xi\colon \Oo_{\mu^1}\to \R$ at $\pi(w')$.

\begin{remark}\rm
By Proposition~\ref{coadjoint preserving},
the natural projection map $\rho_j \colon \Oo_{\mu^{j+1}} \to \Oo_{\mu^j}$
is a strong symplectic fibration 
with fiber the Grassmannian
$P_{\mu^j}/P_{\mu^{j+1}} \simeq Gr_2^+(\R^{2n-2j})$ 
for all $0 \leq j < n$. 
However, since
$H^{2n-2j-2}(Gr_2^+(\R^{2n-2j});\R) = \R^2$, 
we can not use Theorem~\ref{tower symp} to express the
restriction $K_w(w')$
as a sum of polynomial terms.
 \end{remark}

To prove the theorem, we need to analyze how the expression of $P(\gamma)$
changes depending on whether $\gamma$ is complete or incomplete.
\begin{proposition}\labell{precisePD}
Let $\gamma=(\gamma_1,\ldots,\gamma_{|\gamma|+1})$ be a path in $\overline{\Sigma}(p,s)$.
Let $\gt = \pi(\gamma)$, and let $SV(\gt)$ be the skipped vertices of $\gt$;
see Definition~\ref{def:sv}.
Then
$$P(\gamma) = 
\begin{cases}
\displaystyle \prod_{r\in SV(\gt)\setminus \{-\pi(s)\}}\eta(r,\pi(s)) & \mbox{if } \gamma \mbox{ is complete, and} \\
\displaystyle \frac{1}{2\pi(s)}\prod_{r\in SV(\gt)}\eta(r,\pi(s)) & \mbox{if } \gamma \mbox{ is incomplete.}
\end{cases}
$$
\end{proposition}

\begin{proof}
All the edges  $(r,r') \in \gkmet$  have  magnitude $1$; 
see Definition \ref{magnitude}.
Therefore, 
since $\Theta(r,r') = 1 $ for all $(r,r') \in E$, 
the claim follows from an argument similar to the proof of 
Lemma~\ref{explicit P}.
\end{proof}
We are now ready to prove Theorem \ref{mainDnprecise} 
   
\begin{proof}[Proof of Theorem \ref{mainDnprecise}]
Let $p=w(\mu^n)$ and $q=w'(\mu^n)$.
For all $s\in \widehat{\Oo}_q^T$, let $\ah_s\in H_T^*(\widehat{\Oo}_q;\Z)$
be the canonical class on the fiber $\widehat{\Oo}_q=\pi^{-1}(\pi(q))\subset \Oo_{\mu^n}$.
Since Proposition \ref{coadjoint preserving} implies
that $\pi$ is a strong symplectic fibration
and $\psi_1$ is the inclusion, 
Corollary \ref{corollary formula}, Proposition~\ref{cocan} and Lemma \ref{coGKM} together
imply that
$$
\alpha_p(q)=\sum_{s\in \widehat{\Oo}_q^T}
\Big(\sum_{\gamma\in \overline{\Sigma}(p,s)}P(\gamma)\Big)
\ah_s(q).
$$ 
If
$\gamma\in\overline{\Sigma}(p,s)$ is complete, 
then $\gamma$ is  relevant and
$P(\gamma)=Q(\gamma)$.
On the other hand, 
Lemmas~\ref{uneven1} and \ref{uneven2} still hold when $G = SO(2n)$ instead of
$SO(2n+1)$.  Indeed, the  proof is identical,  except that in the
proof of Lemma~\ref{uneven1} we no longer need to consider the case that
$(-x_l, x_l)$ is an edge of $\gt$.
Hence, as before, the set of incomplete paths can be decomposed into
pairs of paths $\gamma$ and $\gamma'$, 
so that $V(\gt') = V(\pi(\gamma'))$ is obtained from 
$V(\gt)=V(\pi(\gamma))$ by replacing $x_{k(\gamma)+1}$ 
by $-x_{k(\gamma)+1}$, where $k(\gamma) = \max\{j \mid \{-x_j, x_j\}
\subset V(\gt)\}$.
In particular,
$SV(\gt)\smallsetminus\{-x_{k(\gamma)+1}\}=
SV(\gt')\smallsetminus\{x_{k(\gamma)+1}\}$. 
Additionally, by the definition of $k(\gamma)$,
$\pi(s) \neq \pm x_{k(\gamma)+1}$, and so
$\eta( \pm x_{k(\gamma)+1},\pi(s))
= 
\pi(s) \mp x_{k(\gamma)+1}.$
Hence by Proposition \ref{precisePD} 

\begin{equation*}
\begin{split}
Q(\gamma) 
& = 
\prod_{r \in SV(\gt) \cap SV(\gt')} \eta(r,\pi(s)) 
\\
&= 
\frac{
 \eta(-x_{k(\gamma)+1},\pi(s))+
\eta(x_{k(\gamma)+1},\pi(s)
}{2\pi(s)} 
\left( 
\prod_{r \in SV(\gt) \cap SV(\gt')} \eta(r,\pi(s)) 
\right)
\\
&= 
P(\gamma)+P(\gamma')
.
\end{split}
\end{equation*}
By Proposition \ref{canonical=Schubert} and \ref{twist},
this proves part $(1)$ of Theorem \ref{mainDnprecise}.
The proof of part $(2)$ also proceeds analogously to the
proof of Theorem~\ref{mainBnprecise} (2)
in the previous subsection.
\end{proof}

\end{document}